\documentclass[nointlimits,11pt]{amsart}
\usepackage{amssymb}
\usepackage{color}

\theoremstyle{plain}
\newtheorem{theorem}{Theorem}[section]
\newtheorem{lemma}[theorem]{Lemma}

\newtheorem{proposition}[theorem]{Proposition}
\theoremstyle{definition}
\newtheorem{remark}[theorem]{Remark}
\newtheorem{remarks}[theorem]{Remarks}
\newtheorem{definition}[theorem]{Definition}

\newtoks\by
\newtoks\paper
\newtoks\book
\newtoks\jour
\newtoks\yr
\newtoks\pages
\newtoks\vol
\newtoks\publ
\def\ota{{\hbox\vol{???}}}
\def\cLear{\by=\ota\paper=\ota\book=\ota\jour=\ota\yr=\ota
\pages=\ota\vol=\ota\publ=\ota}
\def\endpaper{\the\by, \the\paper.
{\it\the\jour\/} {\bf \the\vol} (\the\yr), \the\pages.\cLear}
\def\endbook{\the\by, {\it\the\book}. \the\publ.\cLear}
\def\endprep{\the\by, \the\paper. \the\jour.\cLear}
\def\name#1#2{#1 #2}
\def\et{ and }

\numberwithin{equation}{section}

\def\esup{\operatorname{ess\,sup}}
\def\einf{\operatorname{ess\,inf}}
\def\sgn{\operatorname{sgn}}

\begin{document}

\title{Compactness of higher-order Sobolev embeddings}
\author{Lenka Slav\'{\i}kov\'a}

\address{Department of Mathematical Analysis\\
Faculty of Mathematics and Physics\\
Charles University\\
Sokolovsk\'a~83\\
186~75 Praha~8\\
Czech Republic} \email{slavikova@karlin.mff.cuni.cz}

\subjclass[2000]{46E35, 46E30} \keywords{Compactness, Sobolev space, rearrangement-invariant space, isoperimetric function, almost-compact embedding, John domain, Maz'ya domain, product probability space, integral operator.}

\thanks{This research was partly supported by the
the grant no. P201/13/14743S of the Grant
Agency of the Czech Republic and by the grant SVV-2013-267316.}

\begin{abstract}
We study higher-order compact Sobolev embeddings on a domain $\Omega \subseteq \mathbb R^n$ endowed
with a probability measure $\nu$ and satisfying certain isoperimetric inequality. Given $m\in \mathbb N$, we present a
condition on a pair of rearrangement-invariant spaces $X(\Omega,\nu)$ and $Y(\Omega,\nu)$ which suffices
to guarantee a compact embedding of the Sobolev space $V^mX(\Omega,\nu)$ into $Y(\Omega,\nu)$. The
condition is given in terms of compactness of certain one-dimensional operator depending on the isoperimetric function of $(\Omega,\nu)$.
We then apply this result to the characterization of higher-order compact Sobolev
embeddings on concrete measure spaces, including John domains, Maz'ya classes of Euclidean domains and product probability spaces, whose standard example is the Gauss space.
\end{abstract}

\maketitle

\section{Introduction}\label{S:intro}

Embeddings of Sobolev spaces into other function spaces play a very important role in modern functional analysis. Although Sobolev spaces on the Euclidean space $\mathbb R^n$ and on bounded Euclidean domains having a Lipschitz boundary are discussed most frequently, it turns out that Sobolev spaces on various other domains, possibly endowed with more general measures than just with the Lebesgue one, are of interest as well. For instance, the class of John domains (see Section~\ref{S:sobolev} for a definition), which is strictly larger than the class of domains having a Lipschitz boundary, appears in connection with the study of holomorphic dynamical systems and quasiconformal mappings. It was shown that Sobolev inequalities on John domains have the same form as in the standard case of Lipschitz domains, see~\cite{BO, HK, KM, CPS}. Furthermore, among quite a wide class of Euclidean domains, John domains are exactly those for which the Sobolev inequality holds in this form~\cite{BK}. 
Another important example is the Gauss space, that is, $\mathbb R^n$ endowed with the Gauss measure $\gamma_n$ defined by
$$
d\gamma_n(x)=(2\pi)^{-\frac{n}{2}} e^{\frac{-|x|^2}{2}}\,dx.
$$
In contrast to the Euclidean setting, Sobolev inequalities on the Gauss space are dimension-free, which yields the possibility to extend them also to infinite dimensions. This is of use in the study of quantum fields, since this study can often be reduced to Sobolev inequalities in infinitely many variables.




One possible way how to prove Sobolev embeddings is to derive them from isoperimetric inequalities for the underlying domains. This connection between Sobolev embeddings and isoperimetric inequalities was first found by Maz'ya in~\cite{M1960} and~\cite{M1961}. His discovery then led to an extensive research on this topic, which resulted in a number of important contributions that are considered classical these days (see, e.g., those by Moser~\cite{Moser}, Talenti~\cite{T}, Aubin~\cite{A} and Br\'ezis and Lieb~\cite{BL}), and which has continued until now.

Let us note that, until a very recent time, almost all available results on the interplay between Sobolev embeddings and isoperimetric inequalities involved only first-order embeddings. In our recent paper with Andrea Cianchi and Lubo\v s Pick~\cite{CPS} we have developed a method based on deriving higher-order Sobolev embeddings via subsequent iteration of first-order ones, which enables us to derive also higher-order Sobolev embeddings from isoperimetric inequalities. Furthermore, and more significantly, for customary underlying domains (e.g., for John domains and for the Gauss space, which we have already briefly mentioned) the results obtained by this method are sharp in the context of the class of rearrangement-invariant spaces.

In the present paper we show that not only continuous higher-order Sobolev embeddings but also the compact ones, can be derived from isoperimetric properties of the underlying domains, and that the results obtained in this way are sharp in many customary situations.

\medskip

Let us now describe the subject of the paper more precisely. We shall study compact Sobolev embeddings on a domain $\Omega$ in $\mathbb R^n$ endowed with a probability measure $\nu$ which is absolutely continuous with respect to the Lebesgue measure. We also require that the density of $\nu$ fulfils some technical assumptions, see Section~\ref{S:sobolev} for more details. For any $\nu$-measurable set $E\subseteq \Omega$ we denote by $P_{\nu}(E,\Omega)$ its perimeter in $\Omega$ with respect to $\nu$ (a precise definition can be found in Section~\ref{S:sobolev} again). The isoperimetric properties of $(\Omega,\nu)$ are described by the so-called isoperimetric function of $(\Omega,\nu)$, denoted $I_{\Omega,\nu}$. It is the largest function on $[0,1]$ with values in $[0,\infty]$ which is nondecreasing on $[0,\frac{1}{2}]$, nonincreasing on $[\frac{1}{2},1]$ and for which the isoperimetric inequality
$$
P_\nu(E,\Omega) \geq I_{\Omega,\nu}(\nu(E))
$$
holds for every $\nu$-measurable $E\subseteq \Omega$.

The question of finding the exact form of $I_{\Omega,\nu}$ is very difficult and has been solved only in few special cases, such as the Euclidean ball~\cite{M} and the Gauss space~\cite{Bor}. The asymptotic behaviour of $I_{\Omega,\nu}$ at $0$, in which we are interested, can be however evaluated more easily, and is therefore known for quite a wide class of domains, including Euclidean John domains (see~\cite{HK} combined with~\cite[Corollary 5.2.3, page 297]{M}) or product probability spaces~\cite{BCR}, which extend the Gauss space.

Given $m\in \mathbb N$ and a rearrangement-invariant space $X(\Omega,\nu)$, we will consider the $m$-th order Sobolev space $V^mX(\Omega,\nu)$ consisting of all $m$-times weakly differentiable functions on $\Omega$ whose $m$-th order weak derivatives belong to the space $X(\Omega,\nu)$. A precise definition of the notion {\it rearrangement-invariant space} can be found in Section~\ref{S:ri}, we just briefly recall that a rearrangement-invariant space is, roughly speaking, a Banach space consisting of $\nu$-measurable functions on $\Omega$ in which the norm of a function depends only on the measure of its level sets. A basic example of rearrangement-invariant spaces are Lebesgue spaces; besides them, the class of rearrangement-invariant spaces includes many further families of function spaces, such as Orlicz spaces, Lorentz spaces, etc.



In~\cite{CPS} we have shown that a continuous embedding of the Sobolev space $V^mX(\Omega,\nu)$ into a rearrangement-invariant space $Y(\Omega,\nu)$ is implied by a certain one-dimensional inequality depending on the representation norms $\|\cdot\|_{X(0,1)}$ and $\|\cdot\|_{Y(0,1)}$ of $X(\Omega,\nu)$ and $Y(\Omega,\nu)$, respectively, on $m$ and on the asymptotic behaviour of $I_{\Omega,\nu}$ at $0$, described in terms of a nondecreasing function $I$ giving a lower bound for the isoperimetric function at $0$. We remark that this inequality can be understood as boundedness of a certain integral operator from the representation space $X(0,1)$ into $Y(0,1)$. The above mentioned operator will be denoted by $H^m_I$ in what follows and has the form
\begin{equation}\label{E:H}
H^m_I f(t)=\frac{1}{(m-1)!}\int_t^1 \frac{|f(s)|}{I(s)} \left(\int_t^s \frac{\,dr}{I(r)} \right)^{m-1} \,ds, ~~t\in (0,1),
\end{equation}
for any Lebesgue measurable function $f$ on $(0,1)$. Moreover, if the function $I$ satisfies the additional assumption
\begin{equation}\label{E:assumption_intro}
\int_0^s \frac{\,dr}{I(r)} \approx \frac{s}{I(s)}, ~~s\in (0,1),
\end{equation}
(here, and in what follows, the symbol $\approx$ denotes the equivalence up to multiplicative constants),
then $H^m_I$ can be replaced by a considerably simpler operator, $K^m_I$, defined at every Lebesgue measurable function $f$ on $(0,1)$ by
$$
K^m_If(t)=\int_t^1 |f(s)|\frac{s^{m-1}}{(I(s))^m}\,ds, ~~t\in (0,1).
$$
We note that while $H^m_I$ is (possibly) a kernel operator, $K^m_I$ is just a weighted Hardy-type operator, which is far easier to work with. We also recall that important customary examples are available for the cases when~\eqref{E:assumption_intro} is valid as well as for the cases when~\eqref{E:assumption_intro} fails.


The main aim of the present paper is to prove that compactness of the operator $H^m_I$ from $X(0,1)$ into $Y(0,1)$
implies the compact embedding of $V^mX(\Omega,\nu)$ into $Y(\Omega,\nu)$ (Theorem~\ref{T:main}). We will also show (Theorem~\ref{T:K}) that if~\eqref{E:assumption_intro} is fulfilled, then the same result holds with $H^m_I$ replaced by $K^m_I$. 
The proof of Theorem~\ref{T:main} strongly depends on the use of almost-compact embeddings, called also absolutely continuous embeddings in some literature. They have been studied, e.g., in~\cite{FMP} and~\cite{S}. It is well known that such embeddings have a great significance for deriving compact Sobolev embeddings.

In many customary situations, the sufficient condition in terms of the operator $H^m_I$ turns out to be also necessary for compactness of the corresponding Sobolev embedding. We demonstrate this fact on the cases of Euclidean John domains, product probability spaces and Maz'ya classes of domains. The latter classes consist of those bounded Euclidean domains whose isoperimetric function is bounded from below by a multiple of some fixed power function. Unlike the case of John domains and product probability spaces, in which the necessity holds for each individual domain, for Maz'ya classes the sharpness is fulfilled in a wider sense: there is one domain in each class for which the necessity holds.

\medskip

The structure of the paper is as follows. In the next section we introduce rearrangement-invariant spaces and their almost-compact embeddings. Section~\ref{S:sobolev} contains a description of the measure spaces that will come into play, of their isoperimetric properties and of Sobolev spaces built upon rearrangement-invariant spaces over these measure spaces. We also recall those results of the paper~\cite{CPS} that are used in the proofs of our theorems.

Section~\ref{S:operators} contains one-dimensional results which play a key role in the proofs of our main results, appearing in Section~\ref{S:main}. We prove several theorems concerning compactness of the operator $H^m_J$ which is defined as in~\eqref{E:H} but with $I$ replaced by a more general function $J$ (which, in particular, is \textit{not} a-priori assumed to be non-decreasing on $(0,1]$). These results, thanks to the versatility of $J$, can be later used to handle compactness of both operators $H^m_I$ and $K^m_I$, and to provide thereby a unique scheme appropriate for the proofs of both main theorems~\ref{T:main} and~\ref{T:K}.


An important result of Section~\ref{S:operators} is Theorem~\ref{T:opt_range}, in which we characterize compactness of $H^m_J$ from a rearrangement-invariant space $X(0,1)$ into another rearrangement-invariant space $Y(0,1)$, denoted by
\begin{equation}\label{E:H_compactness}
H^m_J: X(0,1) \rightarrow \rightarrow Y(0,1),
\end{equation}
by the fact that $H^m_J$ maps the unit ball of $X(0,1)$ into a set of functions which is of uniformly absolutely continuous norm in $Y(0,1)$. A characterization of~\eqref{E:H_compactness} given in terms of the operator associate to $H^m_J$ and its uniform absolute continuity from the associate space to $Y(0,1)$, denoted by $Y'(0,1)$, into $X'(0,1)$, is provided in Theorem~\ref{T:opt_domain}. Each of the above mentioned conditions can be reformulated as an almost-compact embedding between certain rearrangement-invariant spaces. These two characterizations of~\eqref{E:H_compactness} are the key step in the proof of our main results.

We note that Theorems~\ref{T:opt_range} and~\ref{T:opt_domain} which characterize compactness of $H^m_J$ from $X(0,1)$ into $Y(0,1)$ require certain restrictions on the spaces involved, namely that $Y(0,1) \neq L^\infty(0,1)$ (Theorem~\ref{T:opt_range}) and $X(0,1)\neq L^1(0,1)$ (Theorem~\ref{T:opt_domain}). 
We also find, in Theorem~\ref{T:compact_operator}, an (almost) universal condition which characterizes~\eqref{E:H_compactness}. It has the form
\begin{equation}\label{E:H_ac}
\lim_{a\to 0_+} \sup_{\|f\|_{X(0,1)}\leq 1} \left\|H^m_J(\chi_{(0,a)}f)\right\|_{Y(0,1)} =0.
\end{equation}
There are still few special situations in which this new condition is not equivalent to~\eqref{E:H_compactness}. However, as observed in the first part of Theorem~\ref{T:main}, this cannot happen in the most important case when $J$ is nondecreasing on $(0,1]$. Furthermore, it turns out that in the cases when~\eqref{E:H_compactness} and~\eqref{E:H_ac} are not equivalent, the condition~\eqref{E:H_ac} is even more suitable to characterize compact Sobolev embeddings (see Theorem~\ref{T:K} and Remarks~\ref{T:remarks_5}, part \textup{(i)}).

Section~\ref{S:main} contains the main results of the paper that have been already described above.

In Section~\ref{S:concrete} we apply the results of Section~\ref{S:main} to the characterization of compact Sobolev embeddings on John domains (Theorem~\ref{T:john}), on Maz'ya classes of domains (Theorem~\ref{T:mazya}) and on product probability spaces (Theorem~\ref{T:probability}). The final Section~\ref{S:examples} then provides examples of compact Sobolev embeddings for concrete pairs of rearrangement-invariant spaces over the measure spaces discussed in Section~\ref{S:concrete}.

\section{Rearrangement-invariant spaces}\label{S:ri}

In this section we recall some basic facts from the theory of rearrangement-invariant spaces. Our standard general reference is~\cite{BS}.

Let $(R,\mu)$ be a nonatomic measure space satisfying $\mu(R)=1$. In fact, $R$ will always be a domain in $\mathbb R^n$ for some $n\in \mathbb N$. If the measure $\mu$ is omitted, we assume that it is the $n$-dimensional Lebesgue measure on $R$. 
We denote by $\mathcal M(R,\mu)$ the collection of all $\mu$-measurable functions on $R$ having its values in $[-\infty,\infty]$. We also set $\mathcal M_+(R,\mu)=\{f\in \mathcal M(R,\mu): f\geq 0 \textup{ on } R\}$.

Suppose that $f\in \mathcal M(R,\mu)$. Then the {\it distribution function} $\mu_f$ of the function $f$ is given by
$$
\mu_f(\lambda)=\mu(\{x\in R: |f(x)|>\lambda\}), ~~\lambda\in [0,\infty),
$$
and the {\it nonincreasing rearrangement} $f^*_\mu$ of $f$ is defined by
$$
f^*_\mu(t)=\inf\{\lambda \in [0,\infty): \mu_f(\lambda)\leq t\}, ~~t\in (0,1).
$$
Furthermore, we define $f^{**}_\mu$, the {\it maximal function} of $f^*_\mu$, by
$$
f^{**}_\mu(t)=\frac{1}{t}\int_0^t f^*_\mu(s)\,ds, ~~t\in (0,1).
$$
If two functions $f$, $g \in \mathcal M(R,\mu)$ fulfil $\mu_f=\mu_g$ (or, equivalently, $f^*_\mu=g^*_\mu$), we say that $f$ and $g$ are equimeasurable and write $f\sim_\mu g$. 

The {\it Hardy-Littlewood inequality} \cite[Chapter 2, Theorem 2.2]{BS} tells us that
$$
\int_R |fg|\,d\mu \leq \int_0^1 f^*_\mu(s) g^*_\mu(s) \,ds
$$
is satisfied for all $f$, $g\in \mathcal M(R,\mu)$. 

A functional $\|\cdot\|_{X(R,\mu)}: \mathcal M(R,\mu) \rightarrow [0,\infty]$ is called a {\it rearrangement-invariant norm} if, for all functions $u$, $v\in \mathcal M(R,\mu)$ and $f$, $g \in \mathcal M_+(R,\mu)$, for all sequences $(f_k)_{k=1}^\infty$ in $\mathcal M_+(R,\mu)$ and for all constants $a \geq 0$, the following properties are satisfied:

\begin{itemize}
\item[(P1)] $\|f\|_{X(R,\mu)}=0 \Leftrightarrow f=0$ $\mu$-a.e., $\|af\|_{X(R,\mu)}=a\|f\|_{X(R,\mu)}$,
\par\noindent 
$\|f+g\|_{X(R,\mu)} \leq\|f\|_{X(R,\mu)}+\|g\|_{X(R,\mu)}$;
\item[(P2)] $f\leq g$ $\mu$-a.e. $\Rightarrow$ $\|f\|_{X(R,\mu)}\leq \|g\|_{X(R,\mu)}$;
\item[(P3)] $f_k\uparrow f$ $\mu$-a.e. $\Rightarrow$ $\|f_k\|_{X(R,\mu)} \uparrow \|f\|_{X(R,\mu)}$;
\item[(P4)] $\|1\|_{X(R,\mu)}<\infty$;
\item[(P5)] $\int_R f\,d\mu \leq C \|f\|_{X(R,\mu)}$ for some constant $C>0$ independent of $f$;
\item[(P6)] $u \sim_\mu v$ $\Rightarrow$ $\|u\|_{X(R,\mu)}=\|v\|_{X(R,\mu)}$.
\end{itemize}
The collection of all functions $f\in \mathcal M(R,\mu)$ for which $\|f\|_{X(R,\mu)}<\infty$ is then called the \textit{rearrangement-invariant space} $X(R,\mu)$. We recall that the functional $\|\cdot\|_{X(R,\mu)}$ defines a norm on $X(R,\mu)$ and that $X(R,\mu)$ is a Banach space with respect to this norm.

We now summarize some basic properties of rearrangement-invariant spaces. We first note that each function $f\in X(R,\mu)$ is finite $\mu$-a.e.\ on $R$. Furthermore, the {\it Fatou lemma} \cite[Chapter 1, Lemma 1.5 (iii)]{BS} yields that whenever $(f_k)_{k=1}^\infty$ is a sequence in $X(R,\mu)$ converging to some function $f$ $\mu$-a.e. and fulfilling that $\liminf_{k\to \infty} \|f_k\|_{X(R,\mu)}<\infty$, then $f\in X(R,\mu)$ and
$$
\|f\|_{X(R,\mu)} \leq \liminf_{k\to \infty} \|f_k\|_{X(R,\mu)}.
$$
Moreover, if $(f_k)_{k=1}^\infty$ is a sequence which converges to some function $f$ in the norm of the space $X(R,\mu)$, then $(f_k)_{k=1}^\infty$ converges to $f$ in measure. In particular, there is a subsequence of $(f_k)_{k=1}^\infty$ which converges to $f$ $\mu$-a.e.\ on $R$.

Suppose that $\|\cdot\|_{X(R,\mu)}$ is a rearrangement-invariant norm. Let us consider the functional $\|\cdot\|_{X'(R,\mu)}: \mathcal M(R,\mu) \rightarrow [0,\infty]$ defined by
$$
\|f\|_{X'(R,\mu)} = \sup_{\|g\|_{X(R,\mu)}\leq 1} \int_R |fg| \,d\mu,  \quad f\in \mathcal M(R,\mu).
$$
Then $\|\cdot\|_{X'(R,\mu)}$ is a rearrangement-invariant norm, called the {\it associate norm} of $\|\cdot\|_{X(R,\mu)}$. The corresponding rearrangement-invariant space $X'(R,\mu)$ is then called the {\it associate space} of $X(R,\mu)$. It is not hard to observe that
$$
\|f\|_{X'(R,\mu)} = \sup_{\|g\|_{X(R,\mu)}\leq 1} \int_0^1 f^*_\mu (s) g^*_\mu(s)\,ds, ~~f\in \mathcal M(R,\mu).
$$


If $\|\cdot\|_{X(R,\mu)}$ and $\|\cdot\|_{Y(R,\mu)}$ are rearrangement-invariant norms, then the continuous embedding $X(R,\mu) \hookrightarrow Y(R,\mu)$ holds if and only if $X(R,\mu) \subseteq Y(R,\mu)$, see~\cite[Chapter 1, Theorem 1.8]{BS}. We shall write $X(R,\mu)=Y(R,\mu)$ if the set of functions belonging to $X(R,\mu)$ coincides with the set of functions belonging to $Y(R,\mu)$. In this case, the rearrangement-invariant norms $\|\cdot\|_{X(R,\mu)}$ and $\|\cdot\|_{Y(R,\mu)}$ are equivalent, in the sense that there are positive constants $C_1$, $C_2$ such that 
$$
C_1 \|f\|_{X(R,\mu)} \leq \|f\|_{Y(R,\mu)} \leq C_2 \|f\|_{X(R,\mu)}, \quad f\in \mathcal M(R,\mu).
$$
Furthermore, according to \cite[Chapter 1, Proposition 2.10]{BS}, the embedding $X(R,\mu) \hookrightarrow Y(R,\mu)$ is fulfilled if and only if $Y'(R,\mu) \hookrightarrow X'(R,\mu)$.

The {\it Luxemburg representation theorem} \cite[Chapter 2, Theorem 4.10]{BS} tells us that for every rearrangement-invariant norm $\|\cdot\|_{X(R,\mu)}$ there is a rearrangement-invariant norm $\|\cdot\|_{X(0,1)}$ fulfilling  
\begin{equation}\label{E:representation}
\|f\|_{X(R,\mu)} = \|f^*_\mu\|_{X(0,1)}, \quad f\in \mathcal M(R,\mu).
\end{equation}
The rearrangement-invariant norm $\|\cdot\|_{X(0,1)}$ is called the {\it representation norm} of $\|\cdot\|_{X(R,\mu)}$, and the corresponding rearrangement-invariant space $X(0,1)$ is called the {\it representation space} of $X(R,\mu)$.

Since the measure space $(R,\mu)$ is nonatomic, the range of $\mu$ is the interval $[0,1]$. Therefore, for every $t\in [0,1]$ there is a set $E_t\subseteq R$ for which $\mu(E_t)=t$. Suppose that $\|\cdot\|_{X(R,\mu)}$ is a rearrangement-invariant norm. Then the {\it fundamental function} $\varphi_X$ of $\|\cdot\|_{X(R,\mu)}$ is defined by
$$
\varphi_X(t)=\|\chi_{E_t}\|_{X(R,\mu)}, \quad t\in [0,1].
$$
It follows from the property \textup{(P6)} of rearrangement-invariant norms that the definition of $\varphi_X$ does not depend on the particular choice of sets $E_t$. Owing to~\cite[Corollary 5.3, Chapter 2]{BS}, $\varphi_X$ is \textit{quasiconcave}, in the sense that $\varphi_X(0)=0$, $\varphi_X$ is nondecreasing on $[0,1]$ and $\frac{\varphi_X(t)}{t}$ is nonincreasing on $(0,1]$.
Furthermore, one can easily observe that
$$
\varphi_X(t)=\|\chi_{(0,t)}\|_{X(0,1)},  \quad t\in (0,1].
$$ 

We say that a function $f\in X(R,\mu)$ has an \textit{absolutely continuous norm} in $X(R,\mu)$ if for every sequence $(E_k)_{k=1}^\infty$ of $\mu$-measurable subsets of $R$ fulfilling $\chi_{E_k} \rightarrow 0$ $\mu$-a.e. we have
$$
\lim_{k \to \infty} \|\chi_{E_k} f\|_{X(R,\mu)}=0.
$$
An easy observation yields that this can be equivalently reformulated by 
$$
\lim_{a\to 0_+} \|\chi_{(0,a)} f^*_\mu\|_{X(0,1)}=0.
$$
The collection of all functions having an absolutely continuous norm in $X(R,\mu)$ is denoted by $X_a(R,\mu)$. 

Further, we say that a subset $S$ of $X(R,\mu)$ is of \textit{uniformly absolutely continuous norm} in $X(R,\mu)$ if for every sequence $(E_k)_{k=1}^\infty$ of $\mu$-measurable subsets of $R$ fulfilling $\chi_{E_k} \rightarrow 0$ $\mu$-a.e.,
$$
\lim_{k \to \infty} \sup_{f\in S} \|\chi_{E_k} f\|_{X(R,\mu)}=0, 
$$
or, equivalently, 
$$
\lim_{a\to 0_+} \sup_{f\in S} \|\chi_{(0,a)} f^*_\mu\|_{X(0,1)}=0.
$$

Suppose that $\|\cdot\|_{X(R,\mu)}$ and $\|\cdot\|_{Y(R,\mu)}$ are rearrangement-invariant norms. We say that $X(R,\mu)$ is \textit{almost-compactly embedded} into $Y(R,\mu)$ and write $X(R,\mu) \overset{*}{\hookrightarrow} Y(R,\mu)$ if 
$$
\lim_{k\to \infty} \sup_{\|f\|_{X(R,\mu)}\leq 1} \|\chi_{E_k} f\|_{Y(R,\mu)}=0
$$
is satisfied for every sequence $(E_k)_{k=1}^\infty$ of $\mu$-measurable subsets of $R$ fulfilling $\chi_{E_k} \rightarrow 0$ $\mu$-a.e. Observe that $X(R,\mu) \overset{*}{\hookrightarrow} Y(R,\mu)$ holds if and only if the unit ball of $X(R,\mu)$ is of uniformly absolutely continuous norm in $Y(R,\mu)$. We shall make use of two characterizations of $X(R,\mu) \overset{*}{\hookrightarrow} Y(R,\mu)$, namely,
$$
\lim_{a\to 0_+} \sup_{\|f\|_{X(R,\mu)}\leq 1} \sup_{\mu(E)\leq a} \|\chi_Ef\|_{Y(R,\mu)}=0
$$
and
$$
\lim_{a\to 0_+} \sup_{\|f\|_{X(0,1)}\leq 1} \|\chi_{(0,a)}f^*\|_{Y(0,1)}=0.
$$ 
Note that the relation $X(R,\mu) \overset{*}{\hookrightarrow} Y(R,\mu)$ always implies $X(R,\mu) \hookrightarrow Y(R,\mu)$. Another necessary condition for $X(R,\mu) \overset{*}{\hookrightarrow} Y(R,\mu)$ is the following:
$$
\lim_{a\to 0_+} \frac{\varphi_{Y}(a)}{\varphi_{X}(a)} = 0,
$$
\cite[Section 3]{FMP}. Furthermore, $X(R,\mu) \overset{*}{\hookrightarrow} Y(R,\mu)$ is fulfilled if and only if $Y'(R,\mu) \overset{*}{\hookrightarrow} X'(R,\mu)$, see~\cite[Section 4, Property 5]{FMP}. 

Let us now give some examples of rearrangement-invariant norms. A basic example are the \textit{Lebesgue norms} $\|\cdot\|_{L^p(0,1)}$, $p\in [1,\infty]$, defined for all $f\in \mathcal M(R,\mu)$ by
$$
\|f\|_{L^p(R,\mu)}=
\begin{cases}
\left(\int_R |f|^p\,d\mu\right)^{1/p},\, &p<\infty;\\
\esup_{R} |f|,\, &p=\infty.
\end{cases}
$$
The corresponding rearrangement-invariant spaces $L^p(R,\mu)$ are then called the {\it Lebesgue spaces}. Recall that for each rearrangement-invariant space $X(R,\mu)$ the embeddings
\begin{equation}\label{E:p5}
L^\infty(R,\mu) \hookrightarrow X(R,\mu) \hookrightarrow L^1(R,\mu)
\end{equation}
hold. We denote by $C_{X}$ the constant from the latter embedding, that is, we have
\begin{equation}\label{E:constant}
\|f\|_{L^1(R,\mu)} \leq C_{X} \|f\|_{X(R,\mu)}, \quad f\in X(R,\mu),
\end{equation}
and $C_X$ is the least real number for which~\eqref{E:constant} is satisfied.

It is a well-known fact that a rearrangement-invariant space $X(R,\mu)$ is different from $L^\infty(R,\mu)$ if and only if $\lim_{s\to 0_+} \varphi_X(s)=0$. Furthermore, owing to~\cite[Theorems 5.2 and 5.3]{S}, $L^\infty(R,\mu) \overset{*}{\hookrightarrow} X(R,\mu)$ is characterized by $X(R,\mu) \neq L^\infty(R,\mu)$, and 
$X(R,\mu) \overset{*}{\hookrightarrow} L^1(R,\mu)$ holds if and only if $X(R,\mu) \neq L^1(R,\mu)$. 

One can consider also more general functionals $\|\cdot\|_{L^{p,q}(R,\mu)}$ and $\|\cdot\|_{L^{p,q;\alpha}(R,\mu)}$ which were studied, e.g., in~\cite{EOP} and~\cite{OP}. They are given for any $f\in \mathcal M(R,\mu)$ by
$$
\|f\|_{L^{p,q}(R,\mu)}=\left\|f^*_\mu(s) s^{\frac{1}{p}-\frac{1}{q}}\right\|_{L^q(R,\mu)}
$$
and
$$
\|f\|_{L^{p,q;\alpha}(R,\mu)}=\left\|f^*_\mu(s) s^{\frac{1}{p}-\frac{1}{q}}\left(\log \frac{2}{s} \right)^\alpha\right\|_{L^q(R,\mu)},
$$
respectively. Here, we assume that $p\in [1,\infty]$, $q\in [1,\infty]$, $\alpha \in \mathbb R$, and use the convention that $1/\infty=0$. Note that $\|\cdot\|_{L^p(R,\mu)}=\|\cdot\|_{L^{p,p}(R,\mu)}$ and $\|\cdot\|_{L^{p,q}(R,\mu)}=\|\cdot\|_{L^{p,q;0}(R,\mu)}$ for every such $p$ and $q$. However, it turns out that under these assumptions on $p$, $q$ and $\alpha$, $\|\cdot\|_{L^{p,q}(R,\mu)}$ and $\|\cdot\|_{L^{p,q;\alpha}(R,\mu)}$ do not have to be rearrangement-invariant norms. To ensure that $\|\cdot\|_{L^{p,q;\alpha}(R,\mu)}$ is equivalent to a rearrangement-invariant norm, we need to assume that one of the following conditions is satisfied:
\begin{align}\label{E:L-Z_r.i.1}
&p=q=1, ~~ \alpha \geq 0;\\
&1<p<\infty; \label{E:L-Z_r.i.2}\\
&p=\infty,~~ q <\infty, ~~ \alpha + \frac{1}{q}<0; \label{E:L-Z_r.i.3}\\
&p=q=\infty, ~~\alpha \leq 0. \label{E:L-Z_r.i.4}
\end{align}
In this case, $\|\cdot\|_{L^{p,q}(R,\mu)}$ is called a \textit{Lorentz norm}, $\|\cdot\|_{L^{p,q;\alpha}(R,\mu)}$ is called a \textit{Lorentz-Zygmund norm} and the corresponding rearrangement-invariant spaces $L^{p,q}(R,\mu)$ and $L^{p,q;\alpha}(R,\mu)$ are called {\it Lorentz spaces} and {\it Lorentz-Zygmund spaces}, respectively.

Furthermore, if $\|\cdot\|_{L^{p_1,q_1;\alpha_1}(R,\mu)}$ and $\|\cdot\|_{L^{p_2,q_2;\alpha_2}(R,\mu)}$ are equivalent to rearrangement-invariant norms then
$$
L^{p_1,q_1;\alpha_1}(R,\mu) \hookrightarrow L^{p_2,q_2;\alpha_2}(R,\mu)
$$
holds if and only if $p_1>p_2$, or $p_1=p_2$ and one of the following conditions is satisfied:
\begin{align} \nonumber
&p_1<\infty,~~q_1 \leq q_2, ~~ \alpha_1 \geq \alpha_2;\\
&p_1=\infty,~~ q_1 \leq q_2, ~~ \alpha_1+\frac{1}{q_1} \geq \alpha_2 + \frac{1}{q_2}; \label{E:lz_embedding}\\
&q_2 < q_1, ~~ \alpha_1+\frac{1}{q_1} > \alpha_2 + \frac{1}{q_2}. \nonumber
\end{align}


\medskip

When dealing with rearrangement-invariant spaces we will proceed in the following way in what follows. We always start with some rearrangement-invariant norm $\|\cdot\|_{X(0,1)}$. Then we consider the functional $\|\cdot\|_{X(R,\mu)}$
defined by~\eqref{E:representation}. According to~\cite[Chapter 2, Theorem 4.9]{BS}, $\|\cdot\|_{X(R,\mu)}$ is a rearrangement-invariant norm. Furthermore, in view of the Luxemburg representation theorem, each rearrangement-invariant norm over $(R,\mu)$ can be constructed in this way, which justifies our procedure.


\section{Sobolev spaces}\label{S:sobolev}

Let $n\in \mathbb N$ and let $\Omega$ be a domain in $\mathbb R^n$ endowed with a measure $\nu$ satisfying $\nu(\Omega)=1$. We assume that $\nu$ is absolutely continuous with respect to the $n$-dimensional Lebesgue measure $\lambda_n$, and we denote by $\omega$ the density of $\nu$ with respect to $\lambda_n$ (that is, whenever $E\subseteq \Omega$ is $\nu$-measurable, we have $\nu(E)=\int_E \omega(x)\,dx$). The function $\omega$ is supposed to be Borel measurable and fulfilling that for a.e.\ $x\in \Omega$ there is an open ball $B_x$ centered in $x$ such that $B_x \subseteq \Omega$ and 
$$
\einf_{B_x} \omega >0.
$$
Notice that a subset of $\Omega$ (or a function defined on $\Omega$) is $\nu$-measurable if and only if it is Lebesgue measurable. We shall write {\it measurable} instead of {\it Lebesgue measurable} in what follows.

For every measurable $E\subseteq \Omega$ we define its {\it perimeter} in $(\Omega,\nu)$ by
$$
P_\nu(E,\Omega)=\int_{\Omega \cap \partial^M E} \omega(x)\,d\mathcal H^{n-1}(x),
$$
where $\partial^M E$ stands for the essential boundary of $E$, in the sense of geometric measure theory (see, e.g., \cite{M}), and $\mathcal H^{n-1}$ denotes the $(n-1)$-dimensional Hausdorff measure. The {\it isoperimetric function} $I_{\Omega,\nu}: [0,1] \rightarrow [0,\infty]$ is then defined by
$$
I_{\Omega,\nu}(s)=\inf \left\{P_\nu(E,\Omega): E\subseteq \Omega, ~s\leq \nu(E)\leq \frac{1}{2} \right\}
$$
if $s\in [0,\frac{1}{2}]$, and by $I_{\Omega,\nu}(s)=I_{\Omega,\nu}(1-s)$ if $s\in (\frac{1}{2},1]$.

\begin{definition}\label{T:definition}
Let $(\Omega,\nu)$ be as above, and let $I: (0,1] \rightarrow (0,\infty)$ be a function. We say that $(\Omega,\nu,I)$ is a {\it compatible triplet} if the following conditions are satisfied:

\medskip
\textup{(C1)} $I$ is nondecreasing on $(0,1]$;

\textup{(C2)}  $I$ satisfies
\begin{equation}\label{E:infimum}
\inf_{t\in (0,1]} \frac{I(t)}{t}>0;
\end{equation}

\textup{(C3)} there exists $c\in (0,2)$ such that
\begin{equation}\label{E:isoperimetric}
I_{\Omega,\nu}(t) \geq cI(ct), ~~t\in (0,1/2].
\end{equation}
\end{definition}

We note that if $I$ fulfils \textup{(C1)} and there is a constant $D>0$ for which
$$
I_{\Omega,\nu}(t) \geq DI(t), ~~t\in (0,1/2], 
$$
then \textup{(C3)} is fulfilled as well, since
$$
I_{\Omega,\nu}(t) \geq DI(t) \geq \min\{D,1\}I(\min\{D,1\}t), ~~t\in (0,1/2],
$$
owing to \textup{(C1)}.

Let us now give a few examples of compatible triplets.

\medskip
Suppose that $n\in \mathbb N$, $n\geq 2$. We recall that a bounded domain $\Omega \subseteq \mathbb R^n$ is called a {\it John domain} if there exist a constant $c\in (0,1)$ and a point $x_0\in \Omega$ such that for every $x\in \Omega$ there are $l >0$ and a rectifiable curve $\varpi :[0,l] \rightarrow \Omega$, parametrized by arclength, such that $\varpi(0)=x$, $\varpi(l)=x_0$, and
$$
dist(\varpi(r),\partial \Omega)\geq cr, ~~r\in [0,l].
$$ 
In what follows, we shall consider (with no loss of generality) only John domains whose Lebesgue measure is equal to $1$. 

It is known that each John domain satisfies
$$
I_\Omega(t) \approx t^{\frac{1}{n'}}, ~~t\in [0,1/2],
$$
where $n'=\frac{n}{n-1}$.
Therefore, if we denote $I(t)=t^{\frac{1}{n'}}$, $t\in (0,1]$, then $(\Omega,\lambda_n,I)$ is a compatible triplet. 

\medskip
Let $\alpha \in [\frac{1}{n'},1]$. We denote by $\mathcal J_\alpha$ the {\it Maz'ya class} of all bounded Euclidean domains $\Omega \subseteq \mathbb R^n$ with $\lambda_n(\Omega)=1$ fulfilling that there is a positive constant $C$, possibly depending on $\Omega$, such that 
$$
I_{\Omega}(s) \geq Cs^\alpha, ~~s\in [0,1/2].
$$
Set $I_\alpha(t)=t^\alpha$, $t\in (0,1]$. Then $(\Omega,\lambda_n,I_\alpha)$ is another example of a compatible triplet.

\medskip
As a final example we mention product probability spaces, namely, $\mathbb R^n$ with the product probability measure defined as follows. 

Assume that $\Phi: [0,\infty) \rightarrow [0,\infty)$ is a strictly increasing convex function such that it is twice continuously differentiable on $(0,\infty)$, $\sqrt{\Phi}$ is concave on $[0,\infty)$ and $\Phi(0)=0$. Define the one-dimensional probability measure $\mu_{\Phi}=\mu_{\Phi,1}$ by
\begin{equation}\label{E:one_dim}
d\mu_{\Phi}(x)=c_\Phi e^{-\Phi(|x|)}\,dx,
\end{equation}
where the constant $c_\Phi>0$ is chosen in such a way that $\mu_\Phi(\mathbb R)=1$. We also define the product measure $\mu_{\Phi,n}$ on $\mathbb R^n$, $n\geq 2$, by
\begin{equation}\label{E:n_dim}
\mu_{\Phi,n}=\underbrace{\mu_\Phi \times \dots \times \mu_\Phi}_{n-\textup{times}}.
\end{equation}
Then $(\mathbb R^n, \mu_{\Phi,n})$ is a probability space for every $n\in \mathbb N$ and we have
$$
d\mu_{\Phi,n}(x)=(c_\Phi)^n e^{-(\Phi(|x_1|)+\Phi(|x_2|)+\dots + \Phi(|x_n|))}\,dx.
$$

Define the function $F_\Phi:\mathbb R \rightarrow (0,1)$ by
$$
F_\Phi(t)=\int_t^\infty c_\Phi e^{-\Phi(|r|)}\,dr, ~t\in \mathbb R,
$$
the function $I_\Phi:(0,1) \rightarrow (0,\infty)$ by
$$
I_\Phi(t)=c_\Phi e^{-\Phi(|F_\Phi^{-1}(t)|)}, ~~t\in (0,1),
$$
and the function $L_\Phi: (0,1] \rightarrow (0,\infty)$ by
\begin{equation}\label{E:l_phi}
L_\Phi(t)=t\Phi'\left(\Phi^{-1}\left(\log \frac{2}{t}\right)\right), ~~t\in (0,1].
\end{equation}
Then the isoperimetric function of $(\mathbb R^n, \mu_{\Phi,n})$ satisfies
\begin{equation}\label{E:probability_isoperimetric}
I_{\mathbb R^n, \mu_{\Phi,n}}(t) \approx I_\Phi(t) \approx L_\Phi(t), ~~t\in (0,1/2],
\end{equation}
see~\cite[Proposition 13 and Theorem 15]{BCR}. Further, it was shown in~\cite[Lemma 11.1 \textup{(i)}]{CPS} that $L_\Phi$ is nondecreasing on $(0,1]$. Therefore, $(\mathbb R^n, \mu_{\Phi,n}, L_\Phi)$ is a compatible triplet. 

The main example of product probability measures we have just defined is the $n$-dimensional \textit{Gauss measure}
$$
\,d\gamma_n(x)=(2\pi)^{-\frac{n}{2}}e^{\frac{-|x|^2}{2}}\,dx,
$$
which can be obtained by setting
$$
\Phi(t)=\frac{1}{2} t^2, ~~t\in [0,\infty),
$$
into~\eqref{E:one_dim} (if $n$=1) or~\eqref{E:n_dim} (if $n>1$). 

More generally, measures associated with
$$
\Phi(t)=\frac{1}{\beta} t^\beta, ~~t\in [0,\infty),
$$
for some $\beta \in [1,2]$ are also examples of product probability measures. They are called the \textit{Boltzmann measures}. For each $\beta \in [1,2]$, such $n$-dimensional measure is denoted by $\gamma_{n,\beta}$.
We of course have $\gamma_{n,2}=\gamma_n$.

\medskip
We shall now define Sobolev spaces built upon rearrangement-invariant spaces over $(\Omega,\nu)$. The measure space $(\Omega,\nu)$ is required to satisfy all the above mentioned properties and, moreover, the inequality 
\begin{equation}\label{E:isop_estim}
I_{\Omega,\nu}(t) \geq Ct, ~~t\in [0,1/2],
\end{equation}
has to be fulfilled for some positive constant $C$ independent of $t$. Notice that condition~\eqref{E:isop_estim} is satisfied whenever there is a function $I$ for which $(\Omega,\nu,I)$ is a compatible triplet.

Let $m\in \mathbb N$ and let $u$ be an $m$-times weakly differentiable function on $\Omega$. Given $k\in \{1,2,\dots,m\}$, we denote by $\nabla^k u$ the vector of all $k$-th order weak derivatives of $u$. Moreover, we set $\nabla^0 u=u$. Then the $m$-th order Sobolev space built upon a rearrangement-invariant space $X(\Omega,\nu)$ is the set
\begin{align*}
V^mX(\Omega, \nu)=
\{u: ~&u \textup{ is an $m$-times weakly differentiable function on }\Omega \\
&\textup{such that }|\nabla^m u| \in X(\Omega, \nu) \}.
\end{align*}
According to~\cite[Corollary 4.3]{CPS}, the inclusions $V^mX(\Omega,\nu) \subseteq L^1(\Omega,\nu)$ and $V^mX(\Omega,\nu) \subseteq V^kL^1(\Omega,\nu)$, $k=1,2,\dots,m-1$, are satisfied. Hence, the expression
\begin{equation}\label{E:sobolev}
\|u\|_{V^mX(\Omega, \nu)}=\sum_{k=0}^{m-1}\| |\nabla^k u| \|_{L^1(\Omega, \nu)} +\| |\nabla^m u| \|_{X(\Omega, \nu)}
\end{equation}
defines a norm on $V^mX(\Omega,\nu)$.

In what follows we shall denote by
\begin{equation}\label{E:c_e}
V^mX(\Omega,\nu) \hookrightarrow Y(\Omega,\nu)
\end{equation}
the continuous embedding of the Sobolev space $V^mX(\Omega,\nu)$ into a rearrangement-invariant space $Y(\Omega,\nu)$,
and we shall write
$$
V^mX(\Omega,\nu) \hookrightarrow \hookrightarrow Y(\Omega,\nu)
$$
in order to denote that the embedding~\eqref{E:c_e} is compact.

\medskip
We now state a theorem and a proposition which were proved in~\cite{CPS} and which will be used in what follows.

\begin{theorem}[{\cite[Theorem 5.1]{CPS}}]\label{T:CPS}
Suppose that $(\Omega,\nu,I)$ is a compatible triplet. Let $m\in \mathbb N$ and let $\|\cdot\|_{X(0,1)}$ and $\|\cdot\|_{Y(0,1)}$ be rearrangement-invariant norms. If there exists a constant $C_1>0$ such that 
\begin{equation}\label{E:bound}
\left\|\int_t^1 \frac{f(s)}{I(s)} \left(\int_t^s \frac{\,dr}{I(r)}\right)^{m-1}\,ds\right\|_{Y(0,1)} \leq C_1 \|f\|_{X(0,1)}
\end{equation}
for every nonnegative $f\in X(0,1)$, then
\begin{equation}\label{E:emb}
V^mX(\Omega,\nu) \hookrightarrow Y(\Omega,\nu)
\end{equation}
and, equivalently, there is a constant $C_2>0$ such that
\begin{equation}\label{E:poincare}
\|u\|_{Y(\Omega,\nu)} \leq C_2 \||\nabla^m u|\|_{X(\Omega,\nu)}
\end{equation}
for every $u\in V^mX(\Omega,\nu)$ fulfilling that $\int_\Omega \nabla^k u\,d\nu =0$ for $k=0,1,\dots,m-1$.
\end{theorem}

We finally mention that, given $m\in \mathbb N$ and a rearrangement-invariant norm $\|\cdot\|_{X(0,1)}$, one can define the more customary Sobolev space $W^mX(\Omega,\nu)$ by 
\begin{align*}
W^mX(\Omega, \nu)=
\{u: ~&u \textup{ is an $m$-times weakly differentiable function on }\Omega \\
&\textup{such that }|\nabla^k u| \in X(\Omega, \nu) \textup{ for } k=0,1,\dots,m\}.
\end{align*}
The set $W^mX(\Omega,\nu)$ equipped with the norm
\begin{equation*}
\|u\|_{W^mX(\Omega, \nu)}=\sum_{k=0}^m\| |\nabla^k u| \|_{X(\Omega, \nu)}
\end{equation*}
is easily seen to be a normed linear space. We always have the continuous embedding $W^mX(\Omega,\nu) \hookrightarrow V^mX(\Omega,\nu)$. The reverse embedding is not true in general, however, we have the following

\begin{proposition}[{\cite[Proposition 4.5]{CPS}}]\label{T:propprop}
Suppose that $(\Omega,\nu)$ is as in the first paragraph of the present section and, moreover, that
$$
\int_0^{\frac 12} \frac{\,ds}{I_{\Omega,\nu}(s)} <\infty.
$$
Let $m\in \mathbb N$ and let $\|\cdot\|_{X(0,1)}$ be a rearrangement-invariant norm. Then
$$
V^mX(\Omega,\nu)=W^mX(\Omega,\nu),
$$
up to equivalent norms. 
\end{proposition}

In particular, if $(\Omega,\nu,I)$ is a compatible triplet such that 
$$
\int_0^1 \frac{\,ds}{I(s)}<\infty,
$$ 
then $V^mX(\Omega,\nu)=W^mX(\Omega,\nu)$ for every $m\in \mathbb N$ and for every rearrangement-invariant norm $\|\cdot\|_{X(0,1)}$. Indeed, property \textup{(C3)} of compatible triplets yields that there is $c\in (0,2)$ for which
$$
\int_0^{\frac{1}{2}} \frac{\,ds}{I_{\Omega,\nu}(s)} \leq \frac{1}{c} \int_0^{\frac{1}{2}} \frac{\,ds}{I(cs)} = \frac{1}{c^2} \int_0^{\frac{c}{2}} \frac{\,ds}{I(s)} \leq \frac{1}{c^2} \int_0^1 \frac{\,ds}{I(s)} <\infty.
$$
The result now follows from Proposition~\ref{T:propprop}.

\section{Compact operators}\label{S:operators}

In this section we give several characterizations of compactness of certain one-dimensional operator on rearrangement-invariant spaces. These characterizations play a central role in the proofs of our main results in the following Section~\ref{S:main}. Moreover, the results of this section will be used to characterize compactness of this operator on concrete classes of rearrangement-invariant spaces (see Section~\ref{S:examples}).

\medskip
Let $J:(0,1]\rightarrow (0,\infty)$ be a measurable function satisfying
\begin{equation}\label{E:supremum}
\inf_{t\in (0,1]} \frac{J(t)}{t} >0.
\end{equation}
We set
\begin{equation}\label{E:inf}
J_a=\inf_{t\in [a,1]} J(t), ~~a\in (0,1),
\end{equation}
and observe that for every $a\in (0,1)$, 
$$
J_a \geq Ca>0,
$$ 
where $C=\inf_{t\in (0,1]} J(t)/t$.

We shall consider the operator $H_J$ defined by 
\begin{equation}\label{E:h}
H_Jf(t)=\int_t^1 \frac{|f(s)|}{J(s)}\,ds, ~~f\in \mathcal M(0,1), ~t\in (0,1),
\end{equation}
and the operator $R_J$ defined by
$$
R_Jf(t)=\frac{1}{J(t)}\int_0^t |f(s)|\,ds, ~~f\in \mathcal M(0,1), ~t\in (0,1).
$$
Furthermore, given $j\in \mathbb N$, we define the operators $H_J^j$ and $R^j_j$ by
\begin{equation}\label{E:composition}
H_J^j=\underbrace{H_J \circ H_J \circ \dots \circ H_J}_{j-\textup{times}} \textup{ and } R_J^j=\underbrace{R_J \circ R_J \circ \dots \circ R_J}_{j-\textup{times}}.
\end{equation}
Then
\begin{equation}\label{E:def_H}
H_J^jf(t)=\frac{1}{(j-1)!}\int_t^1 \frac{|f(s)|}{J(s)} \left(\int_t^s \frac{\,dr}{J(r)}\right)^{j-1}\,ds, ~~f\in \mathcal M(0,1), ~t\in (0,1),
\end{equation}
and
$$
R^j_Jf(t)=\frac{1}{(j-1)!J(t)} \int_0^t \left(\int_s^t \frac{\,dr}{J(r)} \right)^{j-1} |f(s)|\,ds, ~~f\in \mathcal M(0,1), ~t\in (0,1),
$$
see~\cite[Remarks 8.2]{CPS}. For technical reasons, we also set $H^0_J=R^0_J=Id$. 

We remark that the operators $H^j_J$ and $R^j_J$ are associate in the sense that for every $f\in \mathcal M_+(0,1)$ and $g\in \mathcal M_+(0,1)$ we have
\begin{equation}\label{E:associate}
\int_0^1 f(s) H^j_Jg(s)\,ds=\int_0^1 g(s) R^j_Jf(s)\,ds.
\end{equation}
We also observe that whenever $j\in \mathbb N$ and $f\in \mathcal M(0,1)$ then $H^j_Jf$ is nonincreasing on $(0,1)$. Finally, given $a\in (0,1]$, the equality
\begin{align}\label{E:char}
H^j_J(\chi_{(0,a)})(t)
&=\chi_{(0,a)}(t) \frac{1}{(j-1)!} \int_t^a \frac{1}{J(s)}\left(\int_t^s \frac{\,dr}{J(r)}\right)^{j-1}\,ds\\
&=\chi_{(0,a)}(t) \frac{1}{j!} \left(\int_t^a \frac{\,dr}{J(r)}\right)^{j}, ~~t\in (0,1), \nonumber
\end{align}
which follows from the change of variables formula, will be of use.

\medskip
Given two rearrangement-invariant norms $\|\cdot\|_{X(0,1)}$ and $\|\cdot\|_{Y(0,1)}$, we shall write
$$
H^j_J: X(0,1) \rightarrow Y(0,1)
$$
in order to denote that the operator $H^j_J$ is bounded from $X(0,1)$ into $Y(0,1)$. Our goal is to find necessary and sufficient conditions for compactness of $H^j_J$ from $X(0,1)$ into $Y(0,1)$, denoted by
\begin{equation}\label{E:compact_notation}
H^j_J: X(0,1) \rightarrow \rightarrow Y(0,1).
\end{equation}
The first result in this connection is the following

\begin{theorem}\label{T:compact_operator}
Let $J:(0,1]\rightarrow (0,\infty)$ be a measurable function satisfying~\eqref{E:supremum} and let $j\in \mathbb N$. Suppose that $\|\cdot\|_{X(0,1)}$ and $\|\cdot\|_{Y(0,1)}$ are rearrangement-invariant norms. Consider the following two conditions:

\medskip
\textup{(i)} $H_J^j: X(0,1) \rightarrow \rightarrow Y(0,1);$

\textup{(ii)} $\lim_{a\to 0_+} \sup_{\|f\|_{X(0,1)}\leq 1} \|H_J^j(\chi_{(0,a)}f)\|_{Y(0,1)} = 0.$

\medskip\noindent
If $X(0,1)=L^1(0,1)$, $Y(0,1)=L^\infty(0,1)$, $j=1$ and $\lim_{a\to 0_+} \esup_{t\in (0,a)} \frac{1}{J(t)}=0$, then \textup{(ii)} is satisfied but \textup{(i)} is not. In all other cases, \textup{(i)} holds if and only if \textup{(ii)} holds.
\end{theorem}

Theorem~\ref{T:compact_operator} provides a full characterization of compactness of the operator $H^j_J$. By modifications of condition \textup{(ii)} of Theorem~\ref{T:compact_operator} we obtain two more necessary and sufficient conditions for~\eqref{E:compact_notation}. Their equivalence to~\eqref{E:compact_notation} holds in a slightly less general setting, but the strength of these characterizations rests on the possibility to reformulate them as almost-compact embeddings between certain rearrangement-invariant spaces. This connection between compactness of $H^j_J$ and almost-compactness of an embedding becomes a key tool for the proof of our main result, Theorem~\ref{T:main}.

We shall now introduce a family of rearrangement-invariant spaces whose almost-compact embeddings are suitable for characterization of~\eqref{E:compact_notation}.

Let $\|\cdot\|_{X(0,1)}$ be a rearrangement-invariant norm. For every $f \in \mathcal M(0,1)$ we define the functional $\|\cdot\|_{(X_{j,J}^r)'(0,1)}$ by
\begin{equation}\label{E:X}
\|f\|_{(X_{j,J}^r)'(0,1)}= \|R^j_Jf^*\|_{X'(0,1)}
=\frac{1}{(j-1)!}\left\|\frac{1}{J(s)} \int_0^s \left(\int_t^s \frac{\,dr}{J(r)} \right)^{j-1} f^*(t)\,dt \right\|_{X'(0,1)}.
\end{equation}
Then, according to~\cite[Proposition 8.3]{CPS}, $\|\cdot\|_{(X^r_{j,J})'(0,1)}$ is a rearrangement-invariant norm and its associate norm $\|\cdot\|_{X^r_{j,J}(0,1)}$ fulfils
\begin{equation}\label{E:optimal_range}
H_J^j: X(0,1) \rightarrow X_{j,J}^r(0,1).
\end{equation}
Moreover, $X_{j,J}^r(0,1)$ is the optimal range for $X(0,1)$ with respect to the operator $H^j_J$, that is, $X^r_{j,J}(0,1)$ is the smallest rearrangement-invariant space for which~\eqref{E:optimal_range} is satisfied.

\medskip
The following theorem characterizes~\eqref{E:compact_notation} by means of the space $X^r_{j,J}(0,1)$.

\begin{theorem}\label{T:opt_range}
Let $J:(0,1]\rightarrow (0,\infty)$ be a measurable function satisfying~\eqref{E:supremum} and let $j\in \mathbb N$. Suppose that $\|\cdot\|_{X(0,1)}$ and $\|\cdot\|_{Y(0,1)}$ are rearrangement-invariant norms. If
\begin{equation}\label{E:or}
Y(0,1) \neq L^\infty(0,1) \textup{ or } \int_0^1 \frac{\,dr}{J(r)}=\infty,
\end{equation}
then the following conditions are equivalent:

\medskip
\textup{(i)} $H_J^j: X(0,1) \rightarrow \rightarrow Y(0,1);$

\textup{(ii)} $\lim_{a\to 0_+} \sup_{\|f\|_{X(0,1)}\leq 1} \|\chi_{(0,a)}H_J^j f\|_{Y(0,1)} = 0;$

\textup{(iii)} $X_{j,J}^r(0,1) \overset{*}{\hookrightarrow} Y(0,1).$
\end{theorem}

\begin{remarks}\label{T:remark_opt_range}
\textup{(a)} Condition \textup{(ii)} of Theorem~\ref{T:opt_range} tells us that the set $\{H^j_Jf: \|f\|_{X(0,1)}\leq 1\}$ is of uniformly absolutely continuous norm in $Y(0,1)$.

\textup{(b)} Suppose that $Y(0,1)=L^\infty(0,1)$. Then we easily observe that none of the conditions \textup{(ii)} and \textup{(iii)} of Theorem~\ref{T:opt_range} can be satisfied, no matter what $J$, $j$ and $\|\cdot\|_{X(0,1)}$ are. If, moreover, $\int_0^1 \frac{\,dr}{J(r)}=\infty$, it follows from Theorem~\ref{T:opt_range} that condition \textup{(i)} is also never fulfilled (since it is equivalent to \textup{(ii)} and \textup{(iii)}). However, this is no longer true in the case when $\int_0^1 \frac{\,dr}{J(r)}<\infty$, because
\begin{equation}\label{E:compact_l_infty_1}
H^j_J: L^\infty(0,1) \rightarrow \rightarrow L^\infty(0,1)
\end{equation} 
holds for every $j\in \mathbb N$ in this situation. Indeed, by~\eqref{E:char} we have
\begin{align*}
\lim_{a\to 0_+} \sup_{\|f\|_{L^\infty(0,1)}\leq 1} \left\|H^j_J(\chi_{(0,a)}f)\right\|_{L^\infty(0,1)}
&=\lim_{a\to 0_+} \left\|H^j_J(\chi_{(0,a)})\right\|_{L^\infty(0,1)}\\
&=\lim_{a\to 0_+} \frac{1}{j!}\left\|\chi_{(0,a)}(t)\left(\int_t^a \frac{\,dr}{J(r)}\right)^j\right\|_{L^\infty(0,1)}\\
&=\lim_{a\to 0_+} \frac{1}{j!} \left(\int_0^a \frac{\,dr}{J(r)}\right)^j=0.
\end{align*}
Hence, due to Theorem~\ref{T:compact_operator}, condition~\eqref{E:compact_l_infty_1} is satisfied.
\end{remarks}

\begin{remark}\label{T:remark}
Suppose that $J: (0,1] \rightarrow (0,\infty)$ is a measurable function fulfilling~\eqref{E:supremum}, $j\in \mathbb N$ and $\|\cdot\|_{X(0,1)}$ is a rearrangement-invariant norm such that 
\begin{equation}\label{E:bounded}
H^j_J: X(0,1) \rightarrow L^\infty(0,1).
\end{equation}
Then 
\begin{equation}\label{E:compact_l_infty}
H^j_J: X(0,1) \rightarrow \rightarrow Y(0,1)
\end{equation}
is fulfilled for all rearrangement-invariant spaces $Y(0,1)\neq L^\infty(0,1)$. Indeed, since $L^\infty(0,1)$ is the smallest rearrangement-invariant space over $(0,1)$ and~\eqref{E:bounded} is satisfied, $L^\infty(0,1)$ is the optimal range for $X(0,1)$ with respect to the operator $H^j_J$, and therefore $X^r_{j,J}(0,1)=L^\infty(0,1)$. The assumption $Y(0,1)\neq L^\infty(0,1)$ yields that $L^\infty(0,1) \overset{*}{\hookrightarrow} Y(0,1)$. Thus, according to Theorem~\ref{T:opt_range}, we obtain~\eqref{E:compact_l_infty}.

Furthermore, having only the information that~\eqref{E:bounded} holds, we cannot decide whether~\eqref{E:compact_l_infty} is satisfied with $Y(0,1)=L^\infty(0,1)$ or not. As an example, consider the function $J=1$ on $(0,1]$ and the rearrangement-invariant norm $\|\cdot\|_{X(0,1)}=\|\cdot\|_{L^1(0,1)}$. In this case, condition~\eqref{E:bounded} is easily seen to be satisfied for every $j\in \mathbb N$. Due to Theorem~\ref{T:compact_operator}, \eqref{E:compact_l_infty} is fulfilled if and only if
\begin{equation}\label{E:lim_0_l_1}
\lim_{a\to 0_+} \sup_{\|f\|_{L^1(0,1)} \leq 1} \|H^j_J(\chi_{(0,a)}f)\|_{L^\infty(0,1)}=0.
\end{equation}
For every $a\in (0,1)$ we have
$$
\sup_{\|f\|_{L^1(0,1)} \leq 1} \|H^j_J(\chi_{(0,a)}f)\|_{L^\infty(0,1)}
=\sup_{\|f\|_{L^1(0,1)} \leq 1} \int_0^a |f(s)|s^{j-1}\,ds
=\|\chi_{(0,a)}(s)s^{j-1}\|_{L^\infty(0,1)} = a^{j-1},
$$
and hence \eqref{E:lim_0_l_1} does not hold for $j=1$, but it holds for $j>1$. 
\end{remark}

We shall now define another family of rearrangement-invariant spaces whose almost-compact embeddings will be used for characterization of~\eqref{E:compact_notation}.

Let $\|\cdot\|_{Y(0,1)}$ be a rearrangement-invariant norm fulfilling
\begin{equation}\label{E:cond_Y}
\left\|\left(\int_t^1 \frac{\,dr}{J(r)}\right)^j\right\|_{Y(0,1)} <\infty.
\end{equation}
For every $f\in \mathcal M(0,1)$ define the functional $\|\cdot\|_{Y_{j,J}^d(0,1)}$ by
\begin{align}\label{E:dom}
\|f\|_{Y_{j,J}^d(0,1)}
&=\sup_{h\sim f} \|H^j_Jh\|_{Y(0,1)} + \|f\|_{L^1(0,1)}\\
&=\sup_{0\leq h \sim f} \frac{1}{(j-1)!}\left\|\int_t^1 \frac{h(s)}{J(s)} \left(\int_t^s \frac{\,dr}{J(r)}\right)^{j-1} \,ds\right\|_{Y(0,1)} +\|f\|_{L^1(0,1)}. \nonumber
\end{align}
Then $\|\cdot\|_{Y^d_{j,J}(0,1)}$ is a rearrangement-invariant norm and the corresponding rearrangement-invariant space $Y^d_{j,J}(0,1)$ is the optimal domain for $Y(0,1)$ with respect to the operator $H^j_J$, in the sense of the following

\begin{proposition}\label{T:proposition}
Let $J:(0,1]\rightarrow (0,\infty)$ be a measurable function satisfying~\eqref{E:supremum} and let $j\in \mathbb N$. Suppose that $\|\cdot\|_{Y(0,1)}$ is a rearrangement-invariant norm fulfilling~\eqref{E:cond_Y}. Then $\|\cdot\|_{Y_{j,J}^d(0,1)}$ is a rearrangement-invariant norm and
\begin{equation}\label{E:domain}
H_J^j: Y_{j,J}^d(0,1) \rightarrow Y(0,1).
\end{equation}
Moreover, $Y_{j,J}^d(0,1)$ is the largest rearrangement-invariant space for which~\eqref{E:domain} is satisfied. 

Conversely, if $\|\cdot\|_{Y(0,1)}$ is a rearrangement-invariant norm which does not fulfil~\eqref{E:cond_Y} then there is no rearrangement-invariant space $Y_{j,J}^d(0,1)$ such that condition~\eqref{E:domain} is satisfied. 
\end{proposition} 

The last result of this section provides a necessary and sufficient condition for compactness of the operator $H^j_J$ given in terms of the optimal domain space.

\begin{theorem}\label{T:opt_domain}
Let $J:(0,1]\rightarrow (0,\infty)$ be a measurable function satisfying~\eqref{E:supremum} and let $j\in \mathbb N$. Suppose that $\|\cdot\|_{X(0,1)}$ is a rearrangement-invariant norm such that $X(0,1) \neq L^1(0,1)$ and $\|\cdot\|_{Y(0,1)}$ is a rearrangement-invariant norm fulfilling~\eqref{E:cond_Y}. Then the following conditions are equivalent:

\medskip
\textup{(i)} $H_J^j: X(0,1) \rightarrow \rightarrow Y(0,1);$

\textup{(ii)} $\lim_{a\to 0_+} \sup_{\|f\|_{X(0,1)}\leq 1} \sup_{\lambda_1(E)\leq a} \|H_J^j(\chi_Ef)\|_{Y(0,1)} = 0;$

\textup{(iii)} $X(0,1) \overset{*}{\hookrightarrow} Y^d_{j,J}(0,1).$ 
\end{theorem}

\begin{remarks}\label{T:remarks_opt_domain}
\textup{(a)} Using the definition of the associate norm and~\eqref{E:associate} we deduce that condition \textup{(ii)} of Theorem~\ref{T:opt_domain} is equivalent to equality
$$
\lim_{a\to 0_+} \sup_{\|f\|_{Y'(0,1)}\leq 1} \sup_{\lambda_1(E)\leq a} \|\chi_ER^j_J f\|_{X'(0,1)} = 0,
$$
which tells us that the set $\{R^j_Jf: \|f\|_{Y'(0,1)}\leq 1\}$ is of uniformly absolutely continuous norm in $X'(0,1)$. 

\textup{(b)} It is not hard to verify that conditions \textup{(ii)} and \textup{(iii)} of Theorem~\ref{T:opt_domain} are never fulfilled with $X(0,1)=L^1(0,1)$. However, this is not the case of condition \textup{(i)} since we have already observed in Remark~\ref{T:remark} that \textup{(i)} is satisfied with $J \equiv 1$, $X(0,1)=L^1(0,1)$ and $Y(0,1)=L^\infty(0,1)$ whenever $j>1$. Furthermore, in contrast to Theorem~\ref{T:opt_range}, which holds in the exceptional case $Y(0,1)=L^\infty(0,1)$ for quite a wide class of functions $J$, there are only very few functions $J$ for which Theorem~\ref{T:opt_domain} is fulfilled with $X(0,1)=L^1(0,1)$. We shall now characterize all nondecreasing functions having this property (recall that the case when the function $J$ is nondecreasing is the most significant from the point of view of applications to compact Sobolev embeddings). 

Fix a nondecreasing function $J$ fulfilling~\eqref{E:supremum}. Then Theorem~\ref{T:opt_domain} is {\it not} true for $X(0,1)=L^1(0,1)$ if and only if there is $j\in \mathbb N$ and a rearrangement-invariant norm $\|\cdot\|_{Y(0,1)}$ such that 
\begin{equation}\label{E:ypsilon}
H^j_J: L^1(0,1) \rightarrow \rightarrow Y(0,1).
\end{equation}
Notice that whenever a rearrangement-invariant norm $\|\cdot\|_{Y(0,1)}$ satisfies~\eqref{E:ypsilon} then, in particular,
\begin{equation}\label{E:ypsilon_2}
H^j_J: L^1(0,1) \rightarrow Y(0,1),
\end{equation}
and therefore~\eqref{E:cond_Y} holds (consequently, $\|\cdot\|_{Y(0,1)}$ satisfies the assumption of Theorem~\ref{T:opt_domain}). Indeed, if~\eqref{E:cond_Y} was not fulfilled then, due to Proposition~\ref{T:proposition}, there would be no rearrangement-invariant space $Y^d_{j,J}(0,1)$ such that~\eqref{E:domain} is fulfilled. This would contradict~\eqref{E:ypsilon_2}.

Fix $j\in \mathbb N$. In order to decide whether there is a rearrangement-invariant space $Y(0,1)$ for which~\eqref{E:ypsilon} holds, it is enough to study whether
\begin{equation}\label{E:ypsilon=l1}
H^j_J: L^1(0,1) \rightarrow \rightarrow L^1(0,1),
\end{equation}
since $L^1(0,1)$ is the largest rearrangement-invariant space over $(0,1)$. Due to Theorem~\ref{T:compact_operator}, condition~\eqref{E:ypsilon=l1} is equivalent to
$$
\lim_{a\to 0_+} \sup_{\|f\|_{L^1(0,1)}\leq1} \|H^j_J(\chi_{(0,a)}f)\|_{L^1(0,1)}=0,
$$
which is characterized in the last section of the present paper (see Theorem~\ref{T:ll} (b)) by
$$
\lim_{t\to 0_+} \frac{t}{J(t)}=0,
$$
(no matter what $j$ is). Combining this with~\eqref{E:supremum} we deduce that the only case when~\eqref{E:ypsilon=l1} is not fulfilled for some $j\in \mathbb N$ (and hence Theorem~\ref{T:opt_domain} is fulfilled with $X(0,1)=L^1(0,1)$) is the one when there is a set $M\subseteq (0,1)$ such that $0\in \overline{M}$ and
$$
J(s) \approx s \quad \textup{on $M$}.
$$
\end{remarks}

\begin{remark}\label{T:remark_2}
Suppose that $J: (0,1] \rightarrow (0,\infty)$ is a measurable function satisfying~\eqref{E:supremum} and $j\in \mathbb N$. If $\|\cdot\|_{Y(0,1)}$ is a rearrangement-invariant norm such that
\begin{equation}\label{E:bounded_2}
H^j_J: L^1(0,1) \rightarrow Y(0,1),
\end{equation}
then 
\begin{equation}\label{E:compact_l_1}
H^j_J: X(0,1) \rightarrow \rightarrow Y(0,1)
\end{equation}
is fulfilled for all rearrangement-invariant spaces $X(0,1)\neq L^1(0,1)$. To verify this, we first recall that the rearrangement-invariant norm $\|\cdot\|_{Y(0,1)}$ satisfies~\eqref{E:cond_Y} (a proof of this fact was given in Remarks~\ref{T:remarks_opt_domain} (b)). Thus, we can consider the rearrangement-invariant norm $\|\cdot\|_{Y^d_{j,J}(0,1)}$ defined by~\eqref{E:dom}. Since $L^1(0,1)$ is the largest rearrangement-invariant space over $(0,1)$ and~\eqref{E:bounded_2} is satisfied, $L^1(0,1)$ is the optimal domain for $Y(0,1)$ with respect to the operator $H^j_J$, and hence it follows from Proposition~\ref{T:proposition} that $Y^d_{j,J}(0,1)=L^1(0,1)$. The assumption $X(0,1)\neq L^1(0,1)$ yields that $X(0,1) \overset{*}{\hookrightarrow} L^1(0,1)$. Using the last two facts and Theorem~\ref{T:opt_domain}, we obtain~\eqref{E:compact_l_1}.

Furthermore, having only the information that~\eqref{E:bounded_2} holds, we cannot decide whether~\eqref{E:compact_l_infty} is satisfied with $X(0,1)=L^1(0,1)$ or not. An example supporting this was already presented in Remark~\ref{T:remark}.
\end{remark}

\begin{remark}
The classical result due to Luxemburg and Zaanen~\cite{LZ} relates compactness of a kernel integral operator to its absolute continuity, and to absolute continuity of the associate operator. Let us describe this result in some more detail, and then compare it to our Theorems~\ref{T:opt_range} and~\ref{T:opt_domain}.

Let $X(0,1)$ and $Y(0,1)$ be rearrangement-invariant spaces, let $T$ be a kernel integral operator and let $T'$ be the operator associate to $T$ (in a similar sense in which our operator $R^j_J$ is associate to $H^j_J$, see~\cite{LZ} for a precise definition). Assume that $Tf\in Y_a(0,1)$ for every $f\in X(0,1)$, and $T'g\in X_a'(0,1)$ for every $g\in Y'(0,1)$. 
In~\cite{LZ} it is proved (even in a more general setting) that we have the equivalence of the following three conditions:

\medskip
\textup{(a)} $T: X(0,1) \rightarrow \rightarrow Y(0,1)$.

\textup{(b)} The set $\{Tf: \|f\|_{X(0,1)}\leq 1\}$ is of uniformly absolutely continuous norm in $Y(0,1)$. 

\textup{(c)} The set $\{T'g: \|g\|_{Y'(0,1)}\leq 1\}$ is of uniformly absolutely continuous norm in $X'(0,1)$.

If we set $T=H^j_J$, then its associate operator $T'$ is the operator $R^j_J$. We have observed in Remarks~\ref{T:remark_opt_range} and~\ref{T:remarks_opt_domain} that in this case condition \textup{(b)} is exactly the condition \textup{(ii)} of Theorem~\ref{T:opt_range} and condition \textup{(c)} is identical to condition \textup{(ii)} of Theorem~\ref{T:opt_domain}. The main difference between our result and the one proved in~\cite{LZ} is the following: when proving that \textup{(a)} implies \textup{(b)} or \textup{(c)}, we do not need to assume that either the operator $H^j_J$, or $R^j_J$, has its range in the set of functions of absolutely continuous norm, since this fact already follows from \textup{(a)} (under the indispensable assumption that $Y(0,1)\neq L^\infty(0,1)$ or $X(0,1)\neq L^1(0,1)$, respectively). It can be easily observed that such a claim fails when $T$ is an arbitrary kernel integral operator.
\end{remark}

We shall now prove the results of this section. We start with the following

\begin{lemma}\label{T:m-1}
Let $J:(0,1]\rightarrow (0,\infty)$ be a measurable function satisfying~\eqref{E:supremum}, let $j\in \mathbb N$ and let $\alpha \in [0,j]$. Suppose that $\|\cdot\|_{X(0,1)}$ and $\|\cdot\|_{Y(0,1)}$ are rearrangement-invariant norms. Consider the following conditions: 

\medskip
\textup{(i)} $H^j_J: X(0,1) \rightarrow \rightarrow Y(0,1);$

\textup{(ii)} $\lim_{a\to 0_+} \sup_{\|f\|_{X(0,1)}\leq 1} \|H^j_J(\chi_{(0,a)}f)\|_{Y(0,1)} =0;$

\textup{(iii)} $\lim_{a\to 0_+} \left\|\chi_{(0,a)}(t)\left(\int_t^1 \frac{\,dr}{J(r)} \right)^{\alpha}\right\|_{Y(0,1)} =0.$

\noindent
Then \textup{(i)} implies \textup{(ii)}. Furthermore, provided that~\eqref{E:or} is satisfied, \textup{(ii)} implies \textup{(iii)}.
\end{lemma}

\begin{proof}
Suppose that \textup{(i)} holds. Then $H^j_J$ is bounded from $X(0,1)$ into $Y(0,1)$, so, in particular, for every $k\in \mathbb N$
$$
\sup_{\|f\|_{X(0,1)}\leq 1} \|H^j_J (\chi_{(0,1/k)}f)\|_{Y(0,1)} \leq \sup_{\|f\|_{X(0,1)}\leq 1} \|H^j_Jf\|_{Y(0,1)}<\infty.
$$
Therefore, given $k\in \mathbb N$, we can find a nonnegative measurable function $f_k$ on $(0,1)$ such that $\|f_k\|_{X(0,1)}\leq 1$ and
\begin{equation}\label{E:sup}
\sup_{\|f\|_{X(0,1)}\leq 1} \|H^j_J (\chi_{(0,1/k)}f)\|_{Y(0,1)} < \|H^j_J (\chi_{(0,1/k)}f_k)\|_{Y(0,1)} +\frac{1}{k}.
\end{equation}
Since the sequence $(\chi_{(0,1/k)}f_k)_{k=1}^\infty$ is bounded in $X(0,1)$, the assumption \textup{(i)} yields that there is a subsequence $(f_{k_\ell})_{\ell=1}^\infty$ of $(f_k)_{k=1}^\infty$ such that $(H^j_J(\chi_{(0,1/{k_\ell})}f_{k_\ell}))_{\ell=1}^\infty$ converges to some function $f$ in the norm of the space $Y(0,1)$. Moreover, the subsequence can be found in such a way that $(H^j_J(\chi_{(0,1/{k_\ell})}f_{k_\ell}))_{\ell=1}^\infty$ converges to $f$ a.e.\ on $(0,1)$.  But $H^j_J(\chi_{(0,1/{k_\ell})}f_{k_\ell})=0$ on $(1/k_\ell,1)$, which implies that $H^j_J(\chi_{(0,1/{k_\ell})}f_{k_\ell}) \rightarrow 0$ pointwise. Thus, $f=0$ a.e.\ on $(0,1)$. This yields
$$
\lim_{\ell\to \infty} \|H^j_J (\chi_{(0,1/{k_\ell})}f_{k_\ell})\|_{Y(0,1)}=0.
$$
Now, the inequality~\eqref{E:sup} gives 
$$
\lim_{\ell\to \infty} \sup_{\|f\|_{X(0,1)}\leq 1} \|H^j_J (\chi_{(0,1/{k_\ell})}f)\|_{Y(0,1)}=0.
$$
Since the function 
$$
a \mapsto \sup_{\|f\|_{X(0,1)}\leq 1} \|H^j_J (\chi_{(0,a)}f)\|_{Y(0,1)}
$$ 
is nondecreasing on $(0,1)$, we obtain \textup{(ii)}, as required.

Now, suppose that \textup{(ii)} holds and~\eqref{E:or} is satisfied. If $\int_0^1 \frac{\,dr}{J(r)} <\infty$, then necessarily $Y(0,1)\neq L^\infty(0,1)$, so
$$
\lim_{a\to 0_+} \left\|\chi_{(0,a)}(t)\left(\int_t^1 \frac{\,dr}{J(r)} \right)^{\alpha}\right\|_{Y(0,1)}
\leq \left(\int_0^1 \frac{\,dr}{J(r)} \right)^{\alpha} \lim_{a\to 0_+} \left\|\chi_{(0,a)}\right\|_{Y(0,1)}=0.
$$

Conversely, assume that $\int_0^1 \frac{\,dr}{J(r)} =\infty$. Given $a\in (0,1)$, there is $b\in (0,a)$ such that
$$
\max\left(1, \int_a^1 \frac{\,dr}{J(r)}\right) \leq \int_t^a \frac{\,dr}{J(r)}, ~~t\in (0,b).
$$
Then, in particular,
$$
1\leq \int_t^1 \frac{\,dr}{J(r)}, ~~t\in (0,b),
$$
and therefore also
$$
1\leq \left(\int_t^1 \frac{\,dr}{J(r)}\right)^{j-\alpha}, ~~t\in (0,b).
$$
Thus,
\begin{align*}
\limsup_{d\to 0_+} \left\|\chi_{(0,d)}(t) \left( \int_t^1 \frac{\,dr}{J(r)} \right)^{\alpha}\right\|_{Y(0,1)}
&\leq \limsup_{d\to 0_+} \left\|\chi_{(0,d)}(t) \left( \int_t^1 \frac{\,dr}{J(r)} \right)^j\right\|_{Y(0,1)}\\
&= \limsup_{d\to 0_+} \left\|\chi_{(0,d)}(t) \left( \int_t^a \frac{\,dr}{J(r)} + \int_a^1 \frac{\,dr}{J(r)} \right)^j\right\|_{Y(0,1)}\\
&\leq \limsup_{d\to 0_+} 2^j\left\|\chi_{(0,d)}(t) \left( \int_t^a \frac{\,dr}{J(r)} \right)^j\right\|_{Y(0,1)}\\
&\leq 2^j\left\|\chi_{(0,a)}(t) \left( \int_t^a \frac{\,dr}{J(r)} \right)^j\right\|_{Y(0,1)}\\
&=j!2^j \left\|H^j_J(\chi_{(0,a)})\right\|_{Y(0,1)} \quad \textup{(by~\eqref{E:char})}\\
&= j!2^j \|1\|_{X(0,1)} \left\|H^j_J\left(\chi_{(0,a)} \frac{1}{\|1\|_{X(0,1)}} \right)\right\|_{Y(0,1)}\\
&\leq j! 2^j \|1\|_{X(0,1)} \sup_{\|f\|_{X(0,1)} \leq 1} \left\|H^j_J(\chi_{(0,a)}f)\right\|_{Y(0,1)}.
\end{align*}
Passing to limit when $a$ tends to $0$, we obtain \textup{(iii)}, as required.
\end{proof}

\begin{proof}[Proof of Theorem~\ref{T:opt_range}]
\textup{(i)} $\Rightarrow$ \textup{(ii)} According to Lemma~\ref{T:m-1}, condition \textup{(i)} implies
$$
\lim_{a\to 0_+} \sup_{\|f\|_{X(0,1)}\leq 1}\|H^j_J(\chi_{(0,a)}f)\|_{Y(0,1)} =0.
$$ 
Fix $b\in (0,1)$. For every $t\in (0,1)$, we have
\begin{align}\label{E:estimate}
H^j_J(\chi_{(b,1)}f)(t) \nonumber
&= \frac{1}{(j-1)!}\int_{\max (b,t)}^1 \frac{|f(s)|}{J(s)} \left(\int_t^s \frac{\,dr}{J(r)}\right)^{j-1}\,ds\\ 
&\leq \frac{1}{(j-1)!J_b} \left(\int_t^1 \frac{\,dr}{J(r)}\right)^{j-1} \int_b^1 |f(s)|\,ds\\ \nonumber
&\leq \frac{1}{(j-1)!J_b} \left(\int_t^1 \frac{\,dr}{J(r)}\right)^{j-1} \|f\|_{L^1(0,1)}\\ \nonumber
&\leq \frac{C_{X}}{(j-1)!J_b} \left(\int_t^1 \frac{\,dr}{J(r)}\right)^{j-1} \|f\|_{X(0,1)}.
\end{align}
Hence,
\begin{align*}
&\limsup_{a\to 0_+} \sup_{\|f\|_{X(0,1)}\leq1} \left\|\chi_{(0,a)}H^j_Jf\right\|_{Y(0,1)}\\
&\leq \limsup_{a\to 0_+} \sup_{\|f\|_{X(0,1)}\leq1} \left\|\chi_{(0,a)}H^j_J(\chi_{(0,b)}f)\right\|_{Y(0,1)}
+\limsup_{a\to 0_+} \sup_{\|f\|_{X(0,1)}\leq1} \left\|\chi_{(0,a)}H^j_J(\chi_{(b,1)}f)\right\|_{Y(0,1)}\\
&\leq \sup_{\|f\|_{X(0,1)}\leq1} \left\|H^j_J(\chi_{(0,b)}f)\right\|_{Y(0,1)} 
+\limsup_{a\to 0_+} \frac{C_X}{(j-1)!J_b}\left\|\chi_{(0,a)}(t)\left(\int_t^1 \frac{\,dr}{J(r)}\right)^{j-1}\right\|_{Y(0,1)}\\
&=\sup_{\|f\|_{X(0,1)}\leq1} \left\|H^j_J(\chi_{(0,b)}f)\right\|_{Y(0,1)},
\end{align*}
thanks to Lemma~\ref{T:m-1}. Passing to limit when $b$ tends to $0$, we get
$$
\lim_{a\to 0_+} \sup_{\|f\|_{X(0,1)}\leq1} \left\|\chi_{(0,a)}H^j_Jf\right\|_{Y(0,1)}=0,
$$
as required.

\textup{(ii)} $\Rightarrow$ \textup{(i)} The proof is completely analogous to that of~\cite[Theorem 3.1, implication \textup{(ii)} $\Rightarrow$ \textup{(i)}]{PP}, even with simplifications following from the fact that we consider rearrangement-invariant spaces over a finite interval.

\textup{(ii)} $\Leftrightarrow$ \textup{(iii)} Using the definition of the associate norm and the equality~\eqref{E:associate}, we get
\begin{align*}
\lim_{a\to 0_+} \sup_{\|f\|_{Y'(0,1)}\leq 1}  \|\chi_{(0,a)} f^*\|_{(X_{j,J}^r)'(0,1)}
&=\lim_{a\to 0_+} \sup_{\|f\|_{Y'(0,1)}\leq 1} \left\|R^j_J(\chi_{(0,a)}f^*)\right\|_{X'(0,1)} \\
&=\lim_{a\to 0_+} \sup_{\|f\|_{Y'(0,1)}\leq 1} \sup_{\|g\|_{X(0,1)}\leq 1} \int_0^1 |g(s)|R^j_J(\chi_{(0,a)}f^*)(s) \,ds\\
&= \lim_{a\to 0_+} \sup_{\|g\|_{X(0,1)}\leq 1} \sup_{\|f\|_{Y'(0,1)}\leq 1} \int_0^1 \chi_{(0,a)}(s) f^*(s)  H^j_Jg(s)\,ds\\
&=\lim_{a\to 0_+} \sup_{\|g\|_{X(0,1)}\leq 1} \sup_{\|f\|_{Y'(0,1)}\leq 1} \int_0^1 f^*(s) \left(\chi_{(0,a)} H^j_Jg\right)^*(s)\,ds\\
&=\lim_{a\to 0_+} \sup_{\|g\|_{X(0,1)}\leq 1} \|\chi_{(0,a)} H^j_Jg\|_{Y(0,1)}.
\end{align*}
Note that the second last equality holds because $\chi_{(0,a)}H^j_Jg$ is nonincreasing on $(0,1)$ for every $a\in (0,1)$ and $g\in X(0,1)$. Thus, we have proved that \textup{(ii)} holds if and only if $Y'(0,1) \overset{*}{\hookrightarrow} (X_{j,J}^r)'(0,1)$. Since the latter condition is equivalent to \textup{(iii)}, the proof is complete.
\end{proof}

\begin{proof}[Proof of Theorem~\ref{T:compact_operator}]
According to Lemma~\ref{T:m-1}, \textup{(i)} implies \textup{(ii)} (with no restriction on $\|\cdot\|_{X(0,1)}$, $\|\cdot\|_{Y(0,1)}$, $J$ and $j$). 

Suppose that condition~\eqref{E:or} is satisfied. Then the implication \textup{(ii)} $\Rightarrow$ \textup{(i)} follows from the proof of Theorem~\ref{T:opt_range}. So, assume that $Y(0,1)=L^\infty(0,1)$, $\int_0^1 \frac{\,dr}{J(r)} <\infty$ and \textup{(ii)} holds. We observe that to prove \textup{(i)}, it is enough to show that for every $a\in (0,1)$, the operator $H_{J,a}^j: f \mapsto H_J^j(\chi_{(a,1)}f)$ is compact from $X(0,1)$ into $L^\infty(0,1)$. Indeed, thanks to \textup{(ii)} we have
\begin{align*}
\lim_{a\to 0_+} \sup_{\|f\|_{X(0,1)}\leq 1} \|H_J^jf - H_{J,a}^jf\|_{L^\infty(0,1)}
&= \lim_{a\to 0_+} \sup_{\|f\|_{X(0,1)}\leq 1} \|H_J^j(\chi_{(0,a)}f)\|_{L^\infty(0,1)}= 0,
\end{align*}
so $H_J^j$ will be a norm limit of compact operators, and thus itself a compact operator.

Fix $a\in (0,1)$. For every $f\in X(0,1)$ we can consider the function $H^j_{J,a}f$ to be defined by~\eqref{E:def_H} (with $f$ replaced by $\chi_{(a,1)}f$) on the entire $[0,1]$. Then $H^j_{J,a}f$ is continuous on $[0,1]$, and it follows from~\eqref{E:estimate} (with $b=a$) and from the fact that $\int_0^1 \frac{\,dr}{J(r)}<\infty$ that the image by $H^j_{J,a}$ of the unit ball of $X(0,1)$ is bounded in $C([0,1])$ with the standard supremum norm by $\frac{C_X}{(j-1)!J_a}\left(\int_0^1 \frac{\,dr}{J(r)}\right)^{j-1}$. 

Now, assume that $j>1$. Let $0\leq t_1 <t_2\leq 1$. Then, using the result of the previous paragraph with $j$ replaced by $j-1$, we get
\begin{align*}
\sup_{\|f\|_{X(0,1)}\leq 1} \left|H^j_{J,a}f(t_1)-H^j_{J,a}f(t_2)\right|
&=\sup_{\|f\|_{X(0,1)}\leq 1}\int_{t_1}^{t_2} \frac{H^{j-1}_J(\chi_{(a,1)}f)(s)}{J(s)}\,ds\\
&\leq \sup_{\|f\|_{X(0,1)}\leq 1}\left\|H^{j-1}_J(\chi_{(a,1)}f)\right\|_{L^\infty(0,1)}\int_{t_1}^{t_2} \frac{\,ds}{J(s)}\\
&\leq \frac{C_X}{(j-2)!J_a}\left(\int_0^1 \frac{\,dr}{J(r)}\right)^{j-2} \int_{t_1}^{t_2} \frac{\,ds}{J(s)}.
\end{align*}
The last expression goes to $0$ when $t_2-t_1$ tends to $0$ thanks to the absolute continuity of the Lebesgue integral. This proves that the image by $H^j_{J,a}$ of the unit ball of $X(0,1)$ is equicontinuous.

Let $j=1$ and $X(0,1)\neq L^1(0,1)$. Then we deduce that 
\begin{align*}
\sup_{\|f\|_{X(0,1)}\leq 1} |H_{J,a}^jf(t_1)-H_{J,a}^jf(t_2)| 
&=\sup_{\|f\|_{X(0,1)}\leq 1} \int_{t_1}^{t_2} \frac{\chi_{(a,1)}(s) |f(s)|}{J(s)}\,ds\\
&\leq \frac{1}{J_a} \sup_{\|f\|_{X(0,1)}\leq 1} \int_{t_1}^{t_2} |f(s)|\,ds\\ 
&\leq \frac{1}{J_a} \sup_{\|f\|_{X(0,1)}\leq 1} \sup_{\lambda_1(E)\leq t_2-t_1} \|\chi_E f\|_{L^1(0,1)},
\end{align*}
which goes to $0$ when $t_2-t_1$ tends to $0$ thanks to the almost-compact embedding $X(0,1) \overset{*}{\hookrightarrow} L^1(0,1)$. This proves the equicontinuity in this case. Arzela-Ascoli theorem now yields that $H^j_{J,a}$ maps the unit ball of $X(0,1)$ into a relatively compact set in $C([0,1])$. Since the space $C([0,1])$ is continuously embedded into $L^\infty(0,1)$, the operator $H^j_{J,a}$ is compact from $X(0,1)$ into $L^\infty(0,1)$, as required.

Finally, suppose that $X(0,1)=L^1(0,1)$, $Y(0,1)=L^\infty(0,1)$ and $j=1$. We have
\begin{align}\label{E:ll}
\lim_{a\to 0_+} \sup_{\|f\|_{X(0,1)}\leq 1} \left\|H_J(\chi_{(0,a)}f)\right\|_{Y(0,1)}
&\approx\lim_{a\to 0_+} \sup_{\|f\|_{L^1(0,1)}\leq 1} \left\|H_J(\chi_{(0,a)}f)\right\|_{L^\infty(0,1)} \nonumber\\
&=\lim_{a\to 0_+} \sup_{\|f\|_{L^1(0,1)}\leq 1} \int_0^a \frac{|f(s)|}{J(s)}\,ds \\ \nonumber
&=\lim_{a\to 0_+} \left\|\chi_{(0,a)} \frac{1}{J}\right\|_{L^\infty(0,1)} = \lim_{a\to 0_+} \esup_{t\in (0,a)} \frac{1}{J(t)},
\end{align}
hence, condition \textup{(ii)} is satisfied if and only if $\lim_{a\to 0_+} \esup_{t\in (0,a)} \frac{1}{J(t)}=0$. Thus, the implication \textup{(ii)} $\Rightarrow$ \textup{(i)} holds in the case that $X(0,1)=L^1(0,1)$, $Y(0,1)=L^\infty(0,1)$, $j=1$ and $\lim_{a\to 0_+} \esup_{t\in (0,a)} \frac{1}{J(t)} \neq 0$, because the assumption \textup{(ii)} is not fulfilled.

To complete the proof, we will show that if $X(0,1)=L^1(0,1)$, $Y(0,1)=L^\infty(0,1)$ and $j=1$ then condition \textup{(i)} is not satisfied. Indeed, since $\frac{1}{J}>0$ on $(0,1)$, there is $\varepsilon>0$ and a set $M\subseteq (0,1)$ of measure $\frac{1}{2}$ such that $\frac{1}{J}\geq \varepsilon$ on $M$. Let $(x_n)_{n=1}^\infty$ be a sequence of points in $[0,1)$ fulfilling $\lambda_1((x_n,1) \cap M)= \frac{1}{2^n}$, $n\in \mathbb N$. Given $n\in \mathbb N$, set $f_n=2^n\chi_{(x_n,x_{n+1}]\cap M}$. Then
$$
\|f_n\|_{X(0,1)} \approx \|f_n\|_{L^1(0,1)}= 2^n \lambda_1((x_n,x_{n+1}]\cap M)
=2^n(\lambda_1((x_n,1)\cap M)-\lambda_1((x_{n+1},1)\cap M))=\frac{1}{2}.
$$
Therefore, the sequence $(f_n)_{n=1}^\infty$ is bounded in $X(0,1)$. Let $m$, $n\in \mathbb N$, $m<n$. Since both $H_Jf_m$ and $H_Jf_n$ are continuous on $(0,1)$, we have
\begin{align*}
\|H_Jf_n-H_Jf_m\|_{Y(0,1)} 
&\approx \|H_Jf_n-H_Jf_m\|_{L^\infty(0,1)} 
\geq |H_Jf_n(x_n)-H_Jf_m(x_n)|\\
&=2^n \int_{x_n}^{x_{n+1}} \frac{\chi_M(s)}{J(s)}\,ds\geq \frac{\varepsilon}{2}.
\end{align*}
Consequently, there is no subsequence $(f_{n_k})_{k=1}^\infty$ of $(f_n)_{n=1}^\infty$ for which $(H_Jf_{n_k})_{k=1}^\infty$ is convergent in $Y(0,1)$. Hence, $H_J$ is not compact from $X(0,1)$ into $Y(0,1)$. The proof is complete.
\end{proof}

\begin{proof}[Proof of Proposition~\ref{T:proposition}]
Suppose that $\|\cdot\|_{Y(0,1)}$ is a rearrangement-invariant norm fulfilling~\eqref{E:cond_Y}. We start by showing that $\|\cdot\|_{Y^d_{j,J}(0,1)}$ is a rearrangement-invariant norm. Properties \textup{(P5)}, \textup{(P6)} as well as the first two properties in \textup{(P1)} trivially hold. Since the functional $\|\cdot\|_{L^1(0,1)}$ satisfies the axioms of rearrangement-invariant norms, we only have to verify that the functional $\|\cdot \|_{Z(0,1)}$ defined by  
$$
\|f\|_{Z(0,1)} = \sup_{h \sim f} \|H^j_Jh\|_{Y(0,1)}, ~~f\in \mathcal M(0,1),
$$ 
fulfils the triangle inequality and properties \textup{(P2)} -- \textup{(P4)}. However, \textup{(P2)} and \textup{(P3)} can be proved exactly in the same way as it is done in~\cite[proof of Lemma 4.2]{CP}. Furthermore, using the fact that each nonnegative function equimeasurable to $1$ is equal to $1$ a.e., and applying the equality~\eqref{E:char} with $a=1$, we get
\begin{align}\label{E:P4}
\|1\|_{Z(0,1)}
=\|H^j_J(1)\|_{Y(0,1)}
=\frac{1}{j!}\left\|\left(\int_t^1 \frac{\,dr}{J(r)}\right)^{j} \right\|_{Y(0,1)}<\infty,
\end{align}
which proves \textup{(P4)}. Thus, it only remains to verify the triangle inequality. 

Suppose that $u$, $v$ are nonnegative simple functions on $(0,1)$. We will show that 
\begin{equation}\label{E:triangle}
\|u+v\|_{Z(0,1)} \leq \|u\|_{Z(0,1)} + \|v\|_{Z(0,1)}.
\end{equation}
Assume that a nonnegative function $h$ on $(0,1)$ satisfies $h \sim u+v$. Then it is not hard to observe that there exist nonnegative simple functions $h_u$ and $h_v$ on $(0,1)$ such that $h=h_u+h_v$, $h_u \sim u$ and $h_v \sim v$. Hence,
\begin{align*}
\|H^j_Jh\|_{Y(0,1)} \leq \|H^j_Jh_u\|_{Y(0,1)} + \|H^j_Jh_v\|_{Y(0,1)}
\leq \|u\|_{Z(0,1)} + \|v\|_{Z(0,1)}.
\end{align*}
Passing to supremum over all $h$ we get~\eqref{E:triangle}.

Let $f$, $g\in \mathcal M_+(0,1)$. Then there are two sequences of nonnegative simple functions $(u_n)_{n=1}^\infty$ and $(v_n)_{n=1}^\infty$ such that $u_n \uparrow f$ and $v_n \uparrow g$. Then also $u_n+v_n \uparrow f+g$. Thus, using the property \textup{(P3)} for $\|\cdot\|_{Z(0,1)}$ (which has already been verified) and the inequality~\eqref{E:triangle}, we obtain
\begin{align*}
\|f+g\|_{Z(0,1)}  
&=\lim_{n\to \infty} \|u_n+v_n\|_{Z(0,1)}
\leq \lim_{n\to \infty} \left(\|u_n\|_{Z(0,1)} +\|v_n\|_{Z(0,1)}\right)\\
&=\lim_{n\to \infty} \|u_n\|_{Z(0,1)} +\lim_{n\to \infty} \|v_n\|_{Z(0,1)}
=\|f\|_{Z(0,1)} +\|g\|_{Z(0,1)},
\end{align*}
as required.

The assertion~\eqref{E:domain} follows from the definition of the space $Y^d_{j,J}(0,1)$. Furthermore, let $\|\cdot\|_{X(0,1)}$ be a rearrangement-invariant norm such that $H^j_J: X(0,1) \rightarrow Y(0,1)$. Then there is a constant $C>0$ such that for every $f\in X(0,1)$, 
$$
\|H^j_Jf\|_{Y(0,1)} \leq C\|f\|_{X(0,1)}.
$$
Thus,
\begin{equation*}
\|f\|_{Y^d_{j,J}(0,1)} 
= \sup_{h\sim f} \|H^j_Jf\|_{Y(0,1)} +\|f\|_{L^1(0,1)} 
\leq (C+C_X)\|f\|_{X(0,1)}.
\end{equation*}
Hence, we obtain $X(0,1) \hookrightarrow Y^d_{j,J}(0,1)$, that is, $Y^d_{j,J}(0,1)$ is the largest rearrangement-invariant space for which~\eqref{E:domain} is satisfied.

Finally, suppose that a rearrangement-invariant norm $\|\cdot\|_{Y(0,1)}$ does not fulfil~\eqref{E:cond_Y}. By~\eqref{E:char} applied with $a=1$,
$$
\|H^j_J(1)\|_{Y(0,1)} 
=\frac{1}{j!}\left\|\left(\int_t^1 \frac{\,dr}{J(r)}\right)^{j} \right\|_{Y(0,1)} =\infty,
$$
so $H^j_J(1) \notin Y(0,1)$. Since the constant function $1$ belongs to each rearrangement-invariant space over $(0,1)$, there is no rearrangement-invariant space $Y^d_{j,J}(0,1)$ for which~\eqref{E:domain} is satisfied.
\end{proof}

\begin{proof}[Proof of Theorem~\ref{T:opt_domain}]
Due to Theorem~\ref{T:compact_operator}, condition \textup{(i)} is equivalent to 
\begin{equation}\label{E:ii}
\lim_{a\to 0_+} \sup_{\|f\|_{X(0,1)}\leq 1} \|H^j_J(\chi_{(0,a)}f)\|_{Y(0,1)} =0.
\end{equation}
Obviously, we have \textup{(ii)} $\Rightarrow$ \eqref{E:ii}. Conversely, suppose that~\eqref{E:ii} holds and fix $b\in (0,1)$. Using the first part of~\eqref{E:estimate} applied to $\chi_Ef$ instead of $f$, we get
\begin{align*}
\limsup_{a\to 0_+} &\sup_{\|f\|_{X(0,1)}\leq1} \sup_{\lambda_1(E) \leq a} \left\|H^j_J(\chi_Ef)\right\|_{Y(0,1)}\\
&\leq \limsup_{a\to 0_+} \sup_{\|f\|_{X(0,1)}\leq1} \sup_{\lambda_1(E) \leq a} \left\|H^j_J(\chi_{(0,b)} \chi_E f)\right\|_{Y(0,1)}\\
&+\limsup_{a\to 0_+} \sup_{\|f\|_{X(0,1)}\leq1} \sup_{\lambda_1(E) \leq a} \left\|H^j_J(\chi_{(b,1)} \chi_Ef)\right\|_{Y(0,1)}\\
&\leq \sup_{\|f\|_{X(0,1)}\leq1} \left\|H^j_J(\chi_{(0,b)}f)\right\|_{Y(0,1)}\\
&+\frac{1}{(j-1)!J_b}\left\|\left(\int_t^1 \frac{\,dr}{J(r)}\right)^{j-1}\right\|_{Y(0,1)} \limsup_{a\to 0_+} \sup_{\|f\|_{X(0,1)}\leq 1} \sup_{\lambda_1(E)\leq a} \|\chi_Ef\|_{L^1(0,1)}\\
&=\sup_{\|f\|_{X(0,1)}\leq1} \left\|H^j_J(\chi_{(0,b)}f)\right\|_{Y(0,1)},
\end{align*}
because $X(0,1) \overset{*}{\hookrightarrow} L^1(0,1)$ (thanks to the assumption $X(0,1) \neq L^1(0,1)$) and 
\begin{align*}
&\left\|\left(\int_t^1 \frac{\,dr}{J(r)}\right)^{j-1}\right\|_{Y(0,1)}
\leq \left\|\chi_{(0,\frac{1}{2})}(t)\left(\int_t^1 \frac{\,dr}{J(r)}\right)^{j-1}\right\|_{Y(0,1)}
+\left\|\chi_{(\frac{1}{2},1)}(t)\left(\int_t^1 \frac{\,dr}{J(r)}\right)^{j-1}\right\|_{Y(0,1)}\\
&\leq \frac{1}{\int_{\frac{1}{2}}^1 \frac{\,dr}{J(r)}}\left\|\chi_{(0,\frac{1}{2})}(t)\left(\int_t^1 \frac{\,dr}{J(r)}\right)^{j}\right\|_{Y(0,1)}
+\left(\int_{\frac{1}{2}}^1 \frac{\,dr}{J(r)}\right)^{j-1}\left\|\chi_{(\frac{1}{2},1)}(t)\right\|_{Y(0,1)} <\infty,
\end{align*}
thanks to~\eqref{E:cond_Y} and to the fact that $0<\int_{\frac 12}^1 \frac{\,dr}{J(r)}<\infty$. Passing to limit when $b$ tends to $0$, we obtain \textup{(ii)}.

It remains to show that \textup{(ii)} holds if and only if \textup{(iii)} holds. Fix $a\in (0,1)$. Then
\begin{align}\label{E:ineq}
\sup_{\|f\|_{X(0,1)}\leq 1} &\sup_{\lambda_1(E) \leq a} \|H^j_J(\chi_Ef)\|_{Y(0,1)} \nonumber\\
&\leq \sup_{\|f\|_{X(0,1)}\leq 1} \sup_{\lambda_1(E) \leq a} \left( \sup_{h \sim \chi_Ef} \|H^j_Jh\|_{Y(0,1)} +\|\chi_Ef\|_{L^1(0,1)}\right)\\
&\leq \sup_{\|f\|_{X(0,1)}\leq 1} \sup_{\lambda_1(E) \leq a} \|H^j_J(\chi_Ef)\|_{Y(0,1)} 
+ \sup_{\|f\|_{X(0,1)}\leq 1} \sup_{\lambda_1(E) \leq a} \|\chi_Ef\|_{L^1(0,1)}. \nonumber
\end{align}
Note that the second inequality is true thanks to the fact that whenever a function $f$ fulfils $\|f\|_{X(0,1)}\leq 1$, a set $E\subseteq (0,1)$ satisfies $\lambda_1(E)\leq a$ and $h$ is a function equimeasurable to $\chi_{E}f$, then $h=\chi_{\{|h|>0\}}h$, $h$ belongs to the unit ball of $X(0,1)$ and $\lambda_1(\{|h|>0\})=\lambda_1(\{\chi_E|f|>0\})\leq \lambda_1(E)\leq a$. 

Assume that \textup{(ii)} holds. Since $X(0,1) \neq L^1(0,1)$, we have $X(0,1) \overset{*}{\hookrightarrow} L^1(0,1)$, that is,
$$
\lim_{a\to 0_+} \sup_{\|f\|_{X(0,1)}\leq 1} \sup_{\lambda_1(E) \leq a} \|\chi_Ef\|_{L^1(0,1)}=0.
$$
Thus, according to the second inequality in~\eqref{E:ineq}, we obtain
\begin{align*}
&\lim_{a\to 0_+} \sup_{\|f\|_{X(0,1)}\leq 1} \sup_{\lambda_1(E) \leq a} \|\chi_Ef\|_{Y^d_{j,J}(0,1)}\\
&=\lim_{a\to 0_+} \sup_{\|f\|_{X(0,1)}\leq 1} \sup_{\lambda_1(E) \leq a} \left(\sup_{h \sim \chi_Ef} \|H^j_Jh\|_{Y(0,1)} +\|\chi_Ef\|_{L^1(0,1)}\right)=0,
\end{align*}
which yields \textup{(iii)}.

Conversely, assume that \textup{(iii)} holds. Then the first inequality in~\eqref{E:ineq} yields \textup{(ii)}. The proof is complete.
\end{proof}

\section{Main results}\label{S:main}

In the present section we state and prove the main results of this paper. They concern derivation of $m$-th compact Sobolev embeddings from compactness of the one-dimensional operator $H^m_I$ defined in the previous section. Here, $I$ stands for a function which is related to the underlying measure space $(\Omega,\nu)$ by the fact the $(\Omega,\nu,I)$ is a compatible triplet (recall that the notion of a compatible triplet was introduced in Definition~\ref{T:definition}). 

\begin{theorem}\label{T:main}
Assume that $(\Omega,\nu,I)$ is a compatible triplet.
Let $m\in \mathbb N$ and let $\|\cdot\|_{X(0,1)}$ and $\|\cdot\|_{Y(0,1)}$ be rearrangement-invariant norms. Then 
\begin{equation}\label{E:comp_op}
H^m_I: X(0,1) \rightarrow \rightarrow Y(0,1)
\end{equation}
holds if and only if
\begin{equation}\label{E:almost_comp_op}
\lim_{a\to 0_+} \sup_{\|f\|_{X(0,1)} \leq 1} \|H^m_I(\chi_{(0,a)}f)\|_{Y(0,1)} =0
\end{equation}
holds. Moreover, each of the conditions~\eqref{E:comp_op} and~\eqref{E:almost_comp_op} implies
\begin{equation}\label{E:comp_emb}
V^mX(\Omega,\nu) \hookrightarrow \hookrightarrow Y(\Omega,\nu).
\end{equation}
\end{theorem}

Let us remark that further characterization of~\eqref{E:comp_op} and~\eqref{E:almost_comp_op} can be obtained by applying Theorems~\ref{T:opt_range} and~\ref{T:opt_domain} with $J=I$ and $j=m$.

\begin{remark}\label{T:monotone}
It turns out that if we take the supremum in~\eqref{E:almost_comp_op} over the smaller set of all nonincreasing functions belonging to the unit ball of $X(0,1)$, we do not change the validity of~\eqref{E:almost_comp_op}.
In other words, \eqref{E:almost_comp_op} holds if and only if
\begin{equation}\label{E:almost_comp_op_monotone}
\lim_{a\to 0_+} \sup_{\|f\|_{X(0,1)}\leq 1} \|H^m_I(\chi_{(0,a)}f^*)\|_{Y(0,1)}=0.
\end{equation}
This claim can be proved by methods of~\cite[Section 9]{CPS}. Namely, we first observe that, for every $a\in (0,1]$ and $f\in \mathcal M(0,1)$,
\begin{align}\label{E:inequalities}
\left\|\chi_{(0,a)}(t) R^m_If^*(t)\right\|_{X'(0,1)} 
&\leq \left\|\chi_{(0,a)}(t) \sup_{t\leq s \leq a} R^m_If^*(s)\right\|_{X'(0,1)}
\leq 2^{m+1} \left\|\chi_{(0,a)}(t) R^m_If^*(t)\right\|_{X_d'(0,1)} \\
&\leq 2^{m+1} \left\|\chi_{(0,a)}(t) R^m_If^*(t)\right\|_{X'(0,1)}, \nonumber
\end{align}
where the functional $\|\cdot\|_{X'_d(0,1)}$ is defined by
$$
\|f\|_{X'_d(0,1)}=\sup_{\|g\|_{X(0,1)}} \int_0^1 g^*(s)|f(s)|\,ds, ~~f\in \mathcal M(0,1).
$$
Then it suffices to show that for every $a\in (0,1]$, 
\begin{equation}\label{E:asoc}
\sup_{\|f\|_{X(0,1)}\leq 1} \|H^m_I(\chi_{(0,a)} f)\|_{Y(0,1)} = \sup_{\|f\|_{Y'(0,1)}\leq 1} \|\chi_{(0,a)} R^m_I f^*\|_{X'(0,1)}
\end{equation}
and
\begin{equation}\label{E:asoc_down}
\sup_{\|f\|_{X(0,1)}\leq 1} \|H^m_I(\chi_{(0,a)} f^*)\|_{Y(0,1)} = \sup_{\|f\|_{Y'(0,1)}\leq 1} \|\chi_{(0,a)} R^m_I f^*\|_{X_d'(0,1)}.
\end{equation}
We note that the only nontrivial inequality in~\eqref{E:inequalities} is the second one, which was proved in~\cite[Theorem 9.5]{CPS} in the case when $a=1$. Equalities~\eqref{E:asoc} and~\eqref{E:asoc_down} were proved for $a=1$ in~\cite[Corollary 9.8]{CPS}. All the proofs can be easily extended also to general $a\in (0,1]$.

\end{remark}

Suppose that $I: (0,1] \rightarrow (0,\infty)$ is a nondecreasing function satisfying~\eqref{E:infimum} and let $m\in \mathbb N$. Set
\begin{equation}\label{E:J}
J(t)= 
\frac{(I(t))^m}{t^{m-1}}, ~~t\in (0,1].
\end{equation}
We observe that $J$ is measurable on $(0,1]$ and fulfils~\eqref{E:supremum}. We can therefore consider operators $K^m_I$ and $S^m_I$ defined by $K^m_I=H_J$ and $S^m_I=R_J$, respectively. Then
$$
K^m_I f(t)= \int_t^1 |f(s)| \frac{s^{m-1}}{(I(s))^m} \,ds, ~~f\in \mathcal M(0,1), ~~t\in (0,1),
$$
and 
$$
S^m_I f(t)= \frac{t^{m-1}}{(I(t))^m} \int_0^t |f(s)|\,ds, ~~f\in \mathcal M(0,1),~~t\in (0,1).
$$

Although it is of use to define the operators $K^m_I$ and $S^m_I$ for all functions $I$ with the properties stated above (see, e.g., Theorem~\ref{T:ll}), these operators come into play especially in the case when $I$ satisfies
\begin{equation}\label{E:aprox}
\int_0^s \frac{\,dr}{I(r)} \approx \frac{s}{I(s)}, ~~ s\in (0,1),
\end{equation}
up to multiplicative constants possibly depending on $I$. In this situation, conditions~\eqref{E:comp_op} and~\eqref{E:almost_comp_op} can be equivalently reformulated using the rather simple operator $K^m_I$ 
instead of the kernel integral operator $H^m_I$. The corresponding result is the following theorem. Its proof strongly depends on a result proved in~\cite{CPS} which relates boundedness of $H^m_I$ to boundedness of $K^m_I$.

\begin{theorem}\label{T:K}
Assume that $(\Omega,\nu, I)$ is a compatible triplet and that~\eqref{E:aprox} is satisfied. Let $m\in \mathbb N$ and let $\|\cdot\|_{X(0,1)}$ and $\|\cdot\|_{Y(0,1)}$ be rearrangement-invariant norms. 

\medskip
\noindent
\textup{(a)} Suppose that 
\begin{equation}\label{E:lim_not_0}
\lim_{t\to 0_+} \frac{t^{m-1}}{(I(t))^m} \neq 0.
\end{equation}
Then
\begin{equation}\label{E:comp_op_K}
K^m_I: X(0,1) \rightarrow \rightarrow Y(0,1)
\end{equation}
holds if and only if
\begin{equation}\label{E:almost_comp_op_K}
\lim_{a\to 0_+} \sup_{\|f\|_{X(0,1)} \leq 1} \|K^m_I(\chi_{(0,a)}f)\|_{Y(0,1)} =0
\end{equation}
holds. Moreover, each of the conditions~\eqref{E:comp_op_K} and~\eqref{E:almost_comp_op_K} implies
\begin{equation}\label{E:comp_emb_K}
V^mX(\Omega,\nu) \hookrightarrow \hookrightarrow Y(\Omega,\nu).
\end{equation}

\medskip
\noindent
\textup{(b)} Suppose that
\begin{equation}\label{E:lim0}
\lim_{t\to 0_+} \frac{t^{m-1}}{(I(t))^m} = 0.
\end{equation}
Then~\eqref{E:comp_emb_K} is satisfied for all pairs of rearrangement-invariant norms $\|\cdot\|_{X(0,1)}$ and $\|\cdot\|_{Y(0,1)}$. 
\end{theorem}

Analogously to the general case, which we dealt with in Theorem~\ref{T:main}, one can obtain further characterization of~\eqref{E:comp_op_K} and~\eqref{E:almost_comp_op_K} by applying Theorems~\ref{T:opt_range} and~\ref{T:opt_domain} with $J$ as in~\eqref{E:J} and $j=1$. Notice that in this situation,
$$
\int_0^1 \frac{\,dr}{J(r)} = \int_0^1 \frac{r^{m-1}}{(I(r))^m} \,dr \leq \left(\sup_{r\in (0,1]} \frac{r}{I(r)}\right)^{m-1} \int_0^1 \frac{\,dr}{I(r)} \approx \frac{1}{\left(\inf_{r\in (0,1]} \frac{I(r)}{r}\right)^{m-1} I(1)} <\infty,
$$
thanks to~\eqref{E:aprox} and~\eqref{E:infimum}. Therefore, \eqref{E:cond_Y} is satisfied for all rearrangement-invariant norms $\|\cdot\|_{Y(0,1)}$, since
$$
\left\|\left(\int_t^1 \frac{\,dr}{J(r)}\right)^m \right\|_{Y(0,1)} \leq \left(\int_0^1 \frac{\,dr}{J(r)} \right)^m \|1\|_{Y(0,1)} <\infty.
$$

\begin{remarks}\label{T:remarks_5}
\textup{(i)} If $I:(0,1] \rightarrow (0,\infty)$ is a nondecreasing function satisfying~\eqref{E:infimum} and $m\in \mathbb N$ is such that~\eqref{E:lim0} is fulfilled, then it will follow from the proof of Theorem~\ref{T:K} that~\eqref{E:almost_comp_op_K} is satisfied for all pairs of rearrangement-invariant norms $\|\cdot\|_{X(0,1)}$ and $\|\cdot\|_{Y(0,1)}$. However, Theorem~\ref{T:compact_operator} applied with $J$ as in~\eqref{E:J} and $j=1$ yields that~\eqref{E:comp_op_K} is not satisfied if $X(0,1)=L^1(0,1)$ and $Y(0,1)=L^\infty(0,1)$. Therefore, in contrast to the part \textup{(a)}, in the part \textup{(b)} we do not have the equivalence of~\eqref{E:comp_op_K} and~\eqref{E:almost_comp_op_K}. Moreover, compactness of the operator $K^m_I$ seems not to be appropriate to characterize compact Sobolev embeddings in this case, and condition~\eqref{E:almost_comp_op_K} turns out to be a suitable substitute for~\eqref{E:comp_op_K}.

\textup{(ii)} Notice that to prove the equivalence of~\eqref{E:comp_op_K} and~\eqref{E:almost_comp_op_K} in the part \textup{(a)}, we do not need to assume that $I$ satisfies~\eqref{E:aprox}.

\textup{(iii)} The assumption~\eqref{E:aprox} is also not necessary for the validity of the part \textup{(b)} of Theorem~\ref{T:K}. Indeed, Theorem~\ref{T:ll}, which is stated and proved in Section~\ref{S:examples}, yields that for any nondecreasing function $I$ on $(0,1]$ satisfying~\eqref{E:infimum}, condition~\eqref{E:lim0} implies
$$
\lim_{a\to 0_+} \sup_{\|f\|_{L^1(0,1)}\leq 1} \|H^m_I(\chi_{(0,a)}f)\|_{L^\infty(0,1)} =0.
$$
Then, according to Theorem~\ref{T:main}, we get $V^mL^1(\Omega,\nu) \hookrightarrow \hookrightarrow L^\infty(\Omega,\nu)$. Thanks to embeddings $X(\Omega,\nu) \hookrightarrow L^1(\Omega,\nu)$ and $L^\infty(\Omega,\nu) \hookrightarrow Y(\Omega,\nu)$, which hold for all rearrangement-invariant norms $\|\cdot\|_{X(0,1)}$ and $\|\cdot\|_{Y(0,1)}$, we obtain~\eqref{E:comp_emb_K}. 

\textup{(iv)} The assumption~\eqref{E:aprox} is essential for the proof that~\eqref{E:comp_op_K} (or~\eqref{E:almost_comp_op_K}) implies~\eqref{E:comp_emb_K} in the part \textup{(a)}. Suppose that $I$ is the function defined by $I(t)=t\sqrt{\log \frac{2}{t}}$, $t\in (0,1]$. Then it follows from the observations made in Section~\ref{S:sobolev} that
$(\mathbb R^n, \gamma_n, I)$ is a compatible triplet. Furthermore, notice that $I$ satisfies~\eqref{E:lim_not_0} for every $m\in \mathbb N$ but does not satisfy~\eqref{E:aprox}. Moreover, if $m>2$ then
\begin{align*}
\lim_{a\to 0_+} \sup_{\|f\|_{L^\infty(0,1)}\leq 1} \|K^m_I(\chi_{(0,a)}f)\|_{L^\infty(0,1)} 
&=\lim_{a\to 0_+} \sup_{\|f\|_{L^\infty(0,1)}\leq 1} \int_0^a \frac{|f(s)|}{s(\log \frac{2}{s})^{\frac{m}{2}}}\,ds\\
&=\lim_{a\to 0_+} \int_0^a \frac{\,ds}{s(\log \frac{2}{s})^{\frac{m}{2}}} = 0.
\end{align*}
Hence, \eqref{E:almost_comp_op_K} is satisfied with $\|\cdot\|_{X(0,1)} = \|\cdot\|_{Y(0,1)}= \|\cdot\|_{L^\infty(0,1)}$. However, \eqref{E:comp_emb_K} is not fulfilled with $(\Omega,\nu)=(\mathbb R^n, \gamma_n)$ in this case, since even the continuous embedding $V^mL^\infty(\mathbb R^n,\gamma_n) \hookrightarrow L^\infty(\mathbb R^n,\gamma_n)$ does not hold (see~\cite[Theorem 7.13]{CPS}).
\end{remarks}

The remaining part of this section is devoted to proofs of Theorems~\ref{T:main} and~\ref{T:K}. We start with an auxiliary result which shows that, in our setting, the unit ball of each Sobolev space is compact in measure.

\begin{lemma}\label{T:lemma}
Assume that $(\Omega,\nu)$ is as in Section~\ref{S:sobolev}. Let $m\in \mathbb N$ and let $\|\cdot\|_{X(0,1)}$ be a rearrangement-invariant norm. Then every sequence $(u_k)_{k=1}^\infty$ bounded in $V^mX(\Omega,\nu)$ contains a subsequence $(u_{k_\ell})_{\ell=1}^\infty$ converging $\nu$-a.e.\ on $\Omega$. In particular, the subsequence $(u_{k_\ell})_{\ell=1}^\infty$ is convergent in measure.
\end{lemma}

\begin{proof}
For a.e.\ $x\in \Omega$ we can find an open ball $B_x$ centered in $x$ such that $B_x \subseteq \Omega$ and $\einf_{B_x} \omega>0$. Denote by $N$ the set of points in $\Omega$ for which such a ball does not exist. Then $\nu(N)=0$ and we have $\Omega\setminus N \subseteq \bigcup_{x\in \Omega \setminus N} B_x$. Due to the separability of $\Omega \setminus N$, there is a sequence $(x_j)_{j=1}^\infty$ of points in $\Omega \setminus N$ such that $\Omega \setminus N \subseteq \bigcup_{j=1}^\infty B_{x_j}$. Since the sequence $(u_k)_{k=1}^\infty$ is bounded in $V^mX(\Omega,\nu)$, it is also bounded in $V^1L^1(\Omega,\nu)$. Hence, for every $j\in \mathbb N$ and $k\in \mathbb N$ we have
\begin{align*}
\|u_k\|_{V^1L^1(\Omega,\nu)} 
&\geq \int_{B_{x_j}} \left(|u_k(x)|+|\nabla u_k(x)|\right)\omega(x)\,dx \\
&\geq \left(\einf_{B_{x_j}} \omega \right) \int_{B_{x_j}} (|u_k|+|\nabla u_k|)\,dx = \left(\einf_{B_{x_j}} \omega \right)  \|u_k\|_{V^1L^1(B_{x_j})}.
\end{align*}
Therefore, $(u_k)_{k=1}^\infty$ is bounded in $V^1L^1(B_{x_j})$. Denote $u_k^0=u_k$, $k\in \mathbb N$. By induction, for every $j\in \mathbb N$ we will construct a subsequence $(u_k^j)_{k=1}^\infty$ of the sequence $(u_k^{j-1})_{k=1}^\infty$ converging a.e.\ on $B_{x_j}$. Suppose that, for some $j\in \mathbb N$, we have already found the sequence $(u_k^{j-1})_{k=1}^\infty$. Since $(u_k^{j-1})_{k=1}^\infty$ is bounded in $V^1L^1(B_{x_j})$ and the compact embedding $V^1L^1(B_{x_j}) \hookrightarrow \hookrightarrow L^1(B_{x_j})$ holds, we can find a subsequence $(u_k^j)_{k=1}^\infty$ of $(u_k^{j-1})_{k=1}^\infty$ converging in $L^1(B_{x_j})$. Passing, if necessary, to another subsequence, $(u_k^j)_{k=1}^\infty$ can be found in such a way that it converges a.e.\ on $B_{x_j}$. Now, the diagonal sequence $(u_k^k)_{k=1}^\infty$ converges a.e.\ (or, what is the same, $\nu$-a.e.) on $\bigcup_{j=1}^\infty B_{x_j}=\Omega \setminus N$. Since $\nu(N)=0$, $(u_k^k)_{k=1}^\infty$ converges $\nu$-a.e.\ on $\Omega$, as required. Furthermore, it is a well known fact that each sequence converging $\nu$-a.e. is convergent in measure.
\end{proof}

We also need the following

\begin{lemma}\label{T:lemma_infty}
Let $I: (0,1] \rightarrow (0,\infty)$ be a nondecreasing function satisfying~\eqref{E:infimum} and 
\begin{equation}\label{E:cond_conv}
\int_0^1 \frac{\,ds}{I(s)} <\infty. 
\end{equation}
Then
$$
\|f\|_{(L^\infty)^d_{1,I}(0,1)} \approx \int_0^1 \frac{f^*(s)}{I(s)} \,ds, ~~f\in \mathcal M(0,1),
$$
up to multiplicative constants depending on $I$. 
\end{lemma}

\begin{proof}
We first observe that condition~\eqref{E:cond_Y} is fulfilled with $j=1$, $J=I$ and $\|\cdot\|_{Y(0,1)}=\|\cdot\|_{L^\infty(0,1)}$. Indeed, we have
$$
\left\|\int_t^1 \frac{\,dr}{I(r)}\right\|_{L^\infty(0,1)} = \int_0^1 \frac{\,dr}{I(r)} <\infty,
$$
thanks to~\eqref{E:cond_conv}. The rearrangement-invariant norm $\|\cdot\|_{(L^\infty)^d_{1,I}(0,1)}$ is therefore well defined.

Let $f\in \mathcal M(0,1)$. We have
\begin{equation*}
\|f\|_{(L^\infty)^d_{1,I}(0,1)} = \sup_{0\leq h \sim f} \left\|\int_t^1 \frac{h(s)}{I(s)} \,ds\right\|_{L^\infty(0,1)} +\|f\|_{L^1(0,1)}
=\sup_{0\leq h \sim f} \int_0^1 \frac{h(s)}{I(s)}\,ds + \|f\|_{L^1(0,1)}.
\end{equation*}
Since $f^*$ is a nonnegative function equimeasurable to $f$, 
\begin{equation*}
\int_0^1 \frac{f^*(s)}{I(s)} \,ds \leq \sup_{0\leq h \sim f} \int_0^1 \frac{h(s)}{I(s)}\,ds + \|f\|_{L^1(0,1)}.
\end{equation*}
Conversely, using the Hardy-Littlewood inequality and the monotonicity of $I$ we obtain
\begin{equation*}
\sup_{0\leq h \sim f} \int_0^1 \frac{h(s)}{I(s)}\,ds + \|f\|_{L^1(0,1)} 
\leq \int_0^1 \frac{f^*(s)}{I(s)} \,ds +\int_0^1 f^*(s)\,ds 
\leq \left(1+I(1)\right)\int_0^1 \frac{f^*(s)}{I(s)} \,ds.
\end{equation*}
The proof is complete.
\end{proof}

\begin{proof}[Proof of Theorem~\ref{T:main}]
Since $I$ is nondecreasing on $(0,1]$, we have
$$
\esup_{t\in (0,a)} \frac{1}{I(t)}=\lim_{t\to 0_+} \frac{1}{I(t)}, ~~a\in (0,1).
$$
Thus,
\begin{equation}\label{E:limit}
\lim_{a\to 0_+} \esup_{t\in (0,a)} \frac{1}{I(t)} = \lim_{t\to 0_+} \frac{1}{I(t)} \neq 0.
\end{equation}
Theorem~\ref{T:compact_operator} now gives that~\eqref{E:comp_op} holds if and only if~\eqref{E:almost_comp_op} holds. 

Suppose that~\eqref{E:comp_op} (or, equivalently, \eqref{E:almost_comp_op}) is satisfied. Moreover, assume that 
$$
Y(0,1)\neq L^\infty(0,1) \hbox{  or } \int_0^1 \frac{\,ds}{I(s)}=\infty.
$$
Then, due to Theorem~\ref{T:opt_range}, we have $X^r_{m,I}(0,1) \overset{*}{\hookrightarrow} Y(0,1)$, or, what is the same, $X^r_{m,I}(\Omega,\nu) \overset{*}{\hookrightarrow} Y(\Omega,\nu)$.

Assume that $(u_k)_{k=1}^\infty$ is a sequence bounded in $V^mX(\Omega, \nu)$. Due to Lemma~\ref{T:lemma}, we can find its subsequence $(u_{k_\ell})_{\ell=1}^\infty$ which converges to some function $u$ $\nu$-a.e.\ on $\Omega$. 
Since $H^m_I: X(0,1) \rightarrow X^r_{m,I}(0,1)$, Theorem~\ref{T:CPS} implies that $V^mX(\Omega, \nu) \hookrightarrow X^r_{m,I}(\Omega, \nu)$. Hence, $(u_{k_\ell})_{\ell=1}^\infty$ is bounded in $X^r_{m,I}(\Omega, \nu)$. By the Fatou lemma, 
$$
\|u\|_{X^r_{m,I}(\Omega, \nu)} \leq \liminf_{\ell\to \infty} \|u_{k_\ell}\|_{X^r_{m,I}(\Omega, \nu)}<\infty,
$$ 
so $u\in X^r_{m,I}(\Omega,\nu)$ and $(u_{k_\ell}-u)_{\ell=1}^\infty$ is therefore bounded in $X^r_{m,I}(\Omega,\nu)$ as well. We have $X^r_{m,I}(\Omega,\nu) \overset{*}{\hookrightarrow} Y(\Omega, \nu)$, so, according to~\cite[Theorem 3.1]{S}, $(u_{k_\ell}-u) \rightarrow 0$ in $Y(\Omega,\nu)$, i.e., $u_{k_\ell} \rightarrow u$ in $Y(\Omega,\nu)$. Thus, $V^mX(\Omega, \nu)\hookrightarrow \hookrightarrow Y(\Omega, \nu)$.

Conversely, assume that $Y(0,1)=L^\infty(0,1)$ and $\int_0^1 1/I(s)\,ds<\infty$ (recall that the assumption~\eqref{E:comp_op} is still in progress). We start with the case when $m=1$. The proof of Lemma~\ref{T:lemma_infty} then yields that condition~\eqref{E:cond_Y} is fulfilled with $J=I$ and $j=1$. Furthermore, since the operator $H_I$ is not compact from $L^1(0,1)$ into $L^\infty(0,1)$ (see the last part of the proof of Theorem~\ref{T:compact_operator}), we have $X(0,1) \neq L^1(0,1)$. Thus, due to Theorem~\ref{T:opt_domain},
\begin{equation}\label{E:almost_compact_3}
X(0,1) \overset{*}{\hookrightarrow} (L^\infty)^d_{1,I}(0,1).
\end{equation}
Furthermore, Proposition~\ref{T:proposition} combined with Theorem~\ref{T:CPS} yield that 
\begin{equation}\label{E:embedding}
V^1(L^\infty)^d_{1,I}(\Omega,\nu) \hookrightarrow L^\infty(\Omega,\nu).
\end{equation}
Since $\int_0^1 1/I(s)\,ds<\infty$, we obtain by applying Proposition~\ref{T:propprop} that 
\begin{equation}\label{E:VW}
V^1(L^\infty)^d_{1,I}(\Omega,\nu) = W^1(L^\infty)^d_{1,I}(\Omega,\nu) \textup{   and   } V^1X(\Omega,\nu) = W^1X(\Omega,\nu),
\end{equation} 
up to equivalent norms.

Let $(u_k)_{k=1}^\infty$ be a bounded sequence in $V^1X(\Omega,\nu)$. Then it is bounded also in $W^1X(\Omega,\nu)$. Without loss of generality we may assume that 
\begin{equation}\label{E:norm}
\|u_k\|_{W^1X(\Omega,\nu)}\leq 1, ~~ k\in \mathbb N.
\end{equation} 
Due to Lemma~\ref{T:lemma}, there is a subsequence $(v_k)_{k=1}^\infty$ of the sequence $(u_k)_{k=1}^\infty$ which converges in measure to some function $v$. Our aim is to show that $(v_k)_{k=1}^\infty$ is a Cauchy sequence in $L^\infty(\Omega,\nu)$. Then, thanks to the completeness of $L^\infty(\Omega,\nu)$, $(v_k)_{k=1}^\infty$ will converge to $v$ in the norm of the space $L^\infty(\Omega,\nu)$. This will prove that $V^1X(\Omega,\nu)$ is compactly embedded into $L^\infty(\Omega,\nu)$. 

Fix $\varepsilon>0$ and observe that for all $k$, $l\in \mathbb N$ we have
\begin{equation}\label{E:max_min}
|v_k-v_\ell| = \min\{|v_k-v_\ell|,\varepsilon/2\} + \max\{|v_k-v_\ell|-\varepsilon/2,0\}.
\end{equation}
Since $v_k$, $v_l$ and the constant function $\varepsilon/2$ are weakly differentiable on $\Omega$, $|v_k-v_\ell|-\varepsilon/2$ is weakly differentiable on $\Omega$ as well and
$$
\nabla (|v_k-v_\ell|-\varepsilon/2)= \nabla |v_k-v_\ell| = \sgn (v_k-v_\ell) \nabla (v_k-v_l) = \sgn (v_k-v_\ell) (\nabla v_k -\nabla v_\ell)
$$
a.e.\ on $\Omega$. Furthermore, $\max\{|v_k-v_\ell|-\varepsilon/2,0\}$ is weakly differentiable on $\Omega$ and
\begin{align*}
\nabla \max&\{|v_k-v_\ell|-\varepsilon/2,0\} \\&= 
\begin{cases}
\sgn (v_k-v_\ell) (\nabla v_k -\nabla v_\ell)\, & \hbox{a.e.\ on }\{x\in \Omega: |v_k-v_\ell|>\varepsilon/2\},\\
0\, & \hbox{a.e.\ on }\{x\in \Omega: |v_k-v_\ell|\leq \varepsilon/2\},
\end{cases}
\end{align*}
i.e.,
\begin{equation}\label{E:equation}
\nabla \max\{|v_k-v_\ell|-\varepsilon/2,0\} = \chi_{\{ |v_k-v_\ell|>\varepsilon/2\}} \sgn (v_k-v_\ell) (\nabla v_k -\nabla v_\ell)
\end{equation}
a.e.\ on $\Omega$. Thus,
\begin{equation}\label{E:equation_2}
|\nabla \max\{|v_k-v_\ell|-\varepsilon/2,0\}| = \chi_{\{ |v_k-v_\ell|>\varepsilon/2\}} |\nabla v_k -\nabla v_\ell|
\end{equation}
a.e.\ on $\Omega$ (and therefore also $\nu$-a.e.\ on $\Omega$, since $\nu$ is absolutely continuous with respect to the $n$-dimensional Lebesgue measure).

We have
\begin{align}\label{E:cauchy}
&\|v_k-v_\ell\|_{L^\infty(\Omega,\nu)} \nonumber \\ \nonumber
&\leq \|\min\{|v_k-v_\ell|,\varepsilon/2\}\|_{L^\infty(\Omega,\nu)} + \|\max\{|v_k-v_\ell|-\varepsilon/2,0\}\|_{L^\infty(\Omega,\nu)} ~~~ \textup{(by~\eqref{E:max_min})}\\ \nonumber
&\leq \varepsilon/2 + C\|\max\{|v_k-v_\ell|-\varepsilon/2,0\}\|_{W^1(L^\infty)^d_{1,I}(\Omega,\nu)} ~~~ \textup{(by~\eqref{E:embedding} and~\eqref{E:VW})}\\ \nonumber
&=\varepsilon/2 + C\|\chi_{\{|v_k-v_\ell|>\varepsilon/2\}} |\nabla v_k- \nabla v_\ell|\|_{(L^\infty)^d_{1,I}(\Omega,\nu)}\\
&+C\|\chi_{\{|v_k-v_\ell|>\varepsilon/2\}} (|v_k- v_\ell|-\varepsilon/2)\|_{(L^\infty)^d_{1,I}(\Omega,\nu)} ~~~~~~~~~~~~ \textup{(by~\eqref{E:equation_2})}\\ \nonumber
&\leq \varepsilon/2 + C\|\chi_{\{|v_k-v_\ell|>\varepsilon/2\}} |\nabla v_k|\|_{(L^\infty)^d_{1,I}(\Omega,\nu)} + C\|\chi_{\{|v_k-v_\ell|>\varepsilon/2\}} |\nabla v_\ell|\|_{(L^\infty)^d_{1,I}(\Omega,\nu)}\\ \nonumber
&+C\|\chi_{\{|v_k-v_\ell|>\varepsilon/2\}} v_k\|_{(L^\infty)^d_{1,I}(\Omega,\nu)} + C\|\chi_{\{|v_k-v_\ell|>\varepsilon/2\}} v_l\|_{(L^\infty)^d_{1,I}(\Omega,\nu)}\\ \nonumber
&=\varepsilon/2 + C\|(\chi_{\{|v_k-v_\ell|>\varepsilon/2\}}|\nabla v_k|)^*_\nu\|_{(L^\infty)^d_{1,I}(0,1)} + C\|(\chi_{\{|v_k-v_\ell|>\varepsilon/2\}} |\nabla v_\ell|)^*_\nu\|_{(L^\infty)^d_{1,I}(0,1)}\\ \nonumber
&+C\|(\chi_{\{|v_k-v_\ell|>\varepsilon/2\}} v_k)^*_\nu\|_{(L^\infty)^d_{1,I}(0,1)} + C\|(\chi_{\{|v_k-v_\ell|>\varepsilon/2\}} v_l)^*_\nu\|_{(L^\infty)^d_{1,I}(0,1)}\\ \nonumber
&\leq \varepsilon/2 + 4C\sup_{\|f\|_{X(0,1)}\leq 1} \|\chi_{(0,\nu(\{|v_k-v_\ell|>\varepsilon/2\}))}f^*\|_{(L^\infty)^d_{1,I}(0,1)} ~~~ \textup{(by~\eqref{E:norm})},
\end{align}
where $C>0$ is the constant from the embedding $W^1(L^\infty)^d_{1,I}(\Omega,\nu) \hookrightarrow L^\infty(\Omega,\nu)$.

Thanks to~\eqref{E:almost_compact_3}, there is $\delta>0$ such that 
\begin{equation}\label{E:epsilon_delta}
\sup_{\|f\|_{X(0,1)} \leq 1} \|\chi_{(0,\delta)}f^*\|_{(L^\infty)^d_{1,I}(0,1)}< \frac{\varepsilon}{8C}.
\end{equation}
Since $(v_k)_{k=1}^\infty$ converges in measure to $v$, we can find $k_0 \in \mathbb N$ such that for every $k\geq k_0$
$$
\nu(\{x\in \Omega: |v_k(x)-v(x)|>\varepsilon/4 \}) <\delta/2.
$$
We observe that for all $k$, $\ell \geq k_0$,
\begin{align*}
&\{x\in \Omega: |v_k(x)-v_\ell(x)|>\varepsilon/2 \}\\
&\subseteq \{x\in \Omega: |v_k(x)-v(x)|>\varepsilon/4 \} \cup \{x\in \Omega: |v_\ell(x)-v(x)|>\varepsilon/4 \}.
\end{align*}
Therefore,
\begin{align}\label{E:nu} 
&\nu(\{x\in \Omega: |v_k(x)-v_\ell(x)|>\varepsilon/2 \})\\ 
&\leq \nu(\{x\in \Omega: |v_k(x)-v(x)|>\varepsilon/4 \}) + \nu(\{x\in \Omega: |v_\ell(x)-v(x)|>\varepsilon/4 \})<\delta.  \nonumber
\end{align}
Consequently, by~\eqref{E:cauchy} and~\eqref{E:epsilon_delta},
$$
\|v_k - v_\ell\|_{L^\infty(\Omega,\nu)} \leq \varepsilon/2 + 4C \sup_{\|f\|_{X(0,1)}\leq 1} \|\chi_{(0,\delta)}f^*\|_{(L^\infty)^d_{1,I}(0,1)} <\varepsilon.
$$
Hence, $(v_k)_{k=1}^\infty$ is a Cauchy sequence in $L^\infty(\Omega,\nu)$, as required.

Finally, assume that $m>1$. According to Lemma~\ref{T:lemma_infty}, for every $g\in \mathcal M(0,1)$ we have
$$
\|g\|_{(L^\infty)^d_{1,I}(0,1)} \approx \int_0^1 \frac{g^*(s)}{I(s)}\,ds= \|H_Ig^*\|_{L^\infty(0,1)},
$$
up to multiplicative constants depending on $I$. Thus, whenever $f\in \mathcal M(0,1)$ and $a\in (0,1)$, then
\begin{align*}
\|H^{m}_I (\chi_{(0,a)}f)\|_{L^\infty(0,1)}
&=\|H_I (H^{m-1}_I (\chi_{(0,a)}f))\|_{L^\infty(0,1)}\\
&=\|H_I (H^{m-1}_I (\chi_{(0,a)}f))^*\|_{L^\infty(0,1)}
\approx\|H^{m-1}_I (\chi_{(0,a)}f)\|_{(L^\infty)^d_{1,I}(0,1)},
\end{align*}
up to multiplicative constants depending on $I$. Assumption~\eqref{E:almost_comp_op} is therefore equivalent to
\begin{equation}\label{E:ac_m-1}
\lim_{a\to 0_+} \sup_{\|f\|_{X(0,1)}\leq 1} \|H^{m-1}_I (\chi_{(0,a)}f)\|_{(L^\infty)^d_{1,I}(0,1)}=0.
\end{equation}
Since
$$
\lim_{a\to 0_+} \|\chi_{(0,a)}\|_{(L^\infty)^d_{1,I}(0,1)} \approx \lim_{a\to 0_+} \int_0^a \frac{\,ds}{I(s)} = 0,
$$
where the equivalence holds up to multiplicative constants depending on $I$, we obtain that $(L^\infty)^d_{1,I}(0,1) \neq L^\infty(0,1)$. The first part of the proof now implies that
\begin{equation}\label{E:compact_embedding_0}
V^{m-1}X(\Omega,\nu) \hookrightarrow \hookrightarrow (L^\infty)^d_{1,I}(\Omega,\nu).
\end{equation} 

Let $(u_k)_{k=1}^\infty$ be a bounded sequence in $V^{m}X(\Omega,\nu)$. Then $(u_k)_{k=1}^\infty$ is bounded in $L^1(\Omega,\nu)$, so $(\int_{\Omega} u_k\,d\nu)_{k=1}^\infty$ is a bounded sequence of real numbers and we can find a subsequence $(u_k^0)_{k=1}^\infty$ of $(u_k)_{k=1}^\infty$ such that the sequence $(\int_{\Omega} u_k^0\,d\nu)_{k=1}^\infty$ is convergent.

For $i=1,2,\dots, n$, consider the sequence $(D_iu_k^0)_{k=1}^\infty$ consisting of weak derivatives with respect to the $i$-th variable of elements of the sequence $(u_k^0)_{k=1}^\infty$. Owing to the boundedness of $(u_k)_{k=1}^\infty$ in $V^mX(\Omega,\nu)$, all these sequences are bounded in $V^{m-1}X(\Omega,\nu)$. Now, the compact embedding ~\eqref{E:compact_embedding_0} yields that we can inductively find sequences $(u_k^i)_{k=1}^\infty$, $i=1,2,\dots,n$, such that $(u_k^i)_{k=1}^\infty$ is a subsequence of $(u_k^{i-1})_{k=1}^\infty$ fulfilling that $(D_iu_k^i)_{k=1}^\infty$ is convergent in $(L^\infty)^d_{1,I}(\Omega,\nu)$. Since a subsequence of a convergent sequence is still convergent, we have, in particular, that $(D_ju_k^n)_{k=1}^\infty$ is a Cauchy sequence in $(L^\infty)^d_{1,I}(\Omega,\nu)$ for every $j\in \{1,2,\dots,n\}$. 

Let $\varepsilon>0$. By Theorem~\ref{T:CPS}, the embedding $V^1(L^\infty)^d_{1,I}(\Omega,\nu) \hookrightarrow L^\infty(\Omega,\nu)$ is equivalent to a Poincar\'e inequality. Hence, there is a constant $D>0$ such that for every $u\in V^1(L^\infty)^d_{1,I}(\Omega,\nu)$,
\begin{equation}\label{E:c_epsilon}
\left\|u - \int_{\Omega} u\,d\nu \right\|_{L^\infty(\Omega,\nu)} 
\leq D \left\||\nabla u| \right\|_{(L^\infty)^d_{1,I}(\Omega,\nu)} 
\leq D\sum_{j=1}^n \left\|D_j u \right\|_{(L^\infty)^d_{1,I}(\Omega,\nu)}.
\end{equation}
Since $(D_ju_k^n)_{k=1}^\infty$ is a Cauchy sequence in $(L^\infty)^d_{1,I}(\Omega,\nu)$ for every $j\in \{1,2,\dots,n\}$, we can find $k_0\in \mathbb N$ such that $\|D_ju_k^n-D_ju_\ell^n\|_{(L^\infty)^d_{1,I}(\Omega,\nu)} <\varepsilon/Dn$ whenever $k,\ell\geq k_0$ and $j\in \{1,2,\dots,n\}$. Thus, inequality~\eqref{E:c_epsilon} applied to $u=u_k^n-u_\ell^n$ implies that for every $k,\ell \geq k_0$,
$$
\left\|u_k^n -u_\ell^n - \int_{\Omega} (u_k^n - u_\ell^n)\,d\nu \right\|_{L^\infty(\Omega,\nu)} 
\leq D\sum_{j=1}^n \left\|D_ju_k^n-D_ju_\ell^n\right\|_{(L^\infty)^d_{1,I}(\Omega,\nu)} <\varepsilon,
$$
so, $(u_k^n- \int_{\Omega} u_k^n\,d\nu)_{k=1}^\infty$ is a Cauchy sequence in $L^\infty(\Omega,\nu)$. Due to the completeness of $L^\infty(\Omega,\nu)$, $(u_k^n- \int_{\Omega} u_k^n\,d\nu)_{k=1}^\infty$ is convergent in $L^\infty(\Omega,\nu)$. Since the sequence $(\int_{\Omega} u_k^n\,d\nu)_{k=1}^\infty$ consisting of constant functions is convergent in $L^\infty(\Omega,\nu)$ as well, $(u_k^n)_{k=1}^\infty$ is convergent in $L^\infty(\Omega,\nu)$ and $V^mX(\Omega,\nu)\hookrightarrow \hookrightarrow L^\infty(\Omega,\nu)$, as required.
\end{proof}

\begin{proof}[Proof of Theorem~\ref{T:K}]
According to~\cite[Proposition 8.6]{CPS}, for every $f\in \mathcal M(0,1)$ we have
$$
\|H^m_If\|_{Y(0,1)} \approx \|K^m_If\|_{Y(0,1)},
$$
up to multiplicative constants depending on $m$ and $I$. Thus, given $a\in (0,1)$ and $f\in \mathcal M(0,1)$, we obtain
\begin{equation}\label{E:equivalence}
\|H^m_I(\chi_{(0,a)}f)\|_{Y(0,1)} \approx \|K^m_I(\chi_{(0,a)}f)\|_{Y(0,1)},
\end{equation}
up to multiplicative constants depending on $m$ and $I$. This yields that~\eqref{E:almost_comp_op_K} is equivalent to~\eqref{E:almost_comp_op}. Theorem~\ref{T:main} now yields that~\eqref{E:almost_comp_op_K} implies~\eqref{E:comp_emb_K}.

Consider the function $J$ defined by~\eqref{E:J}. We claim that, given $t\in (0,1)$, 
\begin{equation}\label{E:esup}
\esup_{s\in (0,t)} \frac{1}{J(s)}= \sup_{s\in (0,t)} \frac{s^{m-1}}{(I(s))^m}.
\end{equation}
Indeed, we trivially have
$$
\esup_{s\in (0,t)} \frac{1}{J(s)} = \esup_{s\in (0,t)} \frac{s^{m-1}}{(I(s))^m} \leq \sup_{s\in (0,t)} \frac{s^{m-1}}{(I(s))^m}.
$$
Conversely, because $I$ is nondecreasing on $(0,1]$, for every $s\in (0,t)$
$$
\frac{s^{m-1}}{(I(s))^m} \leq \frac{s^{m-1}}{(\lim_{r\to s_-} I(r))^m} = \lim_{r\to s_-} \frac{r^{m-1}}{(I(r))^m} \leq \esup_{r\in (0,t)} \frac{r^{m-1}}{(I(r))^m}.
$$
Passing to supremum over all $s\in (0,t)$, we obtain
$$
\sup_{s\in (0,t)} \frac{s^{m-1}}{(I(s))^m} \leq \esup_{r\in (0,t)} \frac{r^{m-1}}{(I(r))^m},
$$
which completes the proof of~\eqref{E:esup}. Equality~\eqref{E:esup} then implies that 
$$
\lim_{t\to 0_+} \esup_{s\in (0,t)} \frac{1}{J(s)}=0
$$
holds if and only if
\begin{equation}\label{E:lim}
\lim_{t\to 0_+} \frac{t^{m-1}}{(I(t))^m}=0.
\end{equation}

Suppose that~\eqref{E:lim} is not satisfied (i.e., part \textup{(a)} is in progress). Since $K^m_I=H_J$, an application of Theorem~\ref{T:compact_operator} yields that~\eqref{E:comp_op_K} is equivalent to~\eqref{E:almost_comp_op_K}. Furthermore,  according to the first part of the proof, each of~\eqref{E:comp_op_K} and~\eqref{E:almost_comp_op_K} implies~\eqref{E:comp_emb_K}. 

Next, assume that~\eqref{E:lim} is fulfilled (i.e., part \textup{(b)} is in progress). By another using of Theorem~\ref{T:compact_operator} and of the fact that $K^m_I=H_J$, we get that~\eqref{E:almost_comp_op_K} holds with $X(0,1)=L^1(0,1)$ and $Y(0,1)=L^\infty(0,1)$. Since, in general, $X(0,1) \hookrightarrow L^1(0,1)$ and $L^\infty(0,1) \hookrightarrow Y(0,1)$, condition~\eqref{E:almost_comp_op_K} is satisfied for all rearrangement-invariant norms $\|\cdot\|_{X(0,1)}$ and $\|\cdot\|_{Y(0,1)}$. The first part of the proof thus yields that condition~\eqref{E:comp_emb_K} is fulfilled independently of the choice of $\|\cdot\|_{X(0,1)}$ and $\|\cdot\|_{Y(0,1)}$ as well. 
\end{proof}

\section{Compactness of Sobolev embeddings on concrete measure spaces}\label{S:concrete}

In this section we characterize compact Sobolev embeddings on Euclidean John domains, on Maz'ya classes of Euclidean domains and, finally, on product probability spaces, whose standard example is the Gauss space. Recall that definitions and basic properties of the above mentioned measure spaces can be found in Section~\ref{S:sobolev}.

We start with the case of Euclidean John domains in $\mathbb R^n$, $n\geq 2$. In order to characterize $m$-th order compact Sobolev embeddings on these domains (for some $m\in \mathbb N$), we shall consider the operator $Q^m_n$ defined by
$$
Q^m_nf(t) = \int_t^1 |f(s)| s^{\frac{m}{n}-1}\,ds, ~~f\in \mathcal M(0,1), ~~t\in (0,1).
$$

\begin{theorem}\label{T:john}
Let $n\in \mathbb N$, $n\geq 2$, let $m\in \mathbb N$ and let $\Omega$ be a John domain in $\mathbb R^n$. Suppose that $\|\cdot\|_{X(0,1)}$ and $\|\cdot\|_{Y(0,1)}$ are rearrangement-invariant norms. If $m\leq n$ then the compact Sobolev embedding
\begin{equation}\label{E:emb_john}
V^mX(\Omega) \hookrightarrow \hookrightarrow Y(\Omega)
\end{equation}
is equivalent to each of the following two conditions:

\medskip
\textup{(i)} $Q^{m}_{n}: X(0,1) \rightarrow \rightarrow Y(0,1);$ 

\textup{(ii)} $\lim_{a\to 0_+} \sup_{\|f\|_{X(0,1)}\leq 1} \|Q^{m}_{n}(\chi_{(0,a)}f)\|_{Y(0,1)} = 0.$ 

\medskip\noindent
In particular, if $m=n$ then~\eqref{E:emb_john} is satisfied for all pairs of rearrangement-invariant norms $\|\cdot\|_{X(0,1)}$ and $\|\cdot\|_{Y(0,1)}$ except of those for which $X(0,1)=L^1(0,1)$ and $Y(0,1)=L^\infty(0,1)$. Furthermore, if $m>n$ then~\eqref{E:emb_john} is fulfilled independently of the choice of $\|\cdot\|_{X(0,1)}$ and $\|\cdot\|_{Y(0,1)}$.
\end{theorem}

We note that the equivalence of~\eqref{E:emb_john} and \textup{(i)} in Theorem~\ref{T:john} is already known in the special case when $\Omega$ is a domain having a Lipschitz boundary and $m<n$, see~\cite{KP}.

\medskip

Let us now focus on Maz'ya classes of domains in $\mathbb R^n$, $n\geq 2$. When dealing with $m$-th order Sobolev embeddings on a domain from the Maz'ya class $\mathcal J_\alpha$ (for some $m\in \mathbb N$ and $\alpha \in [\frac{1}{n'},1]$), we shall use the operator $T^m_\alpha$ given by
$$
T^m_\alpha f(t)=\int_t^1 |f(s)|s^{-1+m(1-\alpha)}\,ds, ~~f\in \mathcal M(0,1), ~~t\in (0,1),
$$
if $\alpha \in [\frac{1}{n'},1)$, and by
$$
T^m_1f(t)=\frac{1}{(m-1)!}\int_t^1 |f(s)|\frac{(\log \frac{s}{t})^{m-1}}{s}\,ds, ~~f\in \mathcal M(0,1), ~~t\in (0,1). 
$$

\begin{theorem}\label{T:mazya}
Let $n\in \mathbb N$, $n\geq 2$, let $m\in \mathbb N$ and let $\alpha \in [\frac{1}{n'},1]$. Suppose that $\|\cdot\|_{X(0,1)}$ and $\|\cdot\|_{Y(0,1)}$ are rearrangement-invariant norms. If $m(1-\alpha)\leq 1$ (notice that this is true for every $m\in \mathbb N$ provided that $\alpha=1$) then the fact that
\begin{equation}\label{E:emb_mazya}
V^mX(\Omega) \hookrightarrow \hookrightarrow Y(\Omega) \textup{ holds for every } \Omega \in \mathcal J_\alpha
\end{equation} 
is equivalent to each of the following conditions:

\medskip
\textup{(i)} $T^m_\alpha: X(0,1) \rightarrow \rightarrow Y(0,1)$; 

\textup{(ii)} $\lim_{a\to 0_+} \sup_{\|f\|_{X(0,1)}\leq 1} \|T^m_\alpha(\chi_{(0,a)}f)\|_{Y(0,1)} = 0$.

\medskip\noindent
In particular, if $m(1-\alpha)=1$ then~\eqref{E:emb_mazya} is satisfied for all pairs of rearrangement-invariant norms $\|\cdot\|_{X(0,1)}$ and $\|\cdot\|_{Y(0,1)}$ except of those for which $X(0,1)=L^1(0,1)$ and $Y(0,1)=L^\infty(0,1)$. Furthermore, if $m(1-\alpha)>1$ then condition~\eqref{E:emb_mazya} is fulfilled independently of the choice of $\|\cdot\|_{X(0,1)}$ and $\|\cdot\|_{Y(0,1)}$.
\end{theorem}

\begin{remarks}\label{T:remarks6}
\textup{(a)} It will follow from the proof of Theorem~\ref{T:john} that its statement is true for all domains $\Omega$ belonging to the Maz'ya class $\mathcal J_{\frac{1}{n'}}$ (this class contains, in particular, all John domains). 

\textup{(b)} Let $m$, $n\in \mathbb N$, $n\geq 2$, let $\alpha \in [\frac{1}{n'},1]$ and let $\Omega$ be a domain in $\mathbb R^n$ belonging to the Maz'ya class $\mathcal J_\alpha$. Suppose that $\|\cdot\|_{X(0,1)}$ and $\|\cdot\|_{Y(0,1)}$ are rearrangement-invariant norms. Consider the following two assertions: 

\textup{(i)} $V^mX(\Omega) \hookrightarrow \hookrightarrow Y(\Omega)$; 

\textup{(ii)} $V^mX(\Omega') \hookrightarrow \hookrightarrow Y(\Omega')$ holds for each $\Omega' \in \mathcal J_\alpha$.

If $\alpha=\frac{1}{n'}$ then conditions \textup{(i)} and \textup{(ii)} are equivalent (this follows from Theorem~\ref{T:john} combined with the part \textup{(a)} of this remark). However, such an equivalence is no longer true when $\alpha>\frac{1}{n'}$. This can be easily observed since each Maz'ya class $\mathcal J_\alpha$ contains, in particular, all John domains. Compactness of Sobolev embeddings on John domains is characterized by compactness of the operator $Q^m_n$, which does not coincide with compactness of $T^m_\alpha$. On the other hand, given an arbitrary $\alpha \in (\frac{1}{n'},1]$, there is one domain $\Omega \in \mathcal J_\alpha$ for which the equivalence of \textup{(i)} and \textup{(ii)} holds for all rearrangement-invariant norms $\|\cdot\|_{X(0,1)}$ and $\|\cdot\|_{Y(0,1)}$ (an example of such a domain can be found in Proposition~\ref{T:mazya_domain}).

\textup{(c)} The operators $Q^m_n$ and $T^m_\alpha$ can be described via the operators \lq\lq$H$" and \lq\lq$K$" defined in Sections~\ref{S:operators} and~\ref{S:main}, respectively, in the following way:
$$
Q^m_n = K^m_{s^\frac{1}{n'}} = H_{s^{1-\frac mn}}
$$
and
$$
T^m_\alpha=
\begin{cases}
K^m_{s^\alpha} = H_{s^{1-m(1-\alpha)}},\,  &\alpha \in [\frac{1}{n'},1);\\
H^m_s,\, &\alpha=1.
\end{cases}
$$
Hence, Theorems~\ref{T:opt_range} and~\ref{T:opt_domain} applied to an appropriate operator \lq\lq$H$" provide further characterization of compactness of $Q^m_n$ and $T^m_\alpha$.
\end{remarks}

We finally focus on product probability spaces $(\mathbb R^n,\mu_{\Phi,n})$, where $n\in \mathbb N$ and the measure $\mu_{\Phi,n}$ is defined by~\eqref{E:one_dim} if $n=1$ and by~\eqref{E:n_dim} if $n>1$. Given $m\in \mathbb N$, we characterize compact Sobolev embeddings on $(\mathbb R^n,\mu_{\Phi,n})$ in terms of compactness of the operator $H^m_{L_\Phi}$, with $L_\Phi$ as in~\eqref{E:l_phi}. The operator $H^m_{L_\Phi}$ therefore has the form
\begin{align*}
H^m_{L_\Phi}f(t)
&= \frac{1}{(m-1)!}\int_t^1 \frac{|f(s)|}{s\Phi'\left(\Phi^{-1}\left(\log \frac{2}{s}\right)\right)} \left(\int_t^s \frac{\,dr}{r\Phi'\left(\Phi^{-1}\left(\log \frac{2}{r}\right)\right)}\right)^{m-1}\,ds\\
&=\frac{1}{(m-1)!} \int_t^1 |f(s)| \frac{\left(\Phi^{-1}\left(\log\frac{2}{t}\right)-\Phi^{-1}\left(\log\frac{2}{s}\right)\right)^{m-1}}{s\Phi'\left(\Phi^{-1}\left(\log \frac{2}{s}\right)\right)}\,ds, ~~f\in \mathcal M(0,1), ~t\in (0,1).
\end{align*}
We also prove that compactness of the operator $H^m_{L_\Phi}$ coincides with compactness of the somewhat simpler operator $P^m_\Phi$, defined by
$$
P^m_\Phi f(t)=\left(\frac{\Phi^{-1}(\log \frac 2t)}{\log  \frac 2t}\right)^m \int_t^1 \frac{|f(s)|}{s} \left(\log \frac st\right)^{m-1} \,ds, ~~f\in \mathcal M(0,1), ~~t\in (0,1).
$$
We note that the operator $P^m_\Phi$ was introduced in~\cite{CPS} where it was shown that boundedness of $H^m_{L_\Phi}$ is equivalent to boundedness of $P^m_\Phi$. 

Furthermore, we show that anologues of Theorems~\ref{T:compact_operator}, \ref{T:opt_range} and~\ref{T:opt_domain} hold for the operator $P^m_\Phi$, although it does not have the form $H^j_J$ for some $j\in \mathbb N$ and some function $J$. In order to do this, we define two families of rearrangement-invariant norms, playing the role of optimal range and optimal domain norms with respect to the operator $P^m_\Phi$.

Let $\|\cdot\|_{X(0,1)}$ be a rearrangement-invariant norm. Given $m\in \mathbb N$, consider the rearrangement-invariant norm $\|\cdot\|_{X_m(0,1)}$ whose associate norm fulfils
$$
\|f\|_{X_m'(0,1)} =\left\|\frac{1}{s} \int_0^s \left( \log \frac st\right) f^*(t)\,dt \right\|_{X'(0,1)}
$$
for every $f\in \mathcal M(0,1)$. Then the functional $\|\cdot\|_{X^r_{m,\Phi}(0,1)}$, given for every $f\in \mathcal M(0,1)$ by
$$
\|f\|_{X^r_{m,\Phi}(0,1)} = \left\|\left( \frac{\log \frac 2s}{\Phi^{-1}(\log \frac 2s)}\right)^m f^*(s)\right\|_{X_m(0,1)},
$$
is a rearrangement-invariant norm and we have $X^r_{m,\Phi}(0,1)=X^r_{m,L_\Phi}(0,1)$, up to equivalent norms, see~\cite[Theorem 7.3 and its proof]{CPS}.

Further, let $\|\cdot\|_{Y(0,1)}$ be a rearrangement-invariant norm fulfilling
\begin{equation}\label{E:cond_Y_phi}
\left\|\left(\Phi^{-1}\left(\log \frac 2s\right) \right)^m\right\|_{Y(0,1)}<\infty.
\end{equation}
For every $f\in \mathcal M(0,1)$ define the functional $\|\cdot\|_{Y^d_{m,\Phi}(0,1)}$ by
\begin{align*}
\|f\|_{Y^d_{m,\Phi}(0,1)} 
&= \sup_{h \sim f} \|P^m_\Phi h\|_{Y(0,1)} +\|f\|_{L^1(0,1)}\\
&= \sup_{0\leq h \sim f}\left \|\left(\frac{\Phi^{-1}(\log \frac 2t)}{\log  \frac 2t}\right)^m \int_t^1 \frac{h(s)}{s} \left(\log \frac st\right)^{m-1} \,ds\right\|_{Y(0,1)}+\|f\|_{L^1(0,1)}.
\end{align*}
The fact that the functional $\|\cdot\|_{Y^d_{m,\Phi}(0,1)}$ is actually a rearrangement-invariant norm can be proved in the same way as it is done for the functional $\|\cdot\|_{Y^d_{j,J}(0,1)}$ in the proof of Proposition~\ref{T:proposition}.

\begin{theorem}\label{T:probability}
Let $n,m\in \mathbb N$ and let $\Phi$ be as in Section~\ref{S:sobolev}. Suppose that $\|\cdot\|_{X(0,1)}$ and $\|\cdot\|_{Y(0,1)}$ are rearrangement-invariant norms. Then the following conditions are equivalent:

\medskip
\textup{(i)} $V^mX(\mathbb R^n,\mu_{\Phi,n}) \hookrightarrow \hookrightarrow Y(\mathbb R^n,\mu_{\Phi,n});$ 

\textup{(ii)} $H^m_{L_\Phi}: X(0,1) \rightarrow \rightarrow Y(0,1);$

\textup{(iii)} $P_\Phi^m: X(0,1) \rightarrow \rightarrow Y(0,1);$ 

\textup{(iv)} $\lim_{a\to 0_+} \sup_{\|f\|_{X(0,1)}\leq 1} \|P_\Phi^m(\chi_{(0,a)}f)\|_{Y(0,1)} = 0;$

\textup{(v)} $\lim_{a\to 0_+} \sup_{\|f\|_{X(0,1)}\leq 1} \|\chi_{(0,a)} P_\Phi^m f\|_{Y(0,1)} = 0;$

\textup{(vi)} $X^r_{m,\Phi}(0,1) \overset{*}{\hookrightarrow} Y(0,1).$

\medskip
\noindent
Furthermore, if $X(0,1) \neq L^1(0,1)$ and~\eqref{E:cond_Y_phi} is satisfied, then \textup{(i)} -- \textup{(vi)} are equivalent to each of the following conditions:

\medskip
\textup{(vii)} $\lim_{a\to 0_+} \sup_{\|f\|_{X(0,1)}\leq 1} \sup_{\lambda_1(E)\leq a} \|P_\Phi^m(\chi_{E}f)\|_{Y(0,1)} = 0;$

\textup{(viii)} $X(0,1) \overset{*}{\hookrightarrow} Y^d_{m,\Phi}(0,1).$
\end{theorem}

Observe that Theorem~\ref{T:probability} yields that, in contrast to the Euclidean setting, compact Sobolev embeddings on $(\mathbb R^n,\mu_{\Phi,n})$ do not depend on the dimension $n$, in the sense that we have the equivalence of the following two assertions.

\noindent\textup{(i)} There exists $n\in \mathbb N$ for which $V^mX(\mathbb R^n,\mu_{\Phi,n}) \hookrightarrow \hookrightarrow Y(\mathbb R^n,\mu_{\Phi,n})$ is satisfied.

\noindent\textup{(ii)} The compact embedding $V^mX(\mathbb R^n,\mu_{\Phi,n}) \hookrightarrow \hookrightarrow Y(\mathbb R^n,\mu_{\Phi,n})$ is satisfied for every $n\in \mathbb N$. 

\medskip
Let us now prove the results we have stated. The proofs are based on the results of the previous section, and on the following

\begin{proposition}\label{T:reverse}
Assume that $(\Omega,\nu)$ is as in Section~\ref{S:sobolev}. Let $m\in \mathbb N$ and let $\|\cdot\|_{X(0,1)}$ and $\|\cdot\|_{Y(0,1)}$ be rearrangement-invariant norms satisfying
\begin{equation}\label{E:compact_embedding}
V^mX(\Omega,\nu) \hookrightarrow \hookrightarrow Y(\Omega,\nu).
\end{equation}
Let $\alpha\in (0,1]$. Denote
$$
X_+^\alpha=\{ f\in X(0,1) \cap \mathcal M_+(0,1): f=0 \hbox{ a.e. on } (0,1)\backslash (0,\alpha) \}. 
$$
Suppose that $L$ is an operator defined on $X_+^\alpha$, with values in $V^mX(\Omega,\nu)$, fulfilling that 
\begin{equation}\label{E:cond_1}
\|Lf\|_{V^mX(\Omega,\nu)} \leq C \|f\|_{X(0,1)}
\end{equation}
for some positive constant $C$ and for all $f\in X_+^\alpha$. Set $Hf=(Lf)^*_\nu$, $f\in X_+^\alpha$. Assume that
\begin{equation}\label{E:rearrangement}
Hf(t)=\int_t^{1} f(s) K(s,t)\,ds, ~~f\in X_+^\alpha, ~~t\in (0,1),
\end{equation}
for some real valued function $K$ satisfying that $K(\cdot,t)$ is nonnegative and measurable on $(t,1)$ for every $t\in (0,1)$.
Then
\begin{equation}\label{E:H_2}
\lim_{a\to 0_+} \sup_{\|f\|_{X(0,1)}\leq 1} \|H(\chi_{(0,a)}|f|)\|_{Y(0,1)}=0.
\end{equation}
\end{proposition}

\begin{proof}
We first observe that whenever $k$ is a positive integer satisfying $1/k\leq \alpha$ and $f \in X(0,1) \cap \mathcal M_+(0,1)$, then $\chi_{(0,1/k)}f \in X_+^\alpha$ and the functions $L(\chi_{(0,1/k)}f)$ and $H(\chi_{(0,1/k)}f)$ are thus well defined.
Since condition~\eqref{E:compact_embedding} implies $V^mX(\Omega,\nu) \hookrightarrow Y(\Omega,\nu)$, for every $k\in \mathbb N$ satisfying $1/k\leq \alpha$ we have
\begin{align*}
\sup_{\|f\|_{X(0,1)}\leq 1} \|H(\chi_{(0,1/k)}|f|)\|_{Y(0,1)}
&=\sup_{\|f\|_{X(0,1)}\leq 1} \|L(\chi_{(0,1/k)}|f|)\|_{Y(\Omega,\nu)}\\
&\leq \sup_{\|f\|_{X(0,1)}\leq 1} C'\|L(\chi_{(0,1/k)}|f|)\|_{V^mX(\Omega,\nu)}\\
&\leq \sup_{\|f\|_{X(0,1)}\leq 1} C'C \|\chi_{(0,1/k)}|f|\|_{X(0,1)} \leq C'C<\infty,
\end{align*}
where $C$ is the constant from~\eqref{E:cond_1} and $C'$ is the constant from the embedding $V^mX(\Omega,\nu) \hookrightarrow Y(\Omega,\nu)$. Consequently, for every $k$ as above we can find a function $f_k\in \mathcal M_+(0,1)$ such that $\|f_k\|_{X(0,1)}\leq 1$ and 
\begin{equation}\label{E:ineq_0}
\sup_{\|f\|_{X(0,1)}\leq 1} \|H(\chi_{(0,1/k)}|f|)\|_{Y(0,1)} < \|H(\chi_{(0,1/k)}f_k)\|_{Y(0,1)} + \frac{1}{k}.
\end{equation}
Since the sequence $(\chi_{(0,1/k)}f_k)_{k=\left\lceil 1/\alpha \right\rceil}^\infty$ is bounded in $X(0,1)$, it follows that $(L(\chi_{(0,1/k)}f_k))_{k=\left\lceil 1/\alpha \right\rceil}^\infty$ must be bounded in $V^mX(\Omega,\nu)$ due to~\eqref{E:cond_1}. Thanks to~\eqref{E:compact_embedding}, there is a subsequence $(f_{k_\ell})_{\ell=1}^\infty$ of $(f_k)_{k=\left\lceil 1/\alpha \right\rceil}^\infty$ such that $(L(\chi_{(0,1/k_\ell)}f_{k_\ell}))_{\ell=1}^\infty$ converges to some function $g$ in the norm of the space $Y(\Omega,\nu)$. Then, in particular, $L(\chi_{(0,1/k_\ell)}f_{k_\ell}) \rightarrow g$ in measure. 

Observe that for every $\ell \in \mathbb N$, we have $H(\chi_{(0,1/k_\ell)}f_{k_\ell})(t)=0$ when $t\in (1/k_\ell,1)$, thanks to~\eqref{E:rearrangement}. Since $(L(\chi_{(0,1/k_\ell)}f_{k_\ell}))^*_\nu=H(\chi_{(0,1/k_\ell)}f_{k_\ell})$, the distribution function of $L(\chi_{(0,1/k_\ell)}f_{k_\ell})$ with respect to $\nu$ coincides with that of $H(\chi_{(0,1/k_\ell)}f_{k_\ell})$ with respect to the one-dimensional Lebesgue measure $\lambda_1$. In particular,
\begin{align}\label{E:nu_3} \nonumber
&\lim_{\ell\to \infty} \nu\left(\left\{x\in \Omega: \left|L(\chi_{(0,1/{k_\ell})}f_{k_\ell})(x)\right|>0\right\}\right)\\
&= \lim_{\ell\to \infty} \lambda_1\left(\left\{s \in (0,1): H(\chi_{(0,1/{k_\ell})}f_{k_\ell})(s)>0\right\}\right)
\leq \lim_{\ell \to \infty} \frac{1}{k_\ell}=0. 
\end{align}
Thus, $L(\chi_{(0,1/k_\ell)}f_{k_\ell}) \rightarrow 0$ in measure. This implies that $g=0$ $\nu$-a.e.\ on $\Omega$. So, 
$$
\lim_{\ell\to \infty} \|H(\chi_{(0,1/k_\ell)}f_{k_\ell})\|_{Y(0,1)} = \lim_{\ell\to \infty} \|L(\chi_{(0,1/k_\ell)}f_{k_\ell})\|_{Y(\Omega,\nu)}=0.
$$
Inequality~\eqref{E:ineq_0} now yields
$$
\lim_{\ell\to \infty} \sup_{\|f\|_{X(0,1)}\leq 1} \|H(\chi_{(0,1/k_\ell)}|f|)\|_{Y(0,1)} =0.
$$
Using that the function
$$
a\mapsto \sup_{\|f\|_{X(0,1)}\leq 1} \|H(\chi_{(0,a)}|f|)\|_{Y(0,1)}
$$
is nondecreasing on $(0,\alpha]$, we obtain~\eqref{E:H_2}.
\end{proof}



\begin{proof}[Proof of Theorem~\ref{T:john}]
Consider the function $I(t)=t^{\frac{1}{n'}}$, $t\in (0,1]$. It was observed in Section~\ref{S:sobolev} that $(\Omega,\lambda_n,I)$ is a compatible triplet. Furthermore, the function $I$ satisfies~\eqref{E:aprox}. We have that $\lim_{t\to 0_+} \frac{t^{m-1}}{(I(t))^m} =\lim_{t\to 0_+} t^{\frac{m}{n}-1}= 0$ holds if and only if $m>n$. In such a case, the proof follows directly from Theorem~\ref{T:K}. 

Suppose that $m\leq n$. Since $Q^m_n=K^m_I$, Theorem~\ref{T:K} gives the equivalence of \textup{(i)} and \textup{(ii)} and also shows that each of the conditions \textup{(i)} and \textup{(ii)} implies~\eqref{E:emb_john}. Hence, it only remains to prove that~\eqref{E:emb_john} implies \textup{(ii)}.

Assume that $m=n$ and $X(0,1) \neq L^1(0,1)$. Then there is nothing to prove, since condition \textup{(ii)} (or, equivalently, \textup{(i)}) always holds. Indeed, we have 
\begin{equation}\label{E:Q}
Q^m_m: L^1(0,1) \rightarrow L^\infty(0,1)
\end{equation} 
and $L^\infty(0,1) \hookrightarrow Y(0,1)$. Therefore, $Q^m_m: L^1(0,1) \rightarrow Y(0,1)$. The conclusion now follows from Remark~\ref{T:remark_2} and from the fact that $Q^m_m=H_1$. 

Suppose that~\eqref{E:emb_john} is satisfied and $m<n$, or $m=n$ and $X(0,1)=L^1(0,1)$. Let $B_R$ be an open ball of radius $R>0$ such that $\overline{B_R} \subseteq \Omega$. Without loss of generality we may assume that $B_R$ is centered at $0$ and that $\kappa_n R^n\leq 1$, where $\kappa_n$ denotes the Lebesgue measure of the unit ball in $\mathbb R^n$. Let $f$ be a function belonging to the set $X^{\kappa_nR^n}_+$ defined in Proposition~\ref{T:reverse}. Then we set
$$
Lf(x)=
\begin{cases}
\int_{\kappa_n|x|^n}^{\kappa_n R^n} \int_{r_1}^{\kappa_n R^n} \dots \int_{r_{m-1}}^{\kappa_n R^n} f(r_m) r_m^{-m+\frac{m}{n}}\,dr_m\, \dots\, dr_1,\,
&x\in B_R;\\
0,\, &x\in \Omega/B_R.
\end{cases}
$$
It is not hard to observe that $Lf$ is an $m$-times weakly differentiable function on $\Omega$. By subsequent applications of the Fubini theorem, we obtain
$$
Lf(x)=
\begin{cases}
\frac{1}{(m-1)!}\int_{\kappa_n|x|^n}^{\kappa_n R^n} f(s)s^{-m+\frac{m}{n}} (s-\kappa_n|x|^n)^{m-1}\,ds,\,
&x\in B_R;\\
0,\, &x\in \Omega/B_R.
\end{cases}
$$
Denote $Hf=(Lf)^*_{\lambda_n}$. Then
\begin{align}\label{E:H_rearrangement}
Hf(t)
&=\chi_{(0,\kappa_nR^n)}(t) \frac{1}{(m-1)!}\int_{t}^{\kappa_n R^n} f(s)s^{-m+\frac{m}{n}} (s-t)^{m-1}\,ds\\
&=\frac{1}{(m-1)!}\int_{t}^{1} f(s)s^{-m+\frac{m}{n}} (s-t)^{m-1}\,ds, ~~t\in (0,1). \nonumber
\end{align}
One can show similarly as in~\cite[proof of Theorem A]{KP2} that 
\begin{align*}
\sum_{i=1}^m\left\||\nabla^iLf| \right\|_{X(\Omega)}
&= \sum_{i=1}^m \left\|(|\nabla^iLf|)^*_{\lambda_n}\right\|_{X(0,1)}\\ 
&\lesssim \|f\|_{X(0,1)} + \sum_{i=1}^{m-1} \left\|t^{i-\frac{m}{n}}\int_t^{1} f(s)s^{-i+\frac{m}{n}-1}\,ds\right\|_{X(0,1)}.
\end{align*}
If $m<n$ then it follows from~\cite[proof of Theorem A]{KP2} that for every $i\in \{1,2,\dots, m-1\}$,
$$
\left\|t^{i-\frac{m}{n}}\int_t^{1} f(s)s^{-i+\frac{m}{n}-1}\,ds\right\|_{X(0,1)} \lesssim \|f\|_{X(0,1)}.
$$
In the remaining case when $m=n$ and $X(0,1)=L^1(0,1)$ we obtain by the Fubini theorem that
\begin{align*}
\left\|t^{i-1}\int_t^{1} f(s)s^{-i}\,ds\right\|_{X(0,1)}
&\approx \left\|t^{i-1}\int_t^{1} f(s)s^{-i}\,ds\right\|_{L^1(0,1)}=\int_0^{1} f(s)s^{-i}\int_0^s t^{i-1}\,dt \,ds\\ 
&= \frac{1}{i}\|f\|_{L^1(0,1)} \approx \|f\|_{X(0,1)},
\end{align*}
$i\in \{1,2,\dots, m-1\}$. Hence, in all cases we have
$$
\sum_{i=1}^m\left\||\nabla^iLf| \right\|_{X(\Omega)} \lesssim \|f\|_{X(0,1)}.
$$
Furthermore, by~\eqref{E:H_rearrangement},
\begin{align*}
\|Lf\|_{L^1(\Omega)} 
&= \|Hf\|_{L^1(0,1)} \leq \frac{1}{(m-1)!} \left\|\int_t^{1} f(s) s^{\frac{m}{n}-1}\,ds\right\|_{L^1(0,1)}\\
&=\frac{1}{(m-1)!} \int_0^1 f(s)s^{\frac{m}{n}}\,ds \leq \frac{1}{(m-1)!} \|f\|_{L^1(0,1)} \leq \frac{C_X}{(m-1)!} \|f\|_{X(0,1)}.
\end{align*}
Altogether, we obtain
$$
\|Lf\|_{V^mX(\Omega)} \leq \|Lf\|_{L^1(\Omega)} + \max(C_X,1)\sum_{i=1}^m \||\nabla^i Lf|\|_{X(\Omega)} \lesssim \|f\|_{X(0,1)},
$$
up to multiplicative constants independent of $f\in X^{\kappa_n R^n}_+$. The operator $L$ therefore satisfies~\eqref{E:cond_1}. Proposition~\ref{T:reverse} now gives that 
\begin{equation}\label{E:lim_H_0}
\lim_{a\to 0_+} \sup_{\|f\|_{X(0,1)}\leq 1} \|H(\chi_{(0,a)}|f|)\|_{Y(0,1)}=0.
\end{equation}
Since the constant function $1$ fulfils~\eqref{E:aprox}, the equivalence~\eqref{E:equivalence} implies that for all $a\in (0,\kappa_n R^n)$ and for all $f\in X(0,1)$,
\begin{align}\label{E:equivalence_2}
\|H(\chi_{(0,a)}(s)|f(s)|)\|_{Y(0,1)} 
&= \|H^m_1(\chi_{(0,a)}(s)f(s)s^{-m+\frac{m}{n}})\|_{Y(0,1)}\\ \nonumber
&\approx \|K^m_1(\chi_{(0,a)}(s)f(s)s^{-m+\frac{m}{n}})\|_{Y(0,1)}\\
&=\|Q^m_n(\chi_{(0,a)}(s)f(s))\|_{Y(0,1)}, \nonumber
\end{align}
up to multiplicative constants depending on $m$. The assertion \textup{(ii)} follows by combined using of~\eqref{E:lim_H_0} and~\eqref{E:equivalence_2}.

Finally, in order to obtain a characterization of~\eqref{E:emb_john} in the case when $m=n$, it suffices to describe when \textup{(i)} holds with $m=n$. We have already shown that if $X(0,1)\neq L^1(0,1)$ then \textup{(i)} is satisfied. Furthermore, it follows from~\eqref{E:Q}, from the embedding $X(0,1) \hookrightarrow L^1(0,1)$, from $Q^m_m=H_1$ and from Remark~\ref{T:remark} that if $Y(0,1)\neq L^\infty(0,1)$ then \textup{(i)} is fulfilled as well. On the other hand, \textup{(i)} is not fulfilled with $X(0,1)=L^1(0,1)$ and $Y(0,1)=L^\infty(0,1)$, see Remark~\ref{T:remark} again.
\end{proof}

The following proposition, which easily follows from~\cite[Section 5.3.3]{M}, provides examples of Euclidean domains belonging to the class $\mathcal J_\alpha$.

\begin{proposition}\label{T:mazya_domain}
Let $n\in \mathbb N$, $n\geq 2$, and let $\alpha \in [\frac{1}{n'},1]$. Set $L_\alpha=\frac{1}{1-\alpha}$ if $\alpha \in [\frac{1}{n'},1)$, and $L_1=\infty$. Define the function $\eta_\alpha: (0,L_\alpha) \rightarrow (0,\infty)$ by
$$
\eta_\alpha(r)=
\begin{cases}
\kappa_{n-1}^{-\frac{1}{n-1}} (1-(1-\alpha)r)^{\frac{\alpha}{(1-\alpha)(n-1)}},\,
&\alpha \in [\frac{1}{n'},1),\\
\kappa_{n-1}^{-\frac{1}{n-1}} e^{-\frac{r}{n-1}},\, &\alpha=1,
\end{cases}
$$
where $\kappa_{n-1}$ denotes the Lebesgue measure of the unit ball in $\mathbb R^{n-1}$. 
Let $\Omega_\alpha$ be the domain in $\mathbb R^n$ given by
$$
\Omega_\alpha= \{(x',x_n) \in \mathbb R^n: x' \in \mathbb R^{n-1}, x_n\in (0,L_\alpha), |x'|<\eta_\alpha(x_n)\}.
$$
Then $\lambda_n(\Omega_\alpha)=1$ and $I_{\Omega_\alpha}(s) \approx s^\alpha$, $s\in [0,\frac{1}{2}]$. 
\end{proposition}

\begin{proof}[Proof of Theorem~\ref{T:mazya}]
Set $I_\alpha(t)=t^\alpha$, $t\in (0,1]$. As observed in Section~\ref{S:sobolev}, $(\Omega,\lambda_n,I_\alpha)$ is a compatible triplet for each domain $\Omega \in \mathcal J_\alpha$. Suppose that $\alpha \in [\frac{1}{n'},1)$. Then~\eqref{E:aprox} is fulfilled with $I=I_\alpha$ and we have $\frac{t^{m-1}}{((I_\alpha(t))^m}=t^{-1+m(1-\alpha)}$ for $t\in (0,1]$. Thus, $\lim_{t\to 0_+} \frac{t^{m-1}}{((I_\alpha(t))^m}=0$ holds if and only if $m(1-\alpha)>1$. Provided that this condition is satisfied, the proof is a direct consequence of Theorem~\ref{T:K}. On the other hand, if $m(1-\alpha)\leq 1$ then Theorem~\ref{T:K} combined with the fact that $T^m_\alpha=K^m_{I_\alpha}$ yields that \textup{(i)} is equivalent to \textup{(ii)} and that each of the conditions \textup{(i)} and \textup{(ii)} implies~\eqref{E:emb_mazya}. Further, if $\alpha=1$ then the equivalence of \textup{(i)} and \textup{(ii)} and the implication \textup{(i)} (or \textup{(ii)}) implies~\eqref{E:emb_mazya} follow from Theorem~\ref{T:main} and from the fact that $T^m_1=H^m_{I_1}$. Thus, it suffices to prove that, in all cases, \eqref{E:emb_mazya} implies \textup{(ii)}.

Suppose that~\eqref{E:emb_mazya} is satisfied. Then, in particular, $V^mX(\Omega_\alpha) \hookrightarrow \hookrightarrow Y(\Omega_\alpha)$ holds for $\Omega_\alpha$ given by Proposition~\ref{T:mazya_domain}. Define the function $M_\alpha$ for every $t\in (0,L_\alpha)$ (with $L_\alpha$ as in Proposition~\ref{T:mazya_domain}) by
$$
M_\alpha(t)=
\begin{cases}
(1-(1-\alpha)t)^{\frac{1}{1-\alpha}},\, &\alpha \in \big[\frac{1}{n'},1\big);\\
e^{-t},\, &\alpha=1.
\end{cases}
$$
Then
\begin{equation}\label{E:level_set}
\lambda_n(\{(x',x_n) \in \Omega_\alpha: x_n>t\}) = M_\alpha(t), ~~t\in (0,L_\alpha),
\end{equation}
see~\cite[proof of Theorem 6.4]{CPS}.

Let $f$ be any function in $X(0,1) \cap \mathcal M_+(0,1)$ (or, what is the same, let $f$ be an arbitrary function belonging to the set $X^1_+$ defined in Proposition~\ref{T:reverse}). For $x=(x_1,\dots, x_n) \in \Omega_\alpha$, we set
\begin{align*}
Lf(x)
&=\int_{M_\alpha(x_n)}^1 \frac{1}{r_1^\alpha} \int_{r_1}^1 \frac{1}{r_2^\alpha} \dots \int_{r_{m-1}}^1 \frac{f(r_m)}{r_m^\alpha} \,dr_m\,dr_{m-1}\, \dots\,dr_1\\
&=\frac{1}{(m-1)!} \int_{M_\alpha(x_n)}^1 \frac{f(r)}{r^\alpha} \left(\int_{M_\alpha(x_n)}^r \frac{\,ds}{s^\alpha} \right)^{m-1}\,dr.
\end{align*}
Then $Lf$ is an $m$-times weakly differentiable function on $\Omega_\alpha$ and, owing to~\eqref{E:level_set}, we have
$$
(Lf)^*_{\lambda_n}(t) = \frac{1}{(m-1)!} \int_{t}^1 \frac{f(r)}{r^\alpha} \left(\int_{t}^r \frac{\,ds}{s^\alpha} \right)^{m-1}\,dr = H^m_{I_\alpha}f(t), ~~t\in (0,1). 
$$
Furthermore, it follows from~\cite[proof of Theorem 6.4]{CPS} that $L$ satisfies~\eqref{E:cond_1}. Hence, Proposition~\ref{T:reverse} implies that 
\begin{equation}\label{E:lim_H}
\lim_{a\to 0_+} \sup_{\|f\|_{X(0,1)}\leq 1} \|H^m_{I_\alpha}(\chi_{(0,a)}f)\|_{Y(0,1)} = 0.
\end{equation}
If $\alpha=1$ then~\eqref{E:lim_H} is exactly the condition \textup{(ii)} which we wanted to verify. If $\alpha \in [\frac{1}{n'},1)$ then the equivalence~\eqref{E:equivalence} (with $I=I_\alpha$) yields that for every $f\in X(0,1)$ and every $a\in (0,1)$ we have
\begin{equation}\label{E:equivalence_3}
\|H^m_{I_\alpha}(\chi_{(0,a)}f)\|_{Y(0,1)} \approx \|K^m_{I_\alpha}(\chi_{(0,a)}f)\|_{Y(0,1)}
=\|T^m_\alpha(\chi_{(0,a)}f)\|_{Y(0,1)},
\end{equation}
up to multiplicative constants depending on $m$ and $\alpha$. The condition \textup{(ii)} now follows by combined using of~\eqref{E:lim_H} and~\eqref{E:equivalence_3}.

In particular, we have proved that if $m(1-\alpha)=1$, then~\eqref{E:emb_mazya} is satisfied if and only if $T^m_\alpha=H_1: X(0,1) \rightarrow \rightarrow Y(0,1)$. It follows from the proof of Theorem~\ref{T:john} that this happens if and only if $X(0,1) \neq L^1(0,1)$ or $Y(0,1) \neq L^\infty(0,1)$. 
\end{proof}

\begin{proof}[Proof of Theorem~\ref{T:probability}]
We have observed in Section~\ref{S:sobolev} that $(\mathbb R^n, \mu_{\Phi,n},L_\Phi)$ is a compatible triplet. Theorem~\ref{T:main} therefore yields that \textup{(ii)} implies \textup{(i)}.

Conversely, suppose that \textup{(i)} is fulfilled. Let $f$ be an arbitrary function belonging to the set $X_+^{1/2}$ defined in Proposition~\ref{T:reverse}. For every $x=(x_1,\dots, x_n) \in \mathbb R^n$ we set
\begin{align*}
Lf(x)
&=\int_{F_\Phi(x_1)}^1 \frac{1}{I_\Phi(r_1)} \int_{r_1}^1 \frac{1}{I_\Phi(r_2)} \dots \int_{r_{m-1}}^1 \frac{f(r_m)}{I_\Phi(r_m)} \,dr_m\,dr_{m-1}\, \dots\,dr_1\\
&= \frac{1}{(m-1)!} \int_{F_\Phi(x_1)}^1 \frac{f(r)}{I_\Phi(r)}\left(\int_{F_\Phi(x_1)}^r \frac{\,ds}{I_\Phi(s)}\right)^{m-1}\,dr.
\end{align*}
Then $Lf$ is an $m$-times weakly differentiable function on $\mathbb R^n$. Denote $Hf=(Lf)^*_{\mu_{\Phi,n}}$. Then, thanks to the equality
$$
\mu_{\Phi,n}(\{(x_1,x_2,\dots, x_n)\in \mathbb R^n: x_1>t\})= F_\Phi(t), ~~t\in \mathbb R,
$$
we have
$$
Hf(t)= \frac{1}{(m-1)!} \int_{t}^1 \frac{f(r)}{I_\Phi(r)}\left(\int_{t}^r \frac{\,ds}{I_\Phi(s)}\right)^{m-1}\,dr, ~~t\in (0,1).
$$
It follows from~\cite[proof of Theorem 7.1]{CPS} that $L$ satisfies~\eqref{E:cond_1}. Thus, Proposition~\ref{T:reverse} implies that
\begin{equation}\label{E:lim_H_2}
\lim_{a\to 0_+} \sup_{\|f\|_{X(0,1)}\leq 1} \|H(\chi_{(0,a)}|f|)\|_{Y(0,1)}=0.
\end{equation}
Since $I_\Phi(s) \approx L_\Phi(s)$, $s\in (0,\frac{1}{2}]$, we deduce that $Hf \approx H^m_{L_\Phi}f$ for all $f\in X^{1/2}_+$, up to multiplicative constants independent of $f$. Hence, condition~\eqref{E:lim_H_2} is equivalent to 
\begin{equation}\label{E:hh}
\lim_{a\to 0_+} \sup_{\|f\|_{X(0,1)}\leq 1} \|H^m_{L_\Phi}(\chi_{(0,a)} f)\|_{Y(0,1)}=0,
\end{equation}
which is equivalent to \textup{(ii)} according to Theorem~\ref{T:main}. We have thus proved the equivalence of \textup{(i)} and \textup{(ii)}.

Due to Remark~\ref{T:monotone}, \textup{(ii)} is equivalent to
\begin{equation}\label{E:starstar}
\lim_{a\to 0_+} \sup_{\|f\|_{X(0,1)}\leq 1} \|H^m_{L_\Phi}(\chi_{(0,a)} f^*)\|_{Y(0,1)}=0.
\end{equation}
Further, according to~\cite[Proposition 11.2]{CPS}, 
$$
H^m_{L_\Phi} g \approx P^m_{\Phi} g
$$
is fulfilled for all nonnegative nonincreasing functions $g$, up to multiplicative constants depending on $m$. Therefore, \eqref{E:starstar} holds if and only if
\begin{equation}\label{E:pp}
\lim_{a\to 0_+} \sup_{\|f\|_{X(0,1)}\leq 1} \|P^m_{\Phi}(\chi_{(0,a)} f^*)\|_{Y(0,1)}=0.
\end{equation}
Since for every $g\in \mathcal M(0,1)$, 
\begin{equation}\label{E:ph}
P^m_\Phi g \leq 2^m (m-1)! H^m_{L_\Phi} g,
\end{equation}
see~\cite[Proposition 11.2]{CPS}, we have that for every $a\in (0,1)$,
\begin{align*}
\sup_{\|f\|_{X(0,1)}\leq 1} \|P^m_{\Phi}(\chi_{(0,a)} f^*)\|_{Y(0,1)} 
&\leq \sup_{\|f\|_{X(0,1)}\leq 1} \|P^m_{\Phi}(\chi_{(0,a)} f)\|_{Y(0,1)}\\
&\leq  2^m (m-1)! \sup_{\|f\|_{X(0,1)}\leq 1} \|H^m_{L_\Phi}(\chi_{(0,a)} f)\|_{Y(0,1)}.
\end{align*}
Using this chain of inequalities and the equivalence of~\eqref{E:hh} and~\eqref{E:pp}, we obtain that \textup{(iv)} is equivalent to~\eqref{E:hh}, and therefore also to \textup{(ii)}.

Owing to Theorem~\ref{T:opt_range} and to the fact that $\int_0^1 \frac{\,dr}{L_\Phi(r)}=\infty$, \textup{(ii)} is equivalent to
\begin{equation}\label{E:hhh}
\lim_{a\to 0_+} \sup_{\|f\|_{X(0,1)}\leq 1} \|\chi_{(0,a)}H^m_{L_\Phi}f\|_{Y(0,1)} = 0.
\end{equation}
By~\eqref{E:ph}, condition \eqref{E:hhh} implies \textup{(v)}. Trivially, \textup{(v)} implies \textup{(iv)}, and, thanks to the result of the previous paragraph, \textup{(iv)} implies \textup{(ii)}. Consequently, \textup{(v)} is equivalent to \textup{(ii)}.

Condition \textup{(vi)} is equivalent to \textup{(ii)} owing to Theorem~\ref{T:opt_range} and to the fact that $X^r_{m,\Phi}(0,1)=X^r_{m,L_\Phi}(0,1)$.

The implication \textup{(iii)} $\Rightarrow$ \textup{(iv)} can be proved in the same way as the implication \textup{(i)} $\Rightarrow$ \textup{(ii)} in Lemma~\ref{T:m-1}. We have already proved that \textup{(iv)} implies \textup{(v)}. Let us now show that \textup{(v)} implies \textup{(iii)}.

Since there is no nontrivial function having an absolutely continuous norm in $L^\infty(0,1)$, condition \textup{(v)} yields that $Y(0,1) \neq L^\infty(0,1)$. We claim that, for every $a\in (0,1)$, the operator $\chi_{(a,1)}P^m_\Phi$ is compact from $X(0,1)$ to $Y(0,1)$. To prove it, we first observe that
$$
\chi_{(a,1)}P^m_\Phi f(t)= \chi_{(a,1)}(t) \left(\frac{\Phi^{-1}\left(\log \frac 2t \right)}{\log \frac 2t }\right)^m H^m_s f(t).
$$
A proof analogous to a part of~\cite[proof of Theorem 3.1]{PP} yields that 
$$
\chi_{(a,1)} H^m_s: X(0,1) \rightarrow \rightarrow Y(0,1).
$$
Since the function $s \mapsto \left(\frac{\Phi^{-1}\left(\log \frac 2t \right)}{\log \frac 2t}\right)^m$ is bounded on $(a,1)$, the operator $\chi_{(a,1)}P^m_\Phi$ is compact from $X(0,1)$ to $Y(0,1)$ as well. Consequently, $P^m_\Phi$ is compact from $X(0,1)$ into $Y(0,1)$ as a norm limit of compact operators. Altogether, \textup{(iii)} is equivalent to \textup{(v)}.

Finally, assume that $X(0,1) \neq L^1(0,1)$ and condition~\eqref{E:cond_Y_phi} is satisfied. According to Theorem~\ref{T:opt_domain}, \textup{(ii)} is equivalent to 
\begin{equation}\label{E:hhhh}
\lim_{a\to 0_+} \sup_{\|f\|_{X(0,1)}\leq 1} \sup_{\lambda_1(E)\leq a} \|H^m_{L_\Phi}(\chi_E f)\|_{Y(0,1)} = 0.
\end{equation}
Using~\eqref{E:ph} we deduce that~\eqref{E:hhhh} implies \textup{(vii)}. Trivially, \textup{(vii)} implies \textup{(iv)}. Since we have already shown that \textup{(iv)} is equivalent to \textup{(ii)}, we arrive at the equivalence of \textup{(vii)} and \textup{(ii)}. 

We now claim that the space $Y^d_{m,\Phi}(0,1)$ is the largest rearrangement-invariant space from which the operator $P^m_\Phi$ is bounded into $Y(0,1)$. This fact can be proved in the same way as it is done in the proof of Proposition~\ref{T:proposition} for the rearrangement-invariant space $Y^d_{j,J}(0,1)$ and the operator $H^j_J$. Since boundedness of $P^m_\Phi$ coincides with boundedness of $H^m_{L_\Phi}$ (see~\cite[Proposition 11.3]{CPS}), we obtain that $Y^d_{m,\Phi}(0,1)$ is the optimal domain for $Y(0,1)$ with respect to the operator $H^m_{L_\Phi}$, and therefore, by Proposition~\ref{T:proposition}, $Y^d_{m,\Phi}(0,1)=Y^d_{m,L_\Phi}(0,1)$. Consequently, Theorem~\ref{T:opt_domain} yields that \textup{(viii)} is equivalent to \textup{(ii)}. The proof is complete.
\end{proof}

\section{Examples}\label{S:examples}

In the present section we characterize compact Sobolev embeddings on domains from Maz'ya classes, and on product probability spaces, for some of the customary rearrangement-invariant norms. The case of John domains is not discussed explicitly, however, results for John domains can be derived from corresponding results for Maz'ya classes of domains, by applying the equivalence of the following two conditions: 

\medskip
\textup{(i)} $V^mX(\Omega) \hookrightarrow \hookrightarrow Y(\Omega)$ holds for a given John domain $\Omega$;

\textup{(ii)} $V^mX(\Omega) \hookrightarrow \hookrightarrow Y(\Omega)$ holds for every $\Omega \in \mathcal J_{\frac{1}{n'}}$.

\medskip
\noindent
We recall that this equivalence follows from Theorems~\ref{T:john} and~\ref{T:mazya}. 

In the first part of this section we shall study when the compact Sobolev embedding
\begin{equation}\label{E:embedding_examples}
V^mX(\Omega,\nu) \hookrightarrow \hookrightarrow Y(\Omega,\nu)
\end{equation}
holds, provided that $(\Omega,\nu)$ is either an Euclidean domain belonging to the Maz'ya class $\mathcal J_\alpha$ for some $\alpha \in [\frac{1}{n'},1]$, or a product probability space, and one of the rearrangement-invariant spaces $X(\Omega,\nu)$ and $Y(\Omega,\nu)$ is equal to $L^1(\Omega,\nu)$ or $L^\infty(\Omega,\nu)$ (the largest and the smallest rearrangement-invariant space, respectively). To obtain a tool for dealing with this problem, we characterize for a given nondecreasing function $I$ the validity of condition
\begin{equation}\label{E:lim_examples}
\lim_{a\to 0_+} \sup_{\|f\|_{X(0,1)}\leq 1} \|H^m_I(\chi_{(0,a)}f)\|_{Y(0,1)} = 0 
\end{equation}
in each of the four cases when one of the spaces $X(0,1)$ and $Y(0,1)$ coincides with $L^1(0,1)$ or $L^\infty(0,1)$. We start with the two \lq\lq$L^1$-cases". It turns out that in this situation the operator $H^m_I$ can be replaced by the simpler operator $K^m_I$, without assuming any restrictons on $I$ (compare to Theorem~\ref{T:K}). 

\begin{theorem}\label{T:ll}
Let $m\in \mathbb N$ and let $I: (0,1] \rightarrow (0,\infty)$ be a nondecreasing function satisfying~\eqref{E:infimum}. Suppose that $\|\cdot\|_{X(0,1)}$ is a rearrangement-invariant norm and denote by $\varphi_X$ the fundamental function of $\|\cdot\|_{X(0,1)}$.

\medskip
\noindent
\textup{(a)} The following conditions are equivalent:

\textup{(i)} $\lim_{a\to 0_+} \sup_{\|f\|_{L^1(0,1)}\leq 1} \|H^m_I(\chi_{(0,a)}f)\|_{X(0,1)}=0$;

\textup{(ii)} $\lim_{a\to 0_+} \sup_{\|f\|_{L^1(0,1)}\leq 1} \|K^m_I(\chi_{(0,a)}f)\|_{X(0,1)}=0$;

\textup{(iii)} $\lim_{a\to 0_+} \frac{a^{m-1} \varphi_X(a)}{(I(a))^m} = 0$.

\medskip
\noindent
\textup{(b)} The following conditions are equivalent:

\textup{(i)} $\lim_{a\to 0_+} \sup_{\|f\|_{X(0,1)}\leq 1} \|H^m_I(\chi_{(0,a)}f)\|_{L^1(0,1)}=0$;

\textup{(ii)} $\lim_{a\to 0_+} \sup_{\|f\|_{X(0,1)}\leq 1} \|K^m_I(\chi_{(0,a)}f)\|_{L^1(0,1)}=0$.

\noindent
If $X(0,1) \neq L^1(0,1)$ then both \textup{(i)} and \textup{(ii)} are satisfied. Furthermore, if $X(0,1)=L^1(0,1)$ then \textup{(i)} and \textup{(ii)} hold if and only if
\begin{equation}\label{E:lx2}
\lim_{s\to 0_+} \frac{s}{I(s)} =0.
\end{equation}
\end{theorem}

The following theorem characterizes $m$-th order compact Sobolev embeddings on domains from the Maz'ya class $\mathcal J_\alpha$ in the two \lq\lq$L^1$-cases". It can be obtained by combined using of Theorems~\ref{T:mazya} and~\ref{T:ll} (with $I(s)=s^\alpha$). Let us note that here, and also in all further results on Maz'ya classes of domains later in this section, we assume that $m(1-\alpha)<1$. This can be done with no loss of generality, since the case when $m(1-\alpha)\geq 1$ was sufficiently described in Theorem~\ref{T:mazya}.

\begin{theorem}\label{T:mazya_l1}
Let $n\in \mathbb N$, $n\geq 2$, let $m\in \mathbb N$ and let $\alpha \in [\frac{1}{n'},1]$ satisfy $m(1-\alpha)<1$. Suppose that $\|\cdot\|_{X(0,1)}$ is a rearrangement-invariant norm and denote by $\varphi_X$ the fundamental function of $\|\cdot\|_{X(0,1)}$.

\medskip
\noindent
\textup{(a)} The condition
\begin{equation}\label{E:l^1}
V^mL^1(\Omega) \hookrightarrow \hookrightarrow X(\Omega)
\end{equation}
is satisfied for every $\Omega \in \mathcal J_\alpha$ if and only if 
\begin{equation}\label{E:nec_cond}
\lim_{s\to 0_+} \frac{\varphi_X(s)}{s^{1-m(1-\alpha)}} = 0.
\end{equation}
This is never fulfilled for $\alpha=1$.

\medskip
\noindent
\textup{(b)} Suppose that $X(0,1)\neq L^1(0,1)$. Then the condition
\begin{equation}\label{E:l1}
V^mX(\Omega) \hookrightarrow \hookrightarrow L^1(\Omega)
\end{equation}
is satisfied for every $\Omega \in \mathcal J_\alpha$. Furthermore, if $X(0,1)=L^1(0,1)$ then~\eqref{E:l1} is fulfilled for every $\Omega \in \mathcal J_\alpha$ if and only if $\alpha \in [\frac{1}{n'},1)$.
\end{theorem}

An analogous result for product probability spaces is provided in the next theorem.

\begin{theorem}\label{T:probability_l1}
Let $n$, $m\in \mathbb N$, and let $\Phi$ be as in Section~\ref{S:sobolev}. Suppose that $\|\cdot\|_{X(0,1)}$ is a rearrangement-invariant norm and denote by $\varphi_X$ the fundamental function of $\|\cdot\|_{X(0,1)}$.  

\medskip
\noindent
\textup{(a)} The condition
\begin{equation}\label{E:l^1_gauss}
V^mL^1(\mathbb R^n, \mu_{\Phi,n}) \hookrightarrow \hookrightarrow X(\mathbb R^n, \mu_{\Phi,n})
\end{equation}
is satisfied if and only if 
\begin{equation}\label{E:nec_cond_prob}
\lim_{s\to 0_+} \frac{\varphi_X(s)(\Phi^{-1}(\log \frac{2}{s}))^m}{s(\log \frac{2}{s})^m}=0.
\end{equation}

\medskip
\noindent
\textup{(b)} Suppose that $X(0,1)\neq L^1(0,1)$. Then the condition
\begin{equation}\label{E:l^one}
V^mX(\mathbb R^n, \mu_{\Phi,n}) \hookrightarrow \hookrightarrow L^1(\mathbb R^n, \mu_{\Phi,n})
\end{equation}
is satisfied. Furthermore, if $X(0,1)=L^1(0,1)$ then~\eqref{E:l^one} is fulfilled if and only if 
\begin{equation}\label{E:lim_phi_0}
\lim_{s\to \infty} \frac{s}{\Phi(s)}=0.
\end{equation}
\end{theorem}

Note that Theorem~\ref{T:probability_l1} follows from Theorem~\ref{T:probability}, Theorem~\ref{T:compact_operator} (applied with $J=L_\Phi$ and $j=m$) and Theorem~\ref{T:ll} (applied with $I=L_\Phi$). We also need to use the equivalence
$$
\Phi'(\Phi^{-1}(s)) \approx \frac{s}{\Phi^{-1}(s)}, ~~s>0,
$$
which was proved in~\cite[Lemma 11.1 \textup{(iii)}]{CPS}, and the following chain:
\begin{equation}\label{E:lim_1}
\lim_{s\to 0_+} \frac{\Phi^{-1}(\log \frac{2}{s})}{\log \frac{2}{s}} 
=\lim_{s\to 0_+} \frac{\Phi^{-1}(\log \frac{2}{s})}{\Phi(\Phi^{-1}(\log \frac{2}{s}))}
=\lim_{s\to \infty} \frac{s}{\Phi(s)}.
\end{equation}

\begin{remark}\label{T:phi}
It follows from the convexity of $\Phi$ and from the fact that $\Phi(0)=0$ that the function $s\mapsto \frac{s}{\Phi(s)}$ is nonincreasing on $(0,\infty)$. Hence, $\lim_{s \to \infty} \frac{s}{\Phi(s)}$ exists. In particular, if~\eqref{E:lim_phi_0} is not fulfilled then $\lim_{s \to \infty} \frac{s}{\Phi(s)} \in (0,\infty)$. Combining this with the monotonicity of $\frac{s}{\Phi(s)}$ we obtain that in this situation, $\Phi(s)\approx s$ on $(a,\infty)$ for every $a\in (0,\infty)$, up to multiplicative constants possibly depending on $\Phi$ and $a$.
\end{remark}

Let us now focus on the two \lq\lq$L^\infty$-cases". Similarly is in the \lq\lq$L^1$-cases", we start with a one-dimensional result concerning the validity of condition~\eqref{E:lim_examples}, now with $X(0,1)$ or $Y(0,1)$ equal to $L^\infty(0,1)$. In contrast to Theorem~\ref{T:ll}, in this situation we cannot equivalently replace the operator $H^m_I$ by $K^m_I$ (a counterexample follows from Remarks~\ref{T:remarks_5} \textup{(iv)}). 

We will change the notation from $I$ to $J$ and from $m$ to $j$, and we will not assume any monotonicity of $J$. Then, by setting $J(s)=\frac{I(s)^m}{s^{m-1}}$, $s\in (0,1]$, and $j=1$, our result applies also to the characterization of condition~\eqref{E:lim_examples} with $H^m_I$ replaced by $K^m_I$ (notice that if the function $I$ satisfies~\eqref{E:aprox} then condition \eqref{E:lim_examples} is not affected by replacing $H^m_I$ by $K^m_I$).

\begin{theorem}\label{T:lll}
Let $j\in \mathbb N$ and let $J: (0,1] \rightarrow (0,\infty)$ be a measurable function satisfying~\eqref{E:supremum}. Suppose that $\|\cdot\|_{X(0,1)}$ is a rearrangement-invariant norm.

\medskip
\noindent
\textup{(a)} If $\int_0^1 \frac{\,dr}{J(r)} <\infty$ then the condition
\begin{equation}\label{E:linftyx}
\lim_{a\to 0_+} \sup_{\|f\|_{L^\infty(0,1)}\leq 1} \|H^j_J(\chi_{(0,a)}f)\|_{X(0,1)}=0
\end{equation}
is satisfied for all $j\in \mathbb N$ and for all rearrangement-invariant norms $\|\cdot\|_{X(0,1)}$. In the case that $\int_0^1 \frac{\,dr}{J(r)} =\infty$, condition~\eqref{E:linftyx} holds if and only if
\begin{equation}\label{E:linftyxa}
\lim_{a\to 0_+} \left\|\chi_{(0,a)}(t) \left(\int_t^1 \frac{\,dr}{J(r)}\right)^{j}\right\|_{X(0,1)} = 0.
\end{equation}

\medskip
\noindent
\textup{(b)} The condition
\begin{equation}\label{E:lxinfty}
\lim_{a\to 0_+} \sup_{\|f\|_{X(0,1)}\leq 1} \|H^j_J(\chi_{(0,a)}f)\|_{L^\infty(0,1)}=0
\end{equation}
holds if and only if 
\begin{equation}\label{E:llxinfty}
\lim_{a\to 0_+} \left\|\frac{\chi_{(0,a)}(t)}{J(t)} \left(\int_0^t \frac{\,dr}{J(r)}\right)^{j-1}\right\|_{X'(0,1)} = 0.
\end{equation}
This is never fulfilled in the case that $\int_0^1 \frac{\,dr}{J(r)} =\infty$. 
\end{theorem}

The previous theorem combined with Theorem~\ref{T:mazya} easily leads to the following result on $m$-th order compact Sobolev embeddings on domains from the Maz'ya class $\mathcal J_\alpha$ in the \lq\lq$L^\infty$-cases". We note that Theorem~\ref{T:lll} has to be applied with $j=1$ and $J(s)=s^{1-m(1-\alpha)}$, $s\in (0,1]$, if $\alpha \in [\frac{1}{n'},1)$, and with $j=m$ and $J(s)=s$, $s\in (0,1]$, if $\alpha=1$. 

\begin{theorem}\label{T:mazya_lll}
Let $n\in \mathbb N$, $n\geq 2$, and let $m\in \mathbb N$. Suppose that $\|\cdot\|_{X(0,1)}$ is a rearrangement-invariant norm. 

\medskip
\noindent
\textup{(a)} Assume that $\alpha \in [\frac{1}{n'},1)$ and that $m(1-\alpha)<1$. Then the condition
\begin{equation}\label{E:linfty}
V^mL^\infty(\Omega) \hookrightarrow \hookrightarrow X(\Omega)
\end{equation}
is fulfilled for every $\Omega \in \mathcal J_\alpha$ and for all rearrangement-invariant norms $\|\cdot\|_{X(0,1)}$. Furthermore, 
\begin{equation}\label{E:l^infty}
V^mX(\Omega) \hookrightarrow \hookrightarrow L^\infty(\Omega)
\end{equation}
is satisfied for every $\Omega \in \mathcal J_\alpha$ if and only if
\begin{equation}\label{E:l^infty_2}
\lim_{a\to 0_+} \|\chi_{(0,a)}(s)s^{m(1-\alpha)-1}\|_{X'(0,1)} = 0.
\end{equation}

\medskip
\noindent
\textup{(b)} The condition~\eqref{E:linfty} is satisfied for every $\Omega \in \mathcal J_1$ if and only if
$$
\lim_{a\to 0_+} \left\|\chi_{(0,a)}(s)\left(\log \frac{2}{s} \right)^m\right\|_{X(0,1)} = 0.
$$
Furthermore, there is no rearrangement-invariant norm $\|\cdot\|_{X(0,1)}$ and $m\in \mathbb N$ such that~\eqref{E:l^infty} is fulfilled for every $\Omega \in \mathcal J_1$.
\end{theorem}

An analogous result for the product probability space $(\mathbb R^n, \mu_{\Phi,n})$ can be derived from Theorems~\ref{T:probability} and~\ref{T:lll} (with $J=L_\Phi$ and $j=m$), by making use of the equivalence
$$
\int_s^1 \frac{\,dr}{L_\Phi(r)} = \Phi^{-1}\big(\log \frac{2}{s}\big) -\Phi^{-1}\big(\log 2\big) \approx \Phi^{-1}\big(\log \frac{2}{s}\big), ~~s\in \big(0,\frac{1}{2}\big).
$$ 

\begin{theorem}\label{T:l^1_probability}
Let $n$, $m\in \mathbb N$, and let $\Phi$ be as in Section~\ref{S:sobolev}. Suppose that $\|\cdot\|_{X(0,1)}$ is a rearrangement-invariant norm. 

\medskip
\noindent
\textup{(a)} The condition
\begin{equation}\label{E:l^infty_prob}
V^mL^\infty(\mathbb R^n, \mu_{\Phi,n}) \hookrightarrow \hookrightarrow X(\mathbb R^n, \mu_{\Phi,n})
\end{equation}
is satisfied if and only if
\begin{equation}\label{E:l^infty_3}
\lim_{a\to 0_+} \left\|\chi_{(0,a)}(s) \left(\Phi^{-1}\left(\log \frac{2}{s}\right)\right)^{m}\right\|_{X(0,1)}=0.
\end{equation}

\medskip
\noindent
\textup{(b)} The condition
\begin{equation}\label{E:l^infty_prob2}
V^mX(\mathbb R^n, \mu_{\Phi,n}) \hookrightarrow \hookrightarrow L^\infty(\mathbb R^n, \mu_{\Phi,n})
\end{equation}
is never fulfilled.
\end{theorem}

We shall now study the compact Sobolev embedding~\eqref{E:embedding_examples}, provided that $(\Omega,\nu)$ is either an Euclidean Maz'ya domain, or a product probability space, and both $X(\Omega,\nu)$ and $Y(\Omega,\nu)$ are Lebesgue spaces. We shall consider also the more general situation when both $X(\Omega,\nu)$ and $Y(\Omega,\nu)$ are Lorentz spaces (in the case when $(\Omega,\nu)$ is the Maz'ya domain), or Lorentz-Zygmund spaces (in the case when $(\Omega,\nu)$ is the Boltzmann space, a particular example of product probability spaces). We note that Lorentz spaces in the former case and Lorentz-Zygmund spaces in the latter case naturally arise as optimal targets of Lebesgue spaces in the Sobolev embeddings on the corresponding domains, see~\cite[Theorems 6.9 and 7.12]{CPS}.

The result for Maz'ya classes of domains takes the following form.

\begin{theorem}\label{T:lorentz}
Let $n\in \mathbb N$, $n\geq 2$, let $m\in \mathbb N$ and let $\alpha \in [\frac{1}{n'},1]$ satisfy $m(1-\alpha)< 1$. Suppose that $p_1$, $p_2$, $q_1$, $q_2\in [1,\infty]$ are such that both triplets $(p,q,\alpha)=(p_1,q_1,0)$ and $(p,q,\alpha)=(p_2,q_2,0)$ satisfy one of the conditions~\eqref{E:L-Z_r.i.1} -- \eqref{E:L-Z_r.i.4}. Then the following assertions are equivalent.

\textup{(i)} The compact embedding
$$
V^mL^{p_1,q_1}(\Omega) \hookrightarrow \hookrightarrow L^{p_2,q_2}(\Omega)
$$
holds for every $\Omega \in \mathcal J_\alpha$.

\textup{(ii)} The compact embedding
$$
V^mL^{p_1}(\Omega) \hookrightarrow \hookrightarrow L^{p_2}(\Omega)
$$
holds for every $\Omega \in \mathcal J_\alpha$.

\textup{(iii)} One of the following conditions is satisfied:
\begin{align}
&\alpha \in [1/n',1), ~~ p_1<\frac{1}{m(1-\alpha)}, ~~ p_2<\frac{p_1}{1-mp_1(1-\alpha)}; \label{E:lorentz_1}\\
&\alpha \in [1/n',1), ~~ p_1=\frac{1}{m(1-\alpha)}, ~~ p_2<\infty; \label{E:lorentz_2}\\
&\alpha \in [1/n',1), ~~ p_1>\frac{1}{m(1-\alpha)}; \label{E:lorentz_3}\\
&\alpha=1, ~~p_1>p_2.\label{E:lorentz_4}
\end{align}
\end{theorem} 

We now focus on compact Sobolev embeddings in context of Lebesgue spaces over product probability spaces. Interestingly, we can often speak about optimal compact embeddings in this connection.

\begin{theorem}\label{T:lebesgue}
Let $n$, $m\in \mathbb N$, let $\Phi$ be as in Section~\ref{S:sobolev}, and let $p$, $q\in [1,\infty]$. 

\textup{(i)} Suppose that $\lim_{s\to \infty} \frac{s}{\Phi(s)}=0$. Then
\begin{equation}\label{E:comp_emb_prob_leb}
V^mL^p(\mathbb R^n, \mu_{\Phi,n}) \hookrightarrow \hookrightarrow L^q(\mathbb R^n,\mu_{\Phi,n})
\end{equation}
holds if and only if $q\leq p$ and $q<\infty$. In particular, if $p<\infty$ then $L^p(\mathbb R^n,\mu_{\Phi,n})$ is the optimal (i.e., the smallest) Lebesgue space into which $V^mL^p(\mathbb R^n, \mu_{\Phi,n})$ is compactly embedded.

\textup{(ii)} Suppose that $\lim_{s\to \infty} \frac{s}{\Phi(s)} \in (0,\infty)$. Then~\eqref{E:comp_emb_prob_leb} holds if and only if $q<p$.

\end{theorem}

Notice that, according to Remark~\ref{T:phi}, parts \textup{(i)} and \textup{(ii)} of Theorem~\ref{T:lebesgue} indeed cover all cases of the function $\Phi$.

\medskip

The optimality in compact embeddings disappears when more general rearrangement-invariant spaces are called into play. This easily follows from the next result, in which we consider Lorentz-Zygmund spaces over the particular family of product probability spaces consisting of all Boltzmann spaces.

\begin{theorem}\label{T:lz_gauss}
Let $n$, $m\in \mathbb N$ and $\beta\in [1,2]$. Furthermore, let $p_1$, $p_2$, $q_1$, $q_2\in [1,\infty]$, $\alpha_1$, $\alpha_2\in \mathbb R$ be such that both triplets $(p,q,\alpha)=(p_1,q_1,\alpha_1)$ and $(p,q,\alpha)=(p_2,q_2,\alpha_2)$ satisfy one of the conditions~\eqref{E:L-Z_r.i.1} -- \eqref{E:L-Z_r.i.4}.

\textup{(i)} Suppose that $p_1<\infty$. Then
\begin{equation}\label{E:lz_comp}
V^mL^{p_1,q_1;\alpha_1}(\mathbb R^n, \gamma_{n,\beta}) \hookrightarrow \hookrightarrow L^{p_2,q_2;\alpha_2}(\mathbb R^n, \gamma_{n,\beta})
\end{equation}
holds if and only if $p_1>p_2$, or $p_1=p_2$ and one of the following conditions is satisfied:
\begin{align*}
&q_1\leq q_2, ~~ \alpha_1+\frac{m(\beta-1)}{\beta} >\alpha_2;\\
&q_2<q_1, ~~ \alpha_1 + \frac{1}{q_1} +\frac{m(\beta-1)}{\beta} >\alpha_2 +\frac{1}{q_2}.
\end{align*} 

\textup{(ii)} Suppose that $p_1=\infty$. Then~\eqref{E:lz_comp} holds if and only if $p_2<\infty$, or 
$$
p_2=\infty, ~~\alpha_1 + \frac{1}{q_1} -\frac{m}{\beta} >\alpha_2 +\frac{1}{q_2}.
$$
\end{theorem}

We finish the paper by proving those results of this section which have not been verified yet. We need the following auxiliary lemma.

\begin{lemma}\label{T:nondec}
Suppose that $m\in \mathbb N$ and $I:(0,1] \rightarrow (0,\infty)$ is a nondecreasing function fulfilling~\eqref{E:infimum}. 
Then for every $f\in \mathcal M(0,1)$ and $a\in (0,1)$,
\begin{equation}\label{E:r_m_i}
\sup_{s\in (0,a)} R^m_If^*(s) \approx \sup_{s\in (0,a)} S^m_If^*(s),
\end{equation}
up to multiplicative constants depending on $m$. 
\end{lemma}

\begin{proof}
If $m=1$ then~\eqref{E:r_m_i} trivially holds since $R^1_I=S^1_I$. Thus, in what follows we may assume that $m\geq 2$.

Fix $a\in (0,1)$. Given $f\in \mathcal M(0,1)$ and $s\in (0,a)$, we have
$$
R_I f^*(s)= \frac{s}{I(s)} f^{**}(s) \leq (f^{**}(s))^{\frac{m-1}{m}} \sup_{t\in (0,a)} \frac{t}{I(t)} (f^{**}(t))^{\frac{1}{m}}.
$$
Therefore,
\begin{equation}\label{E:1}
R^m_If^*(s) = R^{m-1}_I (R_I f^*)(s) \leq R^{m-1}_I (f^{**})^{\frac{m-1}{m}}(s) \sup_{t\in (0,a)} \frac{t}{I(t)}(f^{**}(t))^{\frac{1}{m}}.
\end{equation}
Furthermore, let $k\in \{1,2,\dots, m-1\}$. Then
\begin{align*}
R_I (f^{**})^{\frac{k}{m}} (s) 
&=\frac{1}{I(s)} \int_0^s \left(\frac{\int_0^r f^*(u)\,du}{r}\right)^{\frac{k}{m}}\,dr
\leq \frac{1}{I(s)} \left(\int_0^s f^*(u)\,du\right)^{\frac{k}{m}} \int_0^s r^{-\frac{k}{m}}\,dr\\
&=\frac{m}{m-k} \cdot \frac{s^{1-\frac{k}{m}}}{I(s)} \left(\int_0^s f^*(u)\,du\right)^{\frac{k}{m}}
=\frac{m}{m-k} \cdot \frac{s}{I(s)} (f^{**}(s))^{\frac{k}{m}}\\
&\leq \frac{m}{m-k} (f^{**}(s))^{\frac{k-1}{m}} \sup_{t\in (0,a)} \frac{t}{I(t)}(f^{**}(t))^{\frac{1}{m}}.
\end{align*}
Hence,
\begin{equation}\label{E:2}
R^k_I(f^{**})^{\frac{k}{m}}(s)= R^{k-1}_I(R_I(f^{**})^{\frac{k}{m}})(s)
\leq \frac{m}{m-k} R^{k-1}_I (f^{**})^{\frac{k-1}{m}}(s) \sup_{t\in (0,a)} \frac{t}{I(t)}(f^{**}(t))^{\frac{1}{m}}.
\end{equation}
Using~\eqref{E:1} and~\eqref{E:2} subsequently for $k=m-1, m-2,\dots,1$, we obtain
\begin{align*}
R^m_If^*(s) 
&\leq R^{m-1}_I (f^{**})^{\frac{m-1}{m}}(s) \sup_{t\in (0,a)} \frac{t}{I(t)}(f^{**}(t))^{\frac{1}{m}}\\
&\leq m R^{m-2}_I (f^{**})^{\frac{m-2}{m}}(s) \left(\sup_{t\in (0,a)} \frac{t}{I(t)}(f^{**}(t))^{\frac{1}{m}}\right)^2\\
&\leq \dots \leq \frac{m^{m-2}}{(m-2)!} R_I (f^{**})^{\frac{1}{m}}(s) \left(\sup_{t\in (0,a)} \frac{t}{I(t)}(f^{**}(t))^{\frac{1}{m}}\right)^{m-1}\\
&\leq \frac{m^{m-1}}{(m-1)!}\left(\sup_{t\in (0,a)} \frac{t}{I(t)}(f^{**}(t))^{\frac{1}{m}}\right)^{m}\\
&=\frac{m^{m-1}}{(m-1)!} \sup_{t\in (0,a)} \frac{t^{m-1}}{(I(t))^m} \int_0^t f^*(r)\,dr = \frac{m^{m-1}}{(m-1)!} \sup_{t\in (0,a)} S^m_If^*(s).
\end{align*}
Passing to supremum over all $s\in (0,a)$, we get
\begin{equation}\label{E:3}
\sup_{s\in (0,a)} R^m_If^*(s) 
\leq \frac{m^{m-1}}{(m-1)!} \sup_{s\in (0,a)} S^m_If^*(s).
\end{equation}

Conversely, given $s\in (0,1)$, we have by the monotonicity of $I$
\begin{align*}
(m-1)!R^m_If^*(s)
&=\frac{1}{I(s)} \int_0^s \left( \int_t^s \frac{\,dr}{I(r)}\right)^{m-1} f^*(t) \,dt 
\geq \frac{1}{I(s)} \left(\int_{\frac{s}{2}}^s \frac{\,dr}{I(r)}\right)^{m-1} \int_0^{\frac{s}{2}} f^*(t)\,dt\\
&\geq \frac{1}{I(s)} \left(\frac{s}{2I(s)}\right)^{m-1} \frac{1}{2} \int_0^s f^*(t)\,dt\\
&=\frac{1}{2^m} \cdot \frac{s^{m-1}}{(I(s))^m} \int_0^s f^*(t)\,dt = \frac{1}{2^m} S^m_If^*(s).
\end{align*}
This yields the inequality in the reverse direction to~\eqref{E:3}. The proof is complete.
\end{proof}

\begin{proof}[Proof of Theorem~\ref{T:ll}]
\textup{(a)} Using the definition of the associate norm and the equation~\eqref{E:associate} (the first time with $j=m$ and $J=I$ and the second time with $j=1$ and $J$ as in~\eqref{E:J}), similarly as in the proof of Theorem~\ref{T:opt_range}, part \textup{(ii)} $\Leftrightarrow$ \textup{(iii)}, and applying Lemma~\ref{T:nondec}, we obtain that for every $a\in (0,1)$
\begin{align}\label{E:rs}
\sup_{\|g\|_{L^1(0,1)}\leq 1} \|H^m_I(\chi_{(0,a)}g)\|_{X(0,1)}
&=\sup_{\|f\|_{X'(0,1)}\leq 1} \left\|\chi_{(0,a)}R^m_If^*\right\|_{L^\infty(0,1)}\\ \nonumber
&=\sup_{\|f\|_{X'(0,1)}\leq 1} \sup_{s\in (0,a)} R^m_If^*(s)
\approx \sup_{\|f\|_{X'(0,1)}\leq 1} \sup_{s\in (0,a)} S^m_If^*(s)\\ \nonumber
&=\sup_{\|f\|_{X'(0,1)}\leq 1} \left\|\chi_{(0,a)}S^m_If^*\right\|_{L^\infty(0,1)}\\ \nonumber
&=\sup_{\|g\|_{L^1(0,1)}\leq 1} \|K^m_I(\chi_{(0,a)}g)\|_{X(0,1)},
\end{align}
up to multiplicative constants depending on $m$. Notice that we are also using that for every $a\in (0,1)$ and every $f\in X'(0,1)$, we have
\begin{equation}\label{E:111}
\esup_{t\in (0,a)} R^m_If^*(t)= \sup_{t\in (0,a)} R^m_If^*(t)
\end{equation}
and
\begin{equation}\label{E:222}
\esup_{t\in (0,a)} S^m_If^*(t)= \sup_{t\in (0,a)} S^m_If^*(t).
\end{equation}
The argument which justifies~\eqref{E:111} and~\eqref{E:222} is the same as the one appearing in the proof of Theorem~\ref{T:K}.
Therefore, we have proved the equivalence of \textup{(i)} and \textup{(ii)}. Furthermore, we have
\begin{align}\label{E:s}
\sup_{\|f\|_{X'(0,1)}\leq 1} \sup_{s\in (0,a)} S^m_If^*(s)
&=\sup_{s\in (0,a)} \sup_{\|f\|_{X'(0,1)}\leq 1} \frac{s^{m-1}}{(I(s))^m} \int_0^s f^*(t)\,dt\\ \nonumber
&=\sup_{s\in (0,a)} \frac{s^{m-1}\|\chi_{(0,s)}\|_{X(0,1)}}{(I(s))^m}=\sup_{s\in (0,a)} \frac{s^{m-1}\varphi_X(s)}{(I(s))^m}.
\end{align}
By combined using of~\eqref{E:rs} and~\eqref{E:s}, we obtain that conditions \textup{(i)} and \textup{(ii)} are equivalent to \textup{(iii)}.

\textup{(b)} Suppose that $X(0,1) \neq L^1(0,1)$. Let $J$ be a positive measurable function on $(0,1]$ fulfilling~\eqref{E:supremum} and let $j\in \mathbb N$. We will show that the condition
\begin{equation}\label{E:k}
\lim_{a\to 0_+} \sup_{\|f\|_{X(0,1)}\leq 1} \|H^j_J(\chi_{(0,a)}f)\|_{L^1(0,1)}=0
\end{equation}
is satisfied. Since both $H^m_I$ and $K^m_I$ have the form $H^j_J$ for a suitable choice of $j$ and $J$, we obtain that \textup{(i)} and \textup{(ii)} are satisfied as well (and, in particular, that they are equivalent).

Let us now prove~\eqref{E:k}. We have
$$
H^j_J: L^1(0,1) \rightarrow (L^1)^r_{j,J}(0,1) \hookrightarrow L^1(0,1).
$$
According to Remark~\ref{T:remark_2}, we obtain
$$
H^j_J: X(0,1) \rightarrow \rightarrow L^1(0,1).
$$
By Theorem~\ref{T:compact_operator}, this implies~\eqref{E:k}.

Furthermore, let $X(0,1)=L^1(0,1)$. Using the part \textup{(a)} with $X(0,1)=L^1(0,1)$, we obtain that \textup{(i)} and \textup{(ii)} are equivalent and that they are satisfied if and only if
$$
\lim_{s\to 0_+} \frac{s^{m-1}\varphi_{L^1}(s)}{(I(s))^m}= \lim_{s\to 0_+} \left(\frac{s}{I(s)}\right)^m=0.
$$
This is equivalent to~\eqref{E:lx2}, as required.
\end{proof}

\begin{proof}[Proof of Theorem~\ref{T:lll}]
\textup{(a)} Using the fact that each function $f$ belonging to the unit ball of $L^\infty(0,1)$ satisfies $|f|\leq 1$ and applying the equality~\eqref{E:char}, we obtain
\begin{align*}
\lim_{a\to 0_+} &\sup_{\|f\|_{L^\infty(0,1)}\leq 1} \left\|H^j_J(\chi_{(0,a)}f)\right\|_{X(0,1)} 
=\lim_{a\to 0_+} \left\|H^j_J(\chi_{(0,a)})\right\|_{X(0,1)}\\
&=\frac{1}{j!} \lim_{a\to 0_+} \left\|\chi_{(0,a)}(t) \left(\int_t^a \frac{\,dr}{J(r)}\right)^{j}\right\|_{X(0,1)}.
\end{align*}

Suppose that $\int_0^1 \frac{\,dr}{J(r)} <\infty$. Then
\begin{align*}
\lim_{a\to 0_+} \left\|\chi_{(0,a)}(t) \left(\int_t^a \frac{\,dr}{J(r)}\right)^{j}\right\|_{X(0,1)}
&\leq \lim_{a\to 0_+} \left(\int_0^a \frac{\,dr}{J(r)}\right)^j \|\chi_{(0,a)}\|_{X(0,1)} \\
&\leq \|1\|_{X(0,1)} \lim_{a\to 0_+} \left(\int_0^a \frac{\,dr}{J(r)}\right)^j=0,
\end{align*}
thanks to the absolute continuity of the Lebesgue integral. Thus, condition~\eqref{E:linftyx} is satisfied for all rearrangement-invariant norms $\|\cdot\|_{X(0,1)}$.

Assume that $\int_0^1 \frac{\,dr}{J(r)}=\infty$. Obviously, condition~\eqref{E:linftyxa} implies
\begin{equation}\label{E:000}
\lim_{a\to 0_+} \left\|\chi_{(0,a)}(t) \left(\int_t^a \frac{\,dr}{J(r)}\right)^{j}\right\|_{X(0,1)}=0,
\end{equation}
and therefore also~\eqref{E:linftyx}. Conversely, if~\eqref{E:linftyx} is fulfilled then, owing to Lemma~\ref{T:m-1}, we have~\eqref{E:linftyxa}.

\textup{(b)} We have
\begin{align*}
\lim_{a\to 0_+} &\sup_{\|f\|_{X(0,1)}\leq 1} \|H^j_J(\chi_{(0,a)}f)\|_{L^\infty(0,1)}\\
&=\lim_{a\to 0_+} \sup_{\|f\|_{X(0,1)}\leq 1} \int_0^a \frac{|f(s)|}{J(s)} \left(\int_0^s \frac{\,dr}{J(r)}\right)^{j-1} \,ds
=\lim_{a\to 0_+} \left\|\frac{\chi_{(0,a)}(s)}{J(s)} \left(\int_0^s \frac{\,dr}{J(r)}\right)^{j-1}\right\|_{X'(0,1)}.
\end{align*}
This yields the equivalence of~\eqref{E:lxinfty} and~\eqref{E:llxinfty}. Further, condition~\eqref{E:llxinfty} is obviously never satisfied in the case that $\int_0^1 \frac{\,dr}{J(r)} =\infty$.
\end{proof}

In order to prove Theorems~\ref{T:lorentz} and~\ref{T:lz_gauss} we shall need the following proposition which characterizes almost-compact embeddings between Lorentz-Zygmund spaces. It can be derived as a particular case of~\cite{K} where almost-compact embeddings between more general classical and weak Lorentz spaces were studied. For the sake of completeness we also give an alternative proof.

\begin{proposition}\label{T:Lorentz-Zygmund}
Let $p_1$, $p_2$, $q_1$, $q_2 \in [1,\infty]$, $\alpha_1$, $\alpha_2 \in \mathbb R$ be such that both triplets $(p,q,\alpha)=(p_1,q_1,\alpha_1)$ and $(p,q,\alpha)=(p_2,q_2,\alpha_2)$ satisfy one of the conditions~\eqref{E:L-Z_r.i.1} -- \eqref{E:L-Z_r.i.4}. Then
\begin{equation}\label{E:ac_lz}
L^{p_1,q_1;\alpha_1}(0,1) \overset{*}{\hookrightarrow} L^{p_2,q_2;\alpha_2}(0,1)
\end{equation}
holds if and only if $p_1 > p_2$, or $p_1=p_2$ and the following conditions are satisfied:
\begin{align}\label{E:lz_cond_1}
&\hbox{if } p_1=p_2<\infty \hbox{ and } q_1 \leq q_2 \hbox{ then } \alpha_1 >\alpha_2;\\
&\hbox{if } p_1=p_2=\infty \hbox{ or } q_1 > q_2 \hbox { then } \alpha_1 + \frac{1}{q_1} > \alpha_2 +\frac{1}{q_2}.\label{E:lz_cond_2}
\end{align}
In particular, if $p_1$, $p_2$, $q_1$, $q_2\in [1,\infty]$ are such that both triplets $(p,q,\alpha)=(p_1,q_1,0)$ and $(p,q,\alpha)=(p_2,q_2,0)$ satisfy one of the conditions~\eqref{E:L-Z_r.i.1} -- \eqref{E:L-Z_r.i.4} then 
\begin{equation}\label{E:lorentz}
L^{p_1,q_1}(0,1) \overset{*}{\hookrightarrow} L^{p_2,q_2}(0,1)
\end{equation}
holds if and only if $p_1>p_2$. 
\end{proposition}

\begin{proof}
Suppose that either $p_1>p_2$, or $p_1=p_2$ and conditions~\eqref{E:lz_cond_1} and~\eqref{E:lz_cond_2} are satisfied. Then we can find $\varepsilon>0$ such that $\|\cdot\|_{L^{p_2,q_2;\alpha_2+\varepsilon}(0,1)}$ is equivalent to a rearrangement-invariant norm and, if $p_1=p_2$, then one of the conditions
~\eqref{E:lz_embedding} is fulfilled with $\alpha_2+\varepsilon$ in place of $\alpha_2$. Therefore, we have
$$
L^{p_1,q_1;\alpha_1}(0,1) \hookrightarrow L^{p_2,q_2;\alpha_2+\varepsilon}(0,1).
$$
Consequently, 
\begin{align*}
&\lim_{a\to 0_+} \sup_{\|f\|_{L^{p_1,q_1;\alpha_1}(0,1)}\leq 1} \|\chi_{(0,a)}f^*\|_{L^{p_2,q_2;\alpha_2}(0,1)} \\
&=\lim_{a\to 0_+} \sup_{\|f\|_{L^{p_1,q_1;\alpha_1}(0,1)}\leq 1} \left\|\chi_{(0,a)}(s)f^*(s) \left(\log  \frac 2s \right)^{-\varepsilon}s^{\frac{1}{p_2}-\frac{1}{q_2}}\left(\log \frac 2s \right)^{\alpha_2+\varepsilon}\right\|_{L^{q_2}(0,1)}\\
&\leq \lim_{a\to 0_+} \left\|\chi_{(0,a)}(s)\left(\log \frac 2s \right)^{-\varepsilon}\right\|_{L^\infty(0,1)} \sup_{\|f\|_{L^{p_1,q_1;\alpha_1}(0,1)}\leq 1} \left\|f^*(s)s^{\frac{1}{p_2}-\frac{1}{q_2}}\left(\log \frac 2s \right)^{\alpha_2+\varepsilon}\right\|_{L^{q_2}(0,1)}\\
&=\sup_{\|f\|_{L^{p_1,q_1;\alpha_1}(0,1)}\leq 1} \|f\|_{L^{p_2,q_2;\alpha_2+\varepsilon}(0,1)} \lim_{a\to 0_+} \left(\log \frac 2a \right)^{-\varepsilon}=0, 
\end{align*}
i.e.,~\eqref{E:ac_lz} is satisfied.

Conversely, suppose that~\eqref{E:ac_lz} is in progress. Then, in particular,
$$
L^{p_1,q_1;\alpha_1}(0,1) \hookrightarrow L^{p_2,q_2;\alpha_2}(0,1),
$$
so either $p_1>p_2$, or $p_1=p_2$ and one of the conditions in~\eqref{E:lz_embedding} is satisfied. Assume that $p_1=p_2$ and denote $p=p_1=p_2$. There are three cases in which both $\|\cdot\|_{L^{p,q_1;\alpha_1}(0,1)}$ and $\|\cdot\|_{L^{p,q_2;\alpha_2}(0,1)}$ are equivalent to rearrangement-invariant norms, one of the conditions in~\eqref{E:lz_embedding} is fulfilled but~\eqref{E:lz_cond_1} or~\eqref{E:lz_cond_2} not. The first one is 
\begin{equation}\label{E:first}
p<\infty, ~~ q_1\leq q_2, ~~\alpha_1=\alpha_2,
\end{equation}
the second one is
\begin{equation}\label{E:second}
p=\infty,~~ q_1\leq q_2, ~~\alpha_1 +\frac{1}{q_1}=\alpha_2+\frac{1}{q_2}<0,
\end{equation}
and the third one is
\begin{equation}\label{E:third}
p=\infty,~~q_1= q_2=\infty, ~~\alpha_1=\alpha_2=0.
\end{equation}
Using~\cite[proof of Theorem 6.3]{OP} we get that in all cases, fundamental functions of $\|\cdot\|_{L^{p,q_1;\alpha_1}(0,1)}$ and $\|\cdot\|_{L^{p,q_2;\alpha_2}(0,1)}$ are equivalent up to multiplicative constants. 
Therefore, a necessary condition for~\eqref{E:ac_lz} to be true,
\begin{equation}\label{E:fund_function_lz}
\lim_{s\to 0_+} \frac{\varphi_{L^{p,q_2;\alpha_2}}(s)}{\varphi_{L^{p,q_1;\alpha_1}}(s)}=0,
\end{equation}
is not satisfied. Hence, whenever $p_1=p_2$ and~\eqref{E:ac_lz} is fulfilled then both~\eqref{E:lz_cond_1} and~\eqref{E:lz_cond_2} hold.


\end{proof}

\begin{proof}[Proof of Theorem~\ref{T:lorentz}]
Let $\alpha \in [\frac{1}{n'},1)$. According to Theorem~\ref{T:mazya}, \textup{(i)} holds if and only if
\begin{equation}\label{E:operator}
T^m_\alpha=H_{s^{1-m(1-\alpha)}}: L^{p_1,q_1}(0,1) \rightarrow \rightarrow L^{p_2,q_2}(0,1).
\end{equation}

First, suppose that $L^{p_2,q_2}(0,1) = L^\infty(0,1)$. It follows from the last part of the proof of Theorem~\ref{T:compact_operator} that~\eqref{E:operator} is not fulfilled with $L^{p_1,q_1}(0,1)=L^1(0,1)$. Assume that $L^{p_1,q_1}(0,1) \neq L^1(0,1)$. Then, according to Theorem~\ref{T:opt_domain}, \eqref{E:operator} is satisfied if and only if
\begin{equation}\label{E:domain_mazya}
L^{p_1,q_1}(0,1) \overset{*}{\hookrightarrow} (L^\infty)^d_{1,s^{1-m(1-\alpha)}}(0,1).
\end{equation}
Due to Lemma~\ref{T:lemma_infty},
$$
\|f\|_{(L^\infty)^d_{1,s^{1-m(1-\alpha)}}(0,1)} \approx \int_0^1 f^*(s)s^{m(1-\alpha)-1}\,ds =\|f\|_{L^{\frac{1}{m(1-\alpha)},1}(0,1)},
$$
up to multiplicative constants independent of $f\in \mathcal M(0,1)$. Hence,
$$
(L^\infty)^d_{1,s^{1-m(1-\alpha)}}(0,1)=L^{\frac{1}{m(1-\alpha)},1}(0,1).
$$
Consequently, by Proposition~\ref{T:Lorentz-Zygmund}, \eqref{E:operator} holds with $L^{p_2,q_2}(0,1)=L^\infty(0,1)$ if and only if $p_1 >\frac{1}{m(1-\alpha)}$. Observe that in this situation, condition~\eqref{E:operator} is fulfilled for all pairs $(p_2,q_2)$ satisfying the assumptions of Theorem~\ref{T:lorentz}, since we always have $L^\infty(0,1) \hookrightarrow L^{p_2,q_2}(0,1)$.

Thus, in what follows we may assume that $L^{p_2,q_2}(0,1) \neq L^\infty(0,1)$ and $p_1 \leq \frac{1}{m(1-\alpha)}$. Due to Theorem~\ref{T:opt_range}, \eqref{E:operator} is satisfied if and only if
\begin{equation}\label{E:alpha}
(L^{p_1,q_1})^r_{1,s^{1-m(1-\alpha)}}(0,1) \overset{*}{\hookrightarrow} L^{p_2,q_2}(0,1).
\end{equation}
It follows from~\cite[Theorem 6.9]{CPS} that 
\begin{equation}\label{E:cases}
(L^{p_1,q_1})^r_{1,s^{1-m(1-\alpha)}}(0,1)=
\begin{cases}
L^{\frac{p_1}{1-mp_1(1-\alpha)},q_1}(0,1),\, &\hbox{if } p_1<\frac{1}{m(1-\alpha)};\\
L^{\infty,q_1;-1}(0,1),\, &\hbox{if } p_1=\frac{1}{m(1-\alpha)} \hbox{ and } q_1>1;\\
L^\infty(0,1),\, &\hbox{if } p_1=\frac{1}{m(1-\alpha)} \hbox{ and } q_1=1.
\end{cases}
\end{equation}
Thus, if $p_1<\frac{1}{m(1-\alpha)}$ then~\eqref{E:alpha} is fulfilled if and only if $p_2 <\frac{p_1}{1-mp_1(1-\alpha)}$, see Proposition~\ref{T:Lorentz-Zygmund}. In the case when $p_1 =\frac{1}{m(1-\alpha)}$, \eqref{E:alpha} is characterized by $p_2<\infty$. Indeed, observe that there is no Lorentz space $L^{p_2,q_2}(0,1)$ different from $L^\infty(0,1)$, having the first index equal to $\infty$ and satisfying one of the conditions~\eqref{E:L-Z_r.i.1} -- \eqref{E:L-Z_r.i.4} with $p=p_2$, $q=q_2$ and $\alpha=0$ at the same time. On the other hand, if $p_1 =\frac{1}{m(1-\alpha)}$ and $p_2<\infty$ then~\eqref{E:alpha} is satisfied according to~\eqref{E:cases} and Proposition~\ref{T:Lorentz-Zygmund}.

Let $\alpha=1$. According to Theorem~\ref{T:mazya}, \textup{(i)} holds if and only if
\begin{equation}\label{E:operator_1}
T^m_1=H^m_s: L^{p_1,q_1}(0,1) \rightarrow \rightarrow L^{p_2,q_2}(0,1).
\end{equation}
First, suppose that $L^{p_1,q_1}(0,1)\neq L^\infty(0,1)$. Then, due to Theorem~\ref{T:opt_range} and~\cite[Theorem 6.11]{CPS}, \eqref{E:operator_1} is satisfied if and only if 
$$
(L^{p_1,q_1})^r_{m,s}(0,1)=L^{p_1,q_1}(0,1) \overset{*}{\hookrightarrow} L^{p_2,q_2}(0,1),
$$
which is equivalent to $p_2<p_1$, see Proposition~\ref{T:Lorentz-Zygmund}. Finally, \eqref{E:operator_1} is satisfied with $L^{p_1,q_1}(0,1)=L^\infty(0,1)$ if and only if  
\begin{equation}\label{E:pqm}
(L^\infty)^r_{m,s}(0,1)=L^{\infty,\infty;-m}(0,1) \overset{*}{\hookrightarrow} L^{p_2,q_2}(0,1).
\end{equation}
As observed above, this cannot be fulfilled when $p_2=\infty$. Furthermore, owing to Proposition~\ref{T:Lorentz-Zygmund}, \eqref{E:pqm} is satisfied if $p_2<\infty$.  

By applying the equivalence of \textup{(i)} and \textup{(iii)} to the particular case when $p_1=q_1$ and $p_2=q_2$, we obtain that \textup{(ii)} is equivalent to \textup{(iii)} as well. The proof is complete.
\end{proof}

\begin{proof}[Proof of Theorem~\ref{T:lebesgue}]
Suppose that $\lim_{s\to \infty} \frac{s}{\Phi(s)}=0$. By Theorem~\ref{T:probability}, condition~\eqref{E:comp_emb_prob_leb} is fulfilled if and only if
\begin{equation}\label{E:pmphi}
(L^p)^r_{m,\Phi}(0,1) \overset{*}{\hookrightarrow} L^q(0,1).
\end{equation}

Let $p\in [1,\infty)$. Since, due to~\eqref{E:lim_1},
$$
\lim_{s\to 0_+} \left(\frac{\Phi^{-1}(\log \frac 2s)}{\log \frac 2s}\right)^m=0,
$$
we have
\begin{align*}
\lim_{a\to 0_+} &\sup_{\|f\|_{(L^p)^r_{m,\Phi}(0,1)}\leq 1} \left\|\chi_{(0,a)}f^*\right\|_{L^p(0,1)}
=\lim_{a\to 0_+} \sup_{\left\|\left(\frac{\log \frac 2t}{\Phi^{-1}(\log \frac 2t)}\right)^m f^*(t)\right\|_{L^p(0,1)}\leq 1} \left\|\chi_{(0,a)}f^*\right\|_{L^p(0,1)}\\
&\leq \lim_{a\to 0_+} \sup_{\left\|\left(\frac{\log \frac 2t}{\Phi^{-1}(\log \frac 2t)}\right)^m f^*(t)\right\|_{L^p(0,1)}\leq 1}
\left\|\left(\frac{\log \frac 2t}{\Phi^{-1}(\log \frac 2t)}\right)^m f^*(t)\right\|_{L^p(0,1)} \sup_{t\in (0,a)} \left(\frac{\Phi^{-1}(\log \frac 2t)}{\log \frac 2t}\right)^m\\
&\leq \lim_{a\to 0_+} \sup_{t\in (0,a)} \left(\frac{\Phi^{-1}(\log \frac 2t)}{\log \frac 2t}\right)^m=0,
\end{align*}
which yields that $(L^p)^r_{m,\Phi}(0,1) \overset{*}{\hookrightarrow} L^p(0,1)$.

Suppose that $q\leq p$. Then $L^p(0,1) \hookrightarrow L^q(0,1)$, and therefore \eqref{E:pmphi} is satisfied. Conversely, assume that $q>p$. Since $\sqrt{\Phi}$ is concave on $[0,\infty)$ and $\sqrt{\Phi(0)}=0$, we deduce that the function $t\mapsto \frac{\sqrt{\Phi(t)}}{t}$ is nonincreasing on $(0,\infty)$. Using that $\Phi^{-1}$ is nondecreasing on $(0,\infty)$, we obtain that the function
$$
s\mapsto \frac{\sqrt{\Phi(\Phi^{-1}(s))}}{\Phi^{-1}(s)}=\frac{\sqrt{s}}{\Phi^{-1}(s)}
$$
is nonincreasing on $(0,\infty)$. Therefore, 
\begin{align*}
\lim_{a\to 0_+} \frac{\varphi_{L^q}(a)}{\varphi_{(L^p)^r_{m,\Phi}}(a)}
&=\lim_{a\to 0_+} \frac{a^{\frac 1q}}{\left\|\chi_{(0,a)} \left(\frac{\log \frac 2t}{\Phi^{-1}(\log \frac 2t)}\right)^m\right\|_{L^p(0,1)}}
=\lim_{a\to 0_+} \frac{a^{\frac 1q}}{\left( \int_0^a \left( \frac{\log \frac 2t}{\Phi^{-1}(\log \frac 2t)}\right)^{mp}\, dt \right)^{\frac 1p}}\\
&\geq \lim_{a\to 0_+} \frac{a^{\frac 1q}}{\left( \frac{\sqrt{\log \frac 2a}}{\Phi^{-1}(\log \frac 2a)}\right)^{m} \left(\int_0^a \left(\sqrt{\log \frac 2t}\right)^{mp}\,dt\right)^{\frac 1p}}\\
&\approx \lim_{a\to 0_+} \frac{a^{\frac 1q}}{\left( \frac{\sqrt{\log \frac 2a}}{\Phi^{-1}(\log \frac 2a)}\right)^{m} a^{\frac 1p} \left(\sqrt{\log \frac 2a}\right)^{m}}\\
&=\frac{\lim_{a\to 0_+} {a^{\frac 1q - \frac 1p}}}{\lim_{a\to 0_+} \left(\frac{\log \frac 2a}{\Phi^{-1}(\log \frac 2a)}\right)^m}=\infty.
\end{align*}
Hence, \eqref{E:pmphi} is not fulfilled.

Suppose that $p=\infty$ and $q<\infty$. Then $L^\infty(0,1) \hookrightarrow L^q(0,1)$, and thus also $(L^\infty)^r_{m,\Phi}(0,1) \hookrightarrow (L^q)^r_{m,\Phi}(0,1)$. It follows from the first part of the proof that $(L^q)^r_{m,\Phi}(0,1) \overset{*}{\hookrightarrow} L^q(0,1)$. Hence, \eqref{E:pmphi} is satisfied. Finally, if $q=\infty$ then it follows from Theorem~\ref{T:l^1_probability} \textup{(b)} that~\eqref{E:comp_emb_prob_leb} is not fulfilled.

Now, assume that $\lim_{s\to \infty} \frac{s}{\Phi(s)}\in (0,\infty)$. Then,
$$
\lim_{s\to \infty} \frac{\Phi^{-1}(s)}{s}
=\lim_{s\to \infty} \frac{\Phi^{-1}(s)}{\Phi(\Phi^{-1}(s))} 
=\lim_{s\to \infty} \frac{s}{\Phi(s)} \in (0,\infty).
$$
Consequently, $\Phi^{-1}(s) \approx s$, $s\in (\log 2,\infty)$. Thus, if we set $I(s)=s$, $s\in (0,1]$, then for every $f\in \mathcal M(0,1)$ we have $P^m_\Phi f \approx H^m_If$, up to multiplicative constants independent of $f\in \mathcal M(0,1)$. Combining this with Theorem~\ref{T:probability}, we obtain that~\eqref{E:comp_emb_prob_leb} holds if and only if
\begin{equation}\label{E:alm_comp}
\lim_{a\to 0_+} \sup_{\|f\|_{L^p(0,1)}\leq 1} \|H^m_I(\chi_{(0,a)}f)\|_{L^q(0,1)}=0.
\end{equation}
By Theorems~\ref{T:mazya} and~\ref{T:lorentz}, both applied with $\alpha=1$, we deduce that~\eqref{E:alm_comp} is fulfilled if and only if $q<p$.
\end{proof}

\begin{proof}[Proof of Theorem~\ref{T:lz_gauss}]
Theorem~\ref{T:probability} applied with $\Phi(t)= \frac{1}{\beta}t^\beta$, $t\in [0,\infty)$, yields that condition~\eqref{E:lz_comp} is satisfied if and only if
\begin{equation}\label{E:gauss}
(L^{p_1,q_1;\alpha_1})^r_{m,\frac{1}{\beta}s^\beta}(0,1) \overset{*}{\hookrightarrow} L^{p_2,q_2;\alpha_2}(0,1).
\end{equation}
Furthermore, it follows from~\cite[Theorem 7.12]{CPS} that 
$$
(L^{p_1,q_1;\alpha_1})^r_{m,\frac{1}{\beta}s^\beta}(0,1)=
\begin{cases}
L^{p_1,q_1;\alpha_1+\frac{m(\beta-1)}{\beta}}(0,1)\, &\hbox{if } p_1<\infty;\\
L^{\infty,q_1;\alpha_1-\frac{m}{\beta}}(0,1)\, &\hbox{if } p_1=\infty.
\end{cases}
$$
By applying Proposition~\ref{T:Lorentz-Zygmund} we get the result.
\end{proof}



\section*{Acknowledgement}

I would like to express my thanks to Lubo\v s Pick for careful reading of this paper and many valuable comments and suggestions.

\end{document}